\documentclass[10pt]{amsart}

\usepackage{amsmath,amsfonts,amsthm,mathrsfs,amssymb,verbatim,cite,cases,fouridx}

\usepackage[usenames]{color}
\usepackage[dvipsnames]{xcolor}
\usepackage{hyperref}
\usepackage{bbm}
\usepackage{mathtools}
\usepackage{tikz,graphicx}
\usetikzlibrary{hobby,math}
\usepackage{upgreek}
\usepackage{accents}

\renewcommand{\sc}{\scriptstyle}
\newcommand{\red}[1]{{\color{red} #1}}
\newcommand{\blue}[1]{{\color{blue} #1}}
\newcommand{\brown}[1]{{\color{Mahogany} #1}}
\newcommand{\green}[1]{{\color{OliveGreen} #1}}

\newcommand{\pink}[1]{{\color{Lavender} #1}}
\newcommand{\orange}[1]{{\color{Orange} #1}}
\newcommand{\purple}[1]{{\color{Purple} #1}}
\newcommand{\yg}[1]{{\color{YellowGreen} #1}}
\newcommand{\A}{\mathscr{A}}
\newcommand{\C}{\mathscr{C}}
\newcommand{\D}{\mathscr{D}}
\newcommand{\G}{\mathscr{G}}
\newcommand{\gfrak}{\mathfrak{g}}
\newcommand{\hfrak}{\mathfrak{h}}
\newcommand{\ip}[2]{\langle#1,#2\rangle}

\allowdisplaybreaks[3]

\newtheorem{theorem}{Theorem}[section]
\newtheorem{lemma}[theorem]{Lemma}
\newtheorem{proposition}[theorem]{Proposition}

\newtheorem{conjecture}[theorem]{Conjecture}

\theoremstyle{remark}

\theoremstyle{definition}

\numberwithin{equation}{section}

\renewcommand{\leq}{\leqslant}
\renewcommand{\geq}{\geqslant}

\DeclareMathOperator{\HL}{HL}
\DeclareMathOperator{\mult}{mult}
\DeclareMathOperator{\lev}{lev}
\DeclareMathOperator{\eup}{e}
\newcommand{\Exp}[1]{\eup(#1)}
\DeclareMathOperator{\ch}{ch}
\DeclareMathOperator{\sgn}{sgn}
\DeclareMathOperator{\Mod}{mod}
\newcommand{\ceil}[1]{\lceil#1\rceil}
\newcommand{\floor}[1]{\lfloor#1\rfloor}
\newcommand{\Floor}[1]{\bigg\lfloor#1\bigg\rfloor}
\newcommand{\qbin}[2]{\genfrac{[}{]}{0pt}{}{#1}{#2}}
\newcommand{\qbinbig}[2]{\Bigg[\genfrac{}{}{0pt}{}{#1}{#2}\Bigg]}
\newcommand{\binombig}[2]{\Big(\!\genfrac{}{}{0pt}{}{#1}{#2}\!\Big)}
\newcommand{\abs}[1]{\vert#1\vert}
\newcommand{\la}{\lambda}
\newcommand{\La}{\Lambda}
\newcommand{\Lap}{\Uplambda}

\begin{document}

\title[The conjectures of Capparelli, Meurman, Primc and Primc]
{Remarks on the conjectures of Capparelli, Meurman, Primc and Primc}

\author[Kanade, Russell, Tsuchioka and Warnaar]
{Shashank Kanade, Matthew C.~Russell, \\ Shunsuke Tsuchioka, S.~Ole Warnaar}
\thanks{This work is supported by the Simons Foundation Travel Support for 
Mathematicians grant \#636937 (Kanade), the JSPS Kakenhi Grant 
23K03051 and JST CREST Grant JPMJCR2113 (Tsuchioka) and the
Australian Research Council Discovery Project DP200102316 (Warnaar).}

\dedicatory{Dedicated to George Andrews and Bruce Berndt in celebration of
their 85th birthdays}

\subjclass[2020]{05A17, 05E05, 05E10, 11P84, 17B67}

\begin{abstract}
In a series of two papers, S.~Capparelli, A.~Meurman, A.~Primc,
M.~Primc (CMPP) and then M.~Primc put forth three remarkable sets of
conjectures, stating that the generating functions of coloured integer
partition in which the parts satisfy restrictions on the multiplicities
admit simple infinite product forms.
While CMPP related one set of conjectures to the principally specialised 
characters of standard modules for the affine Lie algebra
$\mathrm{C}_n^{(1)}$, finding a Lie-algebraic interpretation for the
remaining two sets remained an open problem.
In this paper, we use the work of Griffin, Ono and the fourth
author on Rogers--Ramanujan identities for affine Lie algebras to
solve this problem, relating the remaining two sets of conjectures 
to non-standard specialisations of standard modules for 
$\mathrm{A}_{2n}^{(2)}$ and $\mathrm{D}_{n+1}^{(2)}$.
We also use their work to formulate conjectures for the bivariate
generating function of one-parameter families of CMPP partitions
in terms of Hall--Littlewood symmetric functions.
We make a detailed study of several further aspects of CMPP
partitions, obtaining (i) functional equations
for bivariate generating functions which generalise the 
well-known Rogers--Selberg equations, (ii) a partial level-rank 
duality in the $\mathrm{A}_{2n}^{(2)}$ case, and (iii)
(conjectural) identities of the Rogers--Ramanujan type for
$\mathrm{D}_3^{(2)}$. \\
\textbf{Keywords.} Affine Lie algebras, CMPP conjectures,
Hall--Littlewood symmetric functions, Rogers--Ramanujan identities.
\end{abstract}

\maketitle


\section{Introduction}

Gordon's partition theorem \cite{Gordon61} is one of the deepest and
most beautiful results in the theory of integer partitions,
generalising the combinatorial version of the Rogers--Ramanujan
identities to arbitrary odd moduli.
To describe Gordon's theorem, some basic partition-theoretic notions
are needed.
Let $\la$ be an integer partition, that is, $\la=(\la_1,\la_2,\dots)$ is
a weakly decreasing sequence of nonnegative integers such that only finitely
many $\la_i$ are strictly positive \cite{Andrews76}.
Such positive $\la_i$ are known as the parts of $\la$, and it is standard
convention to suppress the trailing sequence of zeros in a partition,
so that $(5,3,3,2,2,2,1,0,\dots)$ is written as $(5,3,3,2,2,2,1)$.
If $\abs{\la}:=\la_1+\la_2+\cdots=N$ then $\la$ is said to be a partition
of $N$, written as $\la\vdash N$.
The number of parts of $\la$ equal to $i$ is known as the multiplicity
or frequency of the part $i$, and is denoted by $f_i=f_i(\la)$. 
In the previous example, $f_1=f_5=1$, $f_2=3$, $f_3=2$ and $f_i=0$
for $i=4$ and $i\geq 6$.
Clearly, if $\la\vdash N$, then $\sum_{i\geq 1}if_i=N$.

For integers $a,k,N$ such that $0\leq a\leq k$ and $N\geq 0$,
let $A_{k,a}(N)$ be the set of partitions of $N$ into parts not
congruent to $0,\pm (a+1)$ modulo $2k+3$ and $B_{k,a}(N)$ the set of
partitions of $N$ such that
\begin{equation}\label{Eq_ff}
f_i+f_{i+1}\leq k \text{ and } f_1\leq a.
\end{equation}
(Gordon's original description of the set $B_{k,a}(N)$ is slightly different.
He defined it as the set of partitions $\la=(\la_1,\la_2,\dots)\vdash N$
such that $\la_i-\la_{i+k}\geq 2$ for all $i\geq 1$ and $f_1\leq a$.
It is not difficult to see that $\la_i-\la_{i+k}\geq 2$ for all $i$ 
if and only if $f_i+f_{i+1}\leq k$ for all $i$.)
Gordon's theorem states that $A_{k,a}(N)$ and $B_{k,a}(N)$ have equal
cardinality.

\begin{theorem}[Gordon \cite{Gordon61}] \label{Thm_Gordon}
For integers $a,k,N$ such that $0\leq a\leq k$ and $N\geq 0$,
\[
\abs{A_{k,a}(N)}=\abs{B_{k,a}(N)}.
\]
\end{theorem}

For $k=1$ these are the famous Rogers--Ramanujan identities in their
combinatorial incarnation, as first stated by MacMahon \cite{MacMahon04}
and Schur~\cite{Schur17}.
Let
\[
B_{k,a}:=\bigcup_{N\geq 0} B_{k,a}(N)
\]
be the set of all Gordon partitions.
Since the generating function of partitions in $A_{k,a}(N)$ admits
a simple product form, an equivalent form of Gordon's theorem is
\begin{align}\label{Eq_Gordon}
\sum_{\la\in B_{k,a}} q^{\abs{\la}} &= 
\sum_{N=0}^{\infty} \abs{B_{k,a}(N)} q^N \\
&=\prod_{\substack{m=0\\[1pt]m\not\equiv 0,\pm(a+1)\,(\Mod 2k+3)}}^{\infty} 
\frac{1}{1-q^m}=
\frac{(q^{a+1},q^{2k-a+2},q^{2k+3};q^{2k+3})_{\infty}}{(q;q)_{\infty}},
\notag
\end{align}
where 
$(a_1,\dots,a_r;q)_{\infty}:=\prod_{i=1}^r \prod_{j\geq 0} (1-a_i q^j)$.
For a more detailed introduction to Gordon's partition theorem the reader
is referred to \cite[Chapter 7]{Andrews76} and \cite[Chapter 3]{Sills18}.

In a remarkable paper \cite{CMPP22}, Capparelli, Meurman, Primc and
Primc (CMPP) recently conjectured a beautiful generalisation of 
\eqref{Eq_Gordon} for coloured partitions of $N$.
In their partitions, each part is assigned one of $n$ possible colours, 
where the ordering among colours is immaterial.
For example, if $n=3$ with colour-set
$\{\text{\red{red}, \blue{blue}, black}\}$, 
the partitions
$(\red{5},\blue{3},3,\red{2},\red{2},2,\blue{1})$
and
$(\blue{5},\blue{3},3,\red{2},\red{2},2,1)$
are considered distinct partitions of $18$, but
$(\red{5},\blue{3},3,\red{2},\red{2},2,\blue{1})$ and
$(\red{5},3,\blue{3},\red{2},2,\red{2},\blue{1})$
represent one and the same partition.
(Alternatively, one can choose to order the colour-set
as in $\{\text{\red{red}}>\text{\blue{blue}}>\text{black}\}$
and require that parts of the same size and different colour are
ordered accordingly. This would make 
$(\red{5},3,\blue{3},\red{2},2,\red{2},\blue{1})$ inadmissible.)
Given a coloured partition $\la\vdash N$, the multiplicity or
frequency of parts of size $i$ and colour $c$ is denoted by
$f_i^{(c)}=f_i^{(c)}(\la)$, so that
$\sum_{i\geq 1}\sum_{c=1}^n if_i^{(c)}=N$.
The generalisation of \eqref{Eq_ff} to CMPP partitions is
most conveniently expressed in terms of lattice paths. 
Let $\mathcal{P}^{(n)}$ denote the set of $n$-coloured partitions.
Given $\la\in\mathcal{P}^{(n)}$ and nonnegative integers
$k_0,k_1,\dots,k_n\in\mathbb{N}_0$ which play the role
of initial or boundary conditions, arrange the frequencies of $\la$
in a semi-infinite array $\mathscr{F}_n(\la)$ of $2n$ rows as follows:
\smallskip

\tikzmath{\x=1.73205;}
\begin{center}
\begin{tikzpicture}[scale=0.5,line width=0.3pt]
\draw (0,7) node {$k_n$};
\draw (0,4.75) node {$\vdots$};
\draw (0,2) node {$k_2$};
\draw (0,0) node {$k_1$};
\draw (\x,-1) node {$k_0$};
\draw (\x,1) node {$0$};
\draw (\x,3.75) node {$\vdots$};
\draw (\x,6) node {$0$};
\draw (2*\x,0) node {\red{$f_1^{(1)}$}};
\draw (2*\x,2) node {\blue{$f_1^{(2)}$}};
\draw (2*\x,4.75) node {$\vdots$};
\draw (2*\x,7) node {\green{$f_1^{(n)}$}};
\draw (3*\x,-1) node {\red{$f_2^{(1)}$}};
\draw (3*\x,1) node {\blue{$f_2^{(2)}$}};
\draw (3*\x,3.75) node {$\vdots$};
\draw (3*\x,6) node {\green{$f_2^{(n)}$}};
\draw (4*\x,0) node {\red{$f_3^{(1)}$}};
\draw (4*\x,2) node {\blue{$f_3^{(2)}$}};
\draw (4*\x,4.75) node {$\vdots$};
\draw (4*\x,7) node {\green{$f_3^{(n)}$}};
\draw (5*\x,-1) node {\red{$f_4^{(1)}$}};
\draw (5*\x,1) node {\blue{$f_4^{(2)}$}};
\draw (5*\x,3.75) node {$\vdots$};
\draw (5*\x,6) node {\green{$f_4^{(n)}$}};
\draw (6*\x,0) node {\red{$f_5^{(1)}$}};
\draw (6*\x,2) node {\blue{$f_5^{(2)}$}};
\draw (6*\x,4.75) node {$\vdots$};
\draw (6*\x,7) node {\green{$f_5^{(n)}$}};
\draw (7*\x,-1) node {\red{$f_6^{(1)}$}};
\draw (7*\x,1) node {\blue{$f_6^{(2)}$}};
\draw (7*\x,3.75) node {$\vdots$};
\draw (7*\x,6) node {\green{$f_6^{(n)}$}};
\draw (8*\x,0) node {\red{$f_7^{(1)}$}};
\draw (8*\x,2) node {\blue{$f_7^{(2)}$}};
\draw (8*\x,4.75) node {$\vdots$};
\draw (8*\x,7) node {\green{$f_7^{(n)}$}};
\draw (9*\x,-1) node {\red{$f_8^{(1)}$}};
\draw (9*\x,1) node {\blue{$f_8^{(2)}$}};
\draw (9*\x,3.5) node {$\vdots$};
\draw (9*\x,6) node {\green{$f_8^{(n)}$}};
\draw (10*\x,-0.5) node {\red{$\dots$}};
\draw (10*\x,1.5) node {\blue{$\dots$}};
\draw (10*\x,6.5) node {\green{$\dots$}};
\end{tikzpicture}
\end{center}

\smallskip\noindent
For $c\in\{1,\dots,n\}$, let $f_0^{(c)}=k_0\delta_{c,1}$ (with $\delta_{i,j}$
the Kronecker delta) and $f_{-1}^{(c)}=k_c$.
Then a path $P$ on $\mathscr{F}_n(\la)$ is a sequence
\[
(p_1,p_2,\dots,p_{2n})
\]
such that
\[
p_{2c-1}\in\big\{f_0^{(c)},f_2^{(c)},f_4^{(c)},\dots\big\}, \quad
p_{2c}\in\big\{f_{-1}^{(c)},f_1^{(c)},f_3^{(c)},\dots\big\},
\]
and, dropping the colour labels, such that $f_i$ is followed by either 
$f_{i-1}$ (which requires that $i\geq 0$) or $f_{i+1}$.
Two examples of paths on $\mathscr{F}_3(\la)$ are shown below, 
where the (superfluous) colour labels have been omitted:

\smallskip

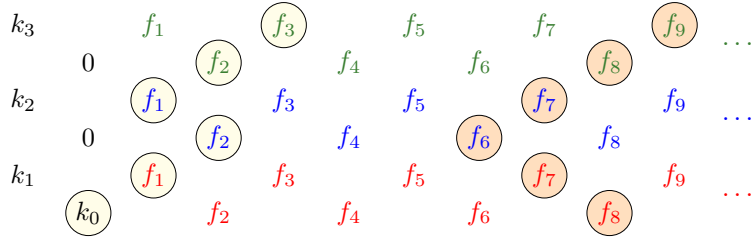
\begin{center}
\begin{tikzpicture}[scale=0.5,line width=0.3pt]
\foreach \n in {1,2,3} {
\filldraw[black,fill=yellow!10!white] (\n*\x,\n-2) circle (6mm);
\filldraw[black,fill=yellow!10!white] (\n*\x+\x,\n+1) circle (6mm);
\filldraw[black,fill=orange!25!white] (10*\x-\n*\x,\n-2) circle (6mm);
\filldraw[black,fill=orange!25!white] (\n*\x+7*\x,\n+1) circle (6mm);};
\draw (0,4) node {$k_3$};
\draw (0,2) node {$k_2$};
\draw (0,0) node {$k_1$};
\draw (\x,-1) node {$k_0$};
\draw (\x,1) node {$0$};
\draw (\x,3) node {$0$};
\draw (2*\x,0) node {\red{$f_1$}};
\draw (2*\x,2) node {\blue{$f_1$}};
\draw (2*\x,4) node {\green{$f_1$}};
\draw (3*\x,-1) node {\red{$f_2$}};
\draw (3*\x,1) node {\blue{$f_2$}};
\draw (3*\x,3) node {\green{$f_2$}};
\draw (4*\x,0) node {\red{$f_3$}};
\draw (4*\x,2) node {\blue{$f_3$}};
\draw (4*\x,4) node {\green{$f_3$}};
\draw (5*\x,-1) node {\red{$f_4$}};
\draw (5*\x,1) node {\blue{$f_4$}};
\draw (5*\x,3) node {\green{$f_4$}};
\draw (6*\x,0) node {\red{$f_5$}};
\draw (6*\x,2) node {\blue{$f_5$}};
\draw (6*\x,4) node {\green{$f_5$}};
\draw (7*\x,-1) node {\red{$f_6$}};
\draw (7*\x,1) node {\blue{$f_6$}};
\draw (7*\x,3) node {\green{$f_6$}};
\draw (8*\x,0) node {\red{$f_7$}};
\draw (8*\x,2) node {\blue{$f_7$}};
\draw (8*\x,4) node {\green{$f_7$}};
\draw (9*\x,-1) node {\red{$f_8$}};
\draw (9*\x,1) node {\blue{$f_8$}};
\draw (9*\x,3) node {\green{$f_8$}};
\draw (10*\x,0) node {\red{$f_9$}};
\draw (10*\x,2) node {\blue{$f_9$}};
\draw (10*\x,4) node {\green{$f_9$}};
\draw (11*\x,-0.5) node {\red{$\dots$}};
\draw (11*\x,1.5) node {\blue{$\dots$}};
\draw (11*\x,3.5) node {\green{$\dots$}};
\end{tikzpicture}
\end{center}

\smallskip

\noindent
By slight abuse of notation, we write $P\in\mathscr{F}_n(\la)$ if $P$
is a path on $\mathscr{F}_n(\la)$.
A partition $\la\in\mathcal{P}^{(n)}$ is a CMPP-partition if for all
$P=(p_1,\dots,p_{2n})\in\mathscr{F}_n(\la)$,
\[
\sum_{i=1}^{2n} p_i\leq k_0+\dots+k_n.
\]
The set of all CMPP-partitions for given fixed $k_0,\dots,k_n$ will
be denoted by $\mathscr{A}_{k_0,\dots,k_n}$.\footnote{The choice
for the letter $\mathscr{A}$ reflects the fact that these partitions are
related to $\mathrm{A}^{(2)}_{2n}$. 
Later we will describe variants of these partitions for the affine
Lie algebras $\mathrm{C}^{(1)}_n$ and $\mathrm{D}^{(2)}_{n+1}$ which
will be denoted by $\mathscr{C}_{k_0,\dots,k_n}$ and
$\mathscr{D}_{k_0,\dots,k_n}$ respectively.}
If $n=1$ then the set of paths on $\mathscr{F}_1(\la)$ is given by 
\[
\{(k_0,k_1),(k_0,f_1)\}\cup\{(f_{2i},f_{2i-1}): i\geq 1\}
\cup\{(f_{2i},f_{2i+1}): i\geq 1\}.
\]
Hence
\begin{align}\label{Eq_B-B}
\A_{k_0,k_1}&=\big\{\la\in\mathcal{P}^{(1)}: f_1(\la)\leq k_1
\text{ and } f_i(\la)+f_{i+1}(\la)\leq k_0+k_1
\text{ for all $i\geq 1$} \big\} \\[1mm]
&=B_{k_0+k_1,k_1}. \notag
\end{align}
CMPP conjectured the following generalisation of Gordon's partition
theorem in the form \eqref{Eq_Gordon}.
For $a,a_1,\dots,a_r\neq 0$, 
let $\theta(a,q):=(a,q/a;q)_{\infty}$ be a modified
theta function and $\theta(a_1,\dots,a_r;q):=\prod_{i=1}^r \theta(a_i;q)$.

\begin{conjecture}[CMPP {\cite[Conjecture 4.1]{CMPP22}}]\label{Con_A2n2}
For $k,n$ nonnegative integers and $k_0,\dots,k_n$
nonnegative integers such that $k_0+\dots+k_n=k$,
\begin{align*}
\sum_{\la\in\A_{k_0,\dots,k_n}} q^{\abs{\la}}
&=\frac{(q^{2k+2n+1};q^{2k+2n+1})_{\infty}^n}{(q;q)_{\infty}^n}
\prod_{i=1}^n \theta\big(q^{\la_i+n-i+1};q^{2k+2n+1}\big) \\
&\quad\times\prod_{1\leq i<j\leq n}
\theta\big(q^{\la_i-\la_j-i+j},q^{\la_i+\la_j+2n-i-j+2};q^{2k+2n+1}\big),
\end{align*}
where
$\la_i:=k_i+\dots+k_n$ for $1\leq i\leq n$.
\end{conjecture}

It should be remarked that CMPP do not state the product on the right
in the explicit form as shown above, but instead provide a prescription
for obtaining the product for fixed $k_0,\dots,k_n$ using what they 
refer to as `congruence triangles'.
The rationale for including the trivial $n=0$ case is that from a
representation-theoretic perspective it is natural to include $k=0$
(for which the product trivialises to $1$) in the conjecture.
Since Conjecture~\ref{Con_A2n2} exhibits a partial `level-rank duality'
which interchanges the roles of $k$ and $n$, it is natural to let
$k$ and $n$ have the same range.

CMPP write about their conjecture that ``for $n>1$ there is no obvious
connection [\dots] with representation theory of affine Lie algebras.''
The first purpose of this paper is to point out that there is a
connection between Conjecture~\ref{Con_A2n2} and the representation
theory of affine Lie algebras, very different from the well-known
interpretation of the $n=1$ case in terms of standard modules for
$\mathrm{A}^{(1)}_1$ at level $2k+1$, see e.g.,~\cite{LM78,Lepowsky82,LW85}.
Adopting standard notation and terminology (see Section~\ref{Sec_Lie}
for details), let $L(\La)$ be the $\mathrm{A}^{(2)}_{2n}$-standard
module of highest weight $\La$.
Assuming the same relation between the $k_i$ and $\la_i$ as in
Conjecture~\ref{Con_A2n2}, parametrise $\La$ in terms of the
fundamental weights $\La_0,\dots,\La_n$ of $\mathrm{A}^{(2)}_{2n}$ as 
\begin{align}\label{Eq_La-para}
\La&=2(k-\la_1)\La_0+(\la_1-\la_2)\La_1+\dots+
(\la_{n-1}-\la_n)\La_{n-1}+\la_n\La_n \\[1mm]
&=2k_0\La_0+k_1\La_1+\dots+k_{n-1}\La_{n-1}+k_n\La_n. \notag
\end{align}
Here the usual labelling of the vertices of the $\mathrm{A}^{(2)}_{2n}$
Dynkin diagram is assumed:
\label{page_A2n}
\begin{center}
\begin{tikzpicture}[scale=0.6]
\draw (-1.2,0) node {$\mathrm{A}^{(2)}_2$:};
\draw (0,0.1)--(1,0.1);
\draw (0, 0.033)--(1, 0.033);
\draw (0,-0.033)--(1,-0.033);
\draw (0,-0.1)--(1,-0.1);
\draw (0.6,0.2)--(0.4,0)--(0.6,-0.2);
\draw[fill=red] (0,0) circle (0.08cm);
\draw[fill=blue] (1,0) circle (0.08cm);
\draw (0,-0.45) node {$\scriptstyle 0$};
\draw (1,-0.45) node {$\scriptstyle 1$};
\draw (7.8,0) node {$\mathrm{A}^{(2)}_{2n}$:};
\draw (9, 0.07)--(10, 0.07);
\draw (9,-0.07)--(10,-0.07);
\draw (9.6,0.2)--(9.4,0)--(9.6,-0.2);
\draw (10,0)--(14,0);
\draw (14.6,0.2)--(14.4,0)--(14.6,-0.2);
\draw (14, 0.07)--(15,0.07);
\draw (14,-0.07)--(15,-0.07);
\draw[fill=red] (9,0) circle (0.08cm);
\foreach \x in {10,...,15} \draw[fill=blue] (\x,0) circle (0.08cm);
\draw (9,-0.45) node {$\scriptstyle 0$};
\draw (10,-0.45) node {$\scriptstyle 1$};
\draw (15,-0.45) node {$\scriptstyle n$};
\end{tikzpicture}
\end{center}

\noindent
where, for now, the colour coding of the vertices may be ignored. 
Since $\lev(\La_0)=1$ and $\lev(\La_i)=2$ for all $i\geq 1$, the level
of $\La$ is even and given by $\lev(\La)=2k$.
Denote the character of $L(\La)$ by $\ch L(\La)$, and define
the normalised character $\chi_{\La}:=\Exp{-\La}\ch L(\La)$,
where $\Exp{\cdot}$ is a formal exponential.
Then
\[
\chi_{\La}\in\mathbb{Z}[[\Exp{-\alpha_0},\dots,\Exp{-\alpha_n}]],
\]
where $\alpha_0,\dots,\alpha_n$ are the simple roots of 
$\mathrm{A}^{(2)}_{2n}$.
If $\varphi_n:\mathbb{Z}[[\Exp{-\alpha_0},\dots,\Exp{-\alpha_n}]]\to
\mathbb{Z}[[q]]$ is the specialisation 
\[
\varphi_n(\Exp{-\alpha_0})=-1\quad\text{and}\quad
\varphi_n(\Exp{-\alpha_i})=q\;\text{ for $i\in\{1,\dots,n\}$},
\]
then it follows from \cite[Equation (3.31)]{GOW16} (see the 
appendix for details) that
\begin{align}\label{Eq_A2n-product}
\varphi_n(\chi_\La)&=
\frac{(q^{2k+2n+1};q^{2k+2n+1})_{\infty}^n}{(q;q)_{\infty}^n}
\prod_{i=1}^n \theta\big(q^{\la_i+n-i+1};q^{2k+2n+1}\big) \\
&\quad\times \prod_{1\leq i<j\leq n}
\theta\big(q^{\la_i-\la_j-i+j},q^{\la_i+\la_j+2n-i-j+2};q^{2k+2n+1}\big).
\notag
\end{align}
(If the coefficient of $\La_0$ in $\La$ is odd, so that $\La$
has odd level, then $\varphi_n(\chi_\La)=0$.)
Importantly, the specialisation $\varphi_n$ does not correspond to the
much-studied principal gradation of $L(\La)$ and hence the product
\eqref{Eq_A2n-product} does not follow from Lepowsky's numerator
formula~\cite{Lepowsky82}.
By \eqref{Eq_A2n-product} the CMPP conjecture can be expressed
in terms of $\mathrm{A}^{(2)}_{2n}$ as follows.

\begin{conjecture}[Representation theoretic form of
Conjecture~\ref{Con_A2n2}]\label{Con_A2n2-rep}
For nonnegative integers $k_0,\dots,k_n$, let
$L(\La)$ be the $\mathrm{A}^{(2)}_{2n}$-standard module
of highest weight $\La=2k_0\La_0+k_1\La_1+\dots+k_n\La_n$.
Then
\begin{equation}\label{Eq_CMPP2}
\sum_{\la\in\A_{k_0,\dots,k_n}}q^{\abs{\la}}=\varphi_n(\chi_{\La}).
\end{equation}
\end{conjecture}

As remarked above, the set $A_{k,a}(N)$ has a simple generating
function given by the infinite product on the right of \eqref{Eq_Gordon}.
In \cite{Andrews74}, Andrews showed that the two-variable generating
function of Gordon partitions is expressible as a $k$-fold multisum:
\begin{equation}\label{Eq_AG}
\sum_{\la\in B_{k,a}} z^{l(\la)} q^{\abs{\la}}
=\sum_{r_1,\dots,r_k\geq 0}
\frac{z^{r_1+\dots+r_k} q^{r_1^2+\dots+r_k^2+r_{a+1}+\dots+r_k}}
{(q;q)_{r_1-r_2}\cdots(q;q)_{r_{k-1}-r_k}(q;q)_{r_k}},
\end{equation}
where $(a;q)_n:=(a;q)_{\infty}/(aq^n;q)_{\infty}=
\prod_{i=0}^{n-1}(1-aq^i)$.
Note in particular that $1/(q;q)_n$ vanishes for $n$ a nonnegative
integer so that the summand is nonzero for
$r_1\geq r_2\geq\cdots\geq r_k$ only.
Equating the $z=1$ case of \eqref{Eq_AG} with the right-hand side of
\eqref{Eq_Gordon} results in what are known as the Andrews--Gordon
identities \cite{Andrews74}:
\begin{equation}\label{Eq_AG2}
\sum_{r_1,\dots, r_k\geq 0}
\frac{q^{r_1^2+\dots+r_k^2+r_{a+1}+\dots+r_k}}
{(q;q)_{r_1-r_2}\cdots(q;q)_{r_{k-1}-r_k}(q;q)_{r_k}}
=\frac{(q^{a+1},q^{2k-a+2},q^{2k+3};q^{2k+3})_{\infty}}{(q;q)_{\infty}}.
\end{equation}
The two $k=1$ cases are the Rogers--Ramanujan identities in
their original analytic form~\cite{Rogers94,Rogers17,RR19}.
 
The second purpose of this paper is to give a conjectural and highly
incomplete generalisation of \eqref{Eq_AG} to CMPP partitions or, as
we may now call them, $\mathrm{A}_{2n}^{(2)}$-partitions.
Let $P_{\la}(x_1,x_2,\dots;t)$ be the Hall--Littlewood symmetric
function in countably-many variables, indexed by the partition $\la$,
see Section~\ref{Sec_HL} for details.
Furthermore, for $\la=(\la_1,\la_2,\dots)$ a partition, let $2\la$
denote the partition $(2\la_1,2\la_2,\dots)$.
Finally, if
\[
k_0,\dots,k_n=
\underbrace{i_1,\dots,i_1}_{m_1 \text{ times}},
\underbrace{i_2,\dots,i_2}_{m_2 \text{ times}},\dots\dots,
\underbrace{i_r,\dots,i_r}_{m_r \text{ times}},
\]
where $m_1+\dots+m_r=n+1$, we more succinctly write this as
$k_0,\dots,k_n=i_1^{m_1},i_2^{m_2},\dots,i_r^{m_r}$, typically
omitting those exponents $m_i$ that are equal to $1$.
For example, if $k_0,\dots,k_8=1,0,0,0,2,1,3,3,0$ this would be
shortened to $k_0,\dots,k_8=1,0^3,2,1,3^2,0$.

\begin{conjecture}\label{Con_A2n2-qseries}
For $k$ a nonnegative integer and $n$ a positive integer,
\begin{subequations}
\begin{align}\label{Eq_kLambdan}
\sum_{\la\in\A_{0^n,k}} z^{l(\la)} q^{\abs{\la}}
&= \sum_{\substack{\la \\[1pt] \la_1\leq k}}
(zq)^{\abs{\la}} P_{2\la}\big(1,q,q^2,\dots;q^{2n-1}\big) 
\intertext{and}
\label{Eq_2kLambda0}
\sum_{\la\in\A_{k,0^n}} z^{l(\la)} q^{\abs{\la}}
&=\sum_{\substack{\la \\[1pt] \la_1\leq k}}
(zq^2)^{\abs{\la}} P_{2\la}\big(1,q,q^2,\dots;q^{2n-1}\big).
\end{align}
\end{subequations}
\end{conjecture}

Given a partition, let $\la'$ be its conjugate, i.e.,
$\la'_i=f_i(\la)-f_{i+1}(\la)$.
Making the identification $\la'_i=r_i$, we
have~\cite[page 213]{Macdonald95}
\begin{equation}\label{Eq_PS-2lambda}
(zq)^{\abs{\la}} P_{2\la}(1,q,q^2,\dots;q)
=\prod_{i\geq 1} \frac{z^{r_i}q^{r_i^2}}{(q;q)_{r_i-r_{i+1}}}.
\end{equation}
Recalling \eqref{Eq_B-B} and using that $\la_1\leq k$ implies that 
$r_i=\la'_i=0$ for $i>k$, it follows that \eqref{Eq_kLambdan} and
\eqref{Eq_2kLambda0} for $n=1$ are the $a=k$ and $a=0$ instances
of \eqref{Eq_AG} respectively.
By \cite[Theorem 1.1]{GOW16}, \eqref{Eq_kLambdan} and
\eqref{Eq_2kLambda0} for $z=1$ are equivalent to
Conjecture~\ref{Con_A2n2-rep} for $k_0,\dots,k_n$ given by $0^n,k$
and $k,0^n$ respectively.

\medskip

The remainder of this paper is organised as follows.
In the next section we introduce terminology and notation
pertaining to affine Lie algebras and Hall--Littlewood symmetric functions
needed for this paper.
This section includes a proof, using the Bailey lemma, of an identity
of \cite{GOW16} for $P_{(2^r)}(1,q,q^2,\dots;q^N)$ for arbitrary
positive integers $N$ (see Proposition~\ref{Prop_GOW}) and a
conjectural Andrews--Gordon-type sum for 
$P_{2\la}(1,q,q^2,\dots;q^2)$ (see Conjecture~\ref{Con_Shun}).
In Section~\ref{Sec_Special} we establish a partial level-rank
duality for the CMPP conjecture and discuss some known or provable
cases of Conjectures~\ref{Con_A2n2}--\ref{Con_A2n2-qseries}.
We also derive a set of functional equations for the two-variable
generating function of CMPP partitions (see Proposition~\ref{Prop_fun}),
generalising the well-known Rogers--Selberg equations.
There are two important variants of the CMPP conjecture
for coloured partitions whose frequency arrays have an odd number
of rows.
One of these, related to the principal specialisation of
characters of $\mathrm{C}^{(1)}_n$-standard modules, is also due
to CMPP, while the other, due to Primc \cite{Primc24}, is related to a
non-standard specialisation of characters of standard modules for
$\mathrm{D}^{(2)}_{n+1}$.
Both these conjectures, which turn out to be closely related,
are the topic of Section~\ref{Sec_CD}.
For almost all of our results and conjectures for $\mathrm{A}^{(2)}_{2n}$,
with the exception of level-rank duality, we formulate analogues for
$\mathrm{C}^{(1)}_n$ and $\mathrm{D}^{(2)}_{n+1}$.
Using the above-mentioned Conjecture~\ref{Con_Shun} for Hall--Littlewood
symmetric functions, this also leads to a number of new conjectures
of Rogers--Ramanujan type, including
\[
\sum_{\substack{r_1,\dots,r_k\geq 0 \\[1pt] s_1,\dots,s_k\geq 0}}\,
\prod_{i=1}^k \frac{q^{(r_i+s_i)^2+s_i^2+s_i}}
{(q;q)_{r_i-r_{i+1}}(q^2;q^2)_{s_i-s_{i-1}}}
=\frac{(q,q^{k+1},q^{k+2};q^{k+2})_{\infty}}
{(q;q^2)_{\infty}(q;q)_{\infty}},
\]
which is a $q$-series identity for the principally specialised 
character of the $\mathrm{A}^{(1)}_1$-standard module $L(k\La_0)$
very different from those in
\cite{Andrews10,CSXY22,Feigin09,KY13,Stembridge90,Tsuchioka23},
and the very similar
\begin{align*}
\sum_{\substack{r_1,\dots,r_k\geq 0 \\[1pt] s_1,\dots,s_k\geq 0}} & \,
\prod_{i=1}^k \frac{q^{(r_i+s_i)^2+s_i^2}}
{(q;q)_{r_i-r_{i+1}}(q^2;q^2)_{s_i-s_{i-1}}} \\
&=\frac{(q^{k+1},q^{k+2},q^{k+2},q^{k+3},q^{2k+4},q^{2k+4};q^{2k+4})_{\infty}}
{(q^2;q^2)_{\infty}(q;q)_{\infty}},
\end{align*}
which is a $q$-series identity related to a non-standard specialisation
of the $\mathrm{D}^{(2)}_3$-standard module $L(k\La_1)$.
The close resemblance of these formulas to the Andrews--Gordon
identities \eqref{Eq_AG} seems quite remarkable.
In Section~\ref{Sec_completion} we state some preliminary results,
mostly conjectural, towards the seemingly very hard problem of completing
Conjecture~\ref{Con_A2n2-qseries} (and its analogues for
$\mathrm{C}^{(1)}_n$ and $\mathrm{D}^{(2)}_{n+1}$) to arbitrary
$k_0,\dots,k_n$.
Finally, in the appendix we provide the details of the proof of the
non-standard specialisation formula \eqref{Eq_A2n-product} and its
$\mathrm{D}^{(2)}_{n+1}$-analogue.


\section{Preliminaries}

\subsection{Affine Lie algebras}\label{Sec_Lie}

In the following we fix the index set $I:=\{0,1,\dots,n\}$ for
$n$ a positive integer, and let $\gfrak=\gfrak(A)$ be an affine Lie 
algebra with generalised Cartan matrix $A=(a_{ij})_{i,j\in I}$
and Cartan subalgebra $\hfrak$, see e.g., \cite{Kac90}.
We choose bases $\{\alpha_0,\dots,\alpha_n,\La_0\}$ and
$\{\alpha_0^{\vee},\dots,\alpha_n^{\vee},d\}$ of $\hfrak^{\ast}$ and
$\hfrak$ respectively, such that 
\[
\ip{\alpha_i^{\vee}}{\alpha_j}=a_{ij},\quad
\ip{\alpha_i^{\vee}}{\La_0}=\ip{d}{\alpha_i}=\delta_{i,0},\quad
\ip{d}{\La_0}=0,
\]
where $\ip{\cdot}{\cdot}$ is the natural pairing between
$\hfrak$ and $\hfrak^{\ast}$.
In the following it will be assumed that $\gfrak$ is one of
the affine Lie algebras $\mathrm{A}^{(2)}_{2n}$, $\mathrm{C}^{(1)}_n$, 
$\mathrm{D}^{(2)}_{n+1}$, and that the labelling of the  
simple roots $\alpha_i$ is in accordance with the labelling of
the corresponding Dynkin diagrams as shown on page~\pageref{page_A2n}
or in the diagrams below: \label{page_CD}
\tikzmath{\y=10;}
\begin{center}
\begin{tikzpicture}[scale=0.6]
\draw (-1.2,0) node {$\mathrm{C}^{(1)}_n$:};
\draw (0, 0.07)--(1,0.07);
\draw (0,-0.07)--(1,-0.07);
\draw (0.4,0.2)--(0.6,0)--(0.4,-0.2);
\draw (1,0)--(5,0);
\draw (5.6,0.2)--(5.4,0)--(5.6,-0.2);
\draw (5, 0.07)--(6, 0.07);
\draw (5,-0.07)--(6,-0.07);
\foreach \x in {0,...,6} \draw[fill=blue] (\x,0) circle (0.08cm);
\draw (0,-0.45) node {$\scriptstyle 0$};
\draw (1,-0.45) node {$\scriptstyle 1$};
\draw (6,-0.45) node {$\scriptstyle n$};
\draw (\y-1.4,0) node {$\mathrm{D}^{(2)}_{n+1}$:};
\draw (\y, 0.07)--(\y+1, 0.07);
\draw (\y,-0.07)--(\y+1,-0.07);
\draw (\y+0.6,0.2)--(\y.4,0)--(\y+0.6,-0.2);
\draw (\y+1,0)--(\y+5,0);
\draw (\y+5.4,0.2)--(\y+5.6,0)--(\y+5.4,-0.2);
\draw (\y+5, 0.07)--(\y+6, 0.07);
\draw (\y+5,-0.07)--(\y+6,-0.07);
\draw[fill=red] (\y,0) circle (0.08cm);
\draw[fill=red] (\y+6,0) circle (0.08cm);
\foreach \z in {1,...,5} \draw[fill=blue] (\z+\y,0) circle (0.08cm);
\draw (\y,-0.45) node {$\scriptstyle 0$};
\draw (\y+1,-0.45) node {$\scriptstyle 1$};
\draw (\y+6,-0.45) node {$\scriptstyle n$};
\end{tikzpicture}
\end{center}

Denoting the set of simple roots of $\gfrak$ by $\Delta$, i.e.,
$\Delta=\{\alpha_i\}_{i\in I}$, we define the subset $\Delta^{\ast}$ of
marked simple roots by
\[
\Delta^{\ast}=\begin{cases}
\{\alpha_0\} & \text{for $\gfrak=\mathrm{A}^{(2)}_{2n}$}, \\
\emptyset & \text{for $\gfrak=\mathrm{C}^{(1)}_n$}, \\
\{\alpha_0,\alpha_n\} & \text{for $\gfrak=\mathrm{D}^{(2)}_{n+1}$}.
\end{cases}
\]
Hence the marked simple roots correspond to the vertices of the
Dynkin diagrams coloured red while the unmarked simple roots 
corresponds to the blue vertices.

The marks and comarks (or labels and colabels) $a_i$ and
$a_i^{\vee}$ for $i\in I$ are positive integers such that
$\sum_{i\in I} a_{ij}a_j=\sum_{i\in I} a_i^{\vee}a_{ij}=0$
and $\gcd(a_i)_{i\in I}=\gcd(a_i^{\vee})_{i\in I}=1$.
In particular, for the three cases of interest the comarks
are given by
\[
(a_0^{\vee},a_1^{\vee},\dots,a_{n-1}^{\vee},a_n^{\vee})=
\begin{cases}
(1,2,\dots,2,2) & \text{for $\gfrak=\mathrm{A}^{(2)}_{2n}$}, \\
(1,1,\dots,1,1) & \text{for $\gfrak=\mathrm{C}^{(1)}_n$}, \\
(1,2,\dots,2,1) & \text{for $\gfrak=\mathrm{D}^{(2)}_{n+1}$}.
\end{cases}
\]
The null or imaginary root $\delta$ is defined as
$\delta=\sum_{i\in I} a_i\alpha_i$ and together with the
fundamental weights $\La_i$ ($i\in I$), given by 
$\ip{\La_i}{\alpha_j^{\vee}}=\delta_{ij}$ and $\ip{\La_i}{d}=0$,
yields an alternative basis of $\hfrak^{\ast}$.
Deviating slightly from \cite{Kac90} by dropping the $\delta$-part,
we define the set of dominant integral weights of $\gfrak$ as
\[
P_{+}:=\sum_{i\in I} \mathbb{N}_0\La_i.
\]
The one-dimensional center of $\gfrak$ is spanned by the canonical
central element $K=\sum_{i\in I} a_i^{\vee}\alpha_i^{\vee}\in\hfrak$.
In terms of $K$, the level of $\La\in\hfrak^{\ast}$ is defined as
$\lev(\La):=\ip{\La}{K}$, and thus $\lev(\La_i)=a_i^{\vee}$.
For $k$ a nonnegative integer, we further let
$P_{+}^k:=\{\La\in P_{+}: \lev(\La)=k\}$ be the set of
level-$k$ dominant integral weight.
In our discussion of the CMPP and Primc conjectures, it will be 
convenient to define a set of scaled fundamental weights
$\{\Uplambda_i\}_{i\in I}$ as follows:
\[
(\Lap_0,\Lap_1,\dots,\Lap_{n-1},\Lap_n)=
\begin{cases}
(2\La_0,\La_1,\dots,\La_{n-1},\La_n)
& \text{for $\gfrak=\mathrm{A}^{(2)}_{2n}$}, \\
(\La_0,\La_1,\dots,\La_{n-1},\La_n)
& \text{for $\gfrak=\mathrm{C}^{(1)}_n$}, \\
(2\La_0,\La_1,\dots,\La_{n-1},2\La_n)
& \text{for $\gfrak=\mathrm{D}^{(2)}_{n+1}$}.
\end{cases}
\]
Note in particular that for all $i\in I$ we have
$\lev(\Lap_i)=2$ for $\gfrak=\mathrm{A}^{(2)}_{2n}$ or
$\gfrak=\mathrm{D}^{(2)}_{n+1}$ and $\lev(\Lap_i)=1$ for
$\gfrak=\mathrm{C}^{(1)}_n$, and that it is exactly the fundamental
weights corresponding to the marked vertices of the Dynkin diagram
that have been scaled by a factor of two in going from $\La_i$ to
$\Lap_i$.

The standard modules (also known as integrable highest weight modules)
of $\gfrak$ are indexed by dominant integral weights, and in the
following $L(\La)$ will denote the unique standard module of highest
weight $\La\in P_{+}$.
The character of $L(\Lambda)$ is defined as
\[
\ch L(\La)=\sum_{\mu\in\hfrak^{\ast}} \dim(V_{\mu}) \Exp{\mu},
\]
where $\Exp{\cdot}$ is a formal exponential and $\dim(V_{\mu})$ is
the dimension of the weight space $V_{\mu}$ in the weight-space 
decomposition of $L(\La)$.
As in the introduction, we define the normalised character 
$\chi_{\La}:=\Exp{-\La}\ch L(\La)$, so that
\[
\chi_{\La}\in\mathbb{Z}[[\Exp{-\alpha_0},\dots,\Exp{-\alpha_n}]].
\]
According to the Weyl--Kac character formula \cite{Kac90},
\begin{equation}\label{Eq_Weyl-Kac}
\chi_{\La}=
\frac{\sum_{w\in W}\sgn(w) \Exp{w(\La+\rho)-\La-\rho}}
{\prod_{\alpha>0}(1-\Exp{-\alpha})^{\mult(\alpha)}},
\end{equation}
where $W$ is the Weyl group of $\gfrak$, $\sgn(w)$ the signature of
$w\in W$ and $\rho=\sum_{i\in I}\La_i$ is the Weyl vector.
The product over $\alpha>0$ is a product over the
positive roots of the root system of $\gfrak$ and
$\mult(\alpha)$ is the dimension of the root space
corresponding to $\alpha$.
Rather than the full characters of $\gfrak$ we require specialisations.
In all three cases under consideration we use the same notation for this
specialisation, despite its $\gfrak$-dependence.
Following \cite{GOW16}, we define
$\varphi_n:\mathbb{Z}[[\Exp{-\alpha_0},\dots,\Exp{-\alpha_n}]]\to
\mathbb{Z}[[q]]$ by
\begin{equation}\label{Eq_NPS}
\varphi_n(\Exp{-\alpha})=\begin{cases}
-1 & \text{if $\alpha\in\Delta^{\ast}$}, \\
q & \text{if $\alpha\in\Delta\setminus\Delta^{\ast}$}.
\end{cases}
\end{equation}
For $\gfrak=\mathrm{C}^{(1)}_n$, since there are no marked simple roots,
this is the well-known principal specialisation, but for the other two
types the above specialisation is non-standard.
From Lepowsky's numerator formula \cite{Lepowsky82} for
$\mathrm{C}^{(1)}_n$ (see also \cite{BKMP24}) and results from \cite{GOW16}
for $\mathrm{A}^{(2)}_{2n}$ and $\mathrm{D}^{(2)}_{n+1}$, it follows that
in all three cases the $\varphi_n$-specialisation of $\chi_{\la}$ admits
a product form.
For $\gfrak=\mathrm{A}^{(2)}_{2n}$ and $\La\in P_{+}^{2k}$ parametrised
as in \eqref{Eq_La-para}, i.e., as
\begin{equation}\label{Eq_La-para2}
\La=(k-\la_1)\Lap_0+(\la_1-\la_2)\Lap_1+\dots+
(\la_{n-1}-\la_n)\Lap_{n-1}+\la_n\Lap_n,
\end{equation}
the specialisation $\varphi_n(\chi_{\La})$ is given by
\eqref{Eq_A2n-product}.
Similarly, for $\gfrak=\mathrm{C}^{(1)}_n$ and $\La\in P_{+}^k$ 
parametrised as in \eqref{Eq_La-para2} we have
\begin{align}\label{Eq_Cn-PS}
&\varphi_n(\chi_{\La}) \\
&\;\:=\frac{(q^{k+n+1};q^{2k+2n+2})_{\infty}
(q^{2k+2n+2};q^{2k+2n+2})_{\infty}^n}{(q;q^2)_{\infty}(q;q)_{\infty}^n} 
\prod_{i=1}^n \theta\big(q^{\la_i+n-i+1};q^{k+n+1}\big) \notag \\
&\qquad\times \prod_{1\leq i<j\leq n} 
\theta\big(q^{\la_i-\la_j-i+j},q^{\la_i+\la_j+2n-i-j+2};q^{2k+2n+2}\big).
\notag
\end{align}
Finally, for $\gfrak=\mathrm{D}^{(2)}_{n+1}$,
\begin{align}\label{Eq_D2n-product}
&\varphi_n(\chi_{\La}) \\
&\;\:=\frac{(q^{2k+2n};q^{2k+2n})_{\infty}^n}
{(q^2;q^2)_{\infty}(q;q)_{\infty}^{n-1}}  
\prod_{1\leq i<j\leq n} 
\theta\big(q^{\la_i-\la_j-i+j},q^{\la_i+\la_j+2n-i-j+1};q^{2k+2n}\big),
\notag
\end{align}
where again $\La$ is given by \eqref{Eq_La-para2}.
Since the proofs of the two non-principal product forms were not included
in paper \cite{GOW16}, we present the full details of their derivation in
the appendix.

\subsection{Hall--Littlewood polynomials}\label{Sec_HL}

The Hall--Littlewood polynomials $P_{\la}(t)$ are an important family of
symmetric functions, interpolating between the Schur functions $s_{\la}$ 
and monomial symmetric functions $m_{\la}$.
They have long been known to be related to characters of affine Lie
algebras and identities of the Rogers--Ramanujan type,
see \cite{BW15,Fulman00,GOW16,HJ23,IJZ06,JZ05,Kirillov00,RW21,Stembridge90,W06a,W06b,WZ12}.

For $\la$ a partition of length at most $k$, the Hall--Littlewood
polynomial is defined as~\cite{Macdonald95}
\[
P_{\la}(t)=P_{\la}(x_1,\dots,x_k;t):=\sum_{w\in S_k/S_k^{\la}}
w\bigg( x_1^{\la_1}\cdots x_k^{\la_k} 
\prod_{\la_i>\la_j} \frac{x_i-tx_j}{x_i-x_j}\bigg),
\]
where $S_k$ is the symmetric group of degree $k$.  
$P_{\la}(x_1,\dots,x_k;t)$ is symmetric in the $x_i$ (with coefficients
in $\mathbb{Z}[t]$) and homogeneous of degree $\abs{\la}$.
By the stability property
\[
P_{\la}(x_1,\dots,x_k,0;t)=
\begin{cases}
P_{\la}(x_1,\dots,x_k;t) & \text{if $l(\la)\leq k$}, \\
0 & \text{if $l(\la)=k+1$},
\end{cases}
\]
the Hall--Littlewood polynomials extend to symmetric functions in 
infinitely many variables in the usual fashion, see~\cite{Macdonald95}
for details.
Well-known special cases of $P_{\la}(t)$ are $P_{\la}(0)=s_{\la}$,
$P_{\la}(1)=m_{\la}$ and $P_{(1^r)}=e_r$, with $e_r$ the $r$th elementary
symmetric function.

An important result for Hall--Littlewood polynomials is the principal
specialisation formula~\cite[page 213]{Macdonald95}
\begin{equation}\label{Eq_P-spec}
P_{\la}(1,t,\dots,t^{k-1};t)=\frac{t^{n(\la)}(t;t)_k}
{(t;t)_{k-l(\la)} \prod_{i\geq 1}(t;t)_{\la'_i-\la'_{i+1}}}
=\frac{t^{n(\la)}(t;t)_k}{\prod_{i\geq 0}(t;t)_{f_i(\la)}},
\end{equation}
where $l(\la)\leq k$, $n(\la):=\sum_{i\geq 1} (i-1)\la_i$ and
$f_0(\la):=k-l(\la)$.
In order to prove Conjecture~\ref{Con_A2n2-qseries} for $k=1$, we
require the following generalisation of the large-$k$ limit of
\eqref{Eq_P-spec} for $\la=(2^r)$, which was stated without proof
in~\cite[Equation (2.7)]{GOW16}.

\begin{proposition}\label{Prop_GOW}
For $r,n$ nonnegative integers and $\delta\in\{0,1\}$, 
\begin{align}\label{Eq_GOW27}
&P_{(2^r)}\big(1,q,q^2,\dots;q^{2n+\delta}\big) \\
&\quad=\sum_{r_1,\dots,r_n\geq 0}
\frac{q^{r^2-r+r_1^2+\dots+r_n^2+r_1+\dots+r_n}}
{(q;q)_{r-r_1}(q;q)_{r_1-r_2}
\cdots(q;q)_{r_{n-1}-r_n}(q^{2-\delta};q^{2-\delta})_{r_n}}.\notag
\end{align}
\end{proposition}

\begin{proof}
For integers $k,n$, let 
\[
\qbin{n}{k}=\qbin{n}{k}_q=
\begin{cases}
\displaystyle
\frac{(q;q)_n}{(q;q)_k(q;q)_{n-k}} & \text{if $0\leq k\leq n$}, \\[3mm]
0 & \text{otherwise},
\end{cases}
\]
be a $q$-binomial coefficient.
Our starting point for proving \eqref{Eq_GOW27} is a formula for the 
Hall--Littlewood polynomials due to Lassalle and Schlosser stated 
in \cite[Theorem 7.1]{LS06}.
For $\la=(2^r,1^s)$ (the partition with $r$ parts of size $2$ and $s$ parts
of size $1$) the Lassalle--Schlosser result simplifies to
\[
P_{(2^r,1^s)}(t)=
\sum_{i=0}^r (-1)^i t^{\binom{i}{2}} \frac{1-t^{2i+s}}{1-t^{i+s}}
\qbin{i+s}{s}_t e_{r-i} e_{r+s+i},
\]
where $(1-t^{2i+s})/(1-t^{i+s})$ should be interpreted as $1$ if 
$i=0$ for all $s\geq 0$ (i.e., including when $s=0$).
By \cite[page 27]{Macdonald95}
\[
e_r(1,q,\dots,q^{k-1})=q^{\binom{r}{2}}\qbin{k}{r},
\]
(which is \eqref{Eq_P-spec} for $\la=(1^r)$), this yields
\begin{align}\label{Eq_LS-specialised}
&P_{(2^r,1^s)}(1,q,\dots,q^{k-1};t) \\
&\quad=\sum_{i=0}^r (-1)^i
t^{\binom{i}{2}}q^{\binom{r-i}{2}+\binom{r+s+i}{2}}
\frac{1-t^{2i+s}}{1-t^{i+s}} \qbin{i+s}{s}_t
\qbin{k}{r-i}\qbin{k}{r+s+i}. \notag
\end{align}
We note that for $t=q$ the sum over $i$ can be carried out by the
very-well poised $_6\phi_5$ summation \cite[Equation (II.21)]{GR04}
with $(a,b,c,n)\mapsto (q^s,q^{-(k-r-s)},\infty,r)$ to give
\begin{equation}\label{Eq_6phi5}
P_{(2^r,1^s)}(1,q,\dots,q^{k-1};q)=
\frac{q^{\binom{r}{2}+\binom{r+s}{2}}(q;q)_k}
{(q;q)_{k-r-s}(q;q)_r(q;q)_s},
\end{equation}
in accordance with \eqref{Eq_P-spec}.

A pair of sequences
$(\alpha,\beta)=\big((\alpha_r)_{r\geq 0},(\beta_r)_{r\geq 0}\big)$
such that
\begin{equation}\label{Eq_Bailey-pair}
\beta_r=\sum_{i=0}^r \frac{\alpha_i}{(q;q)_{r-i}(aq;q)_{r+i}}
\end{equation}
is known as a Bailey pair relative to $a$,
see \cite{Bailey49,Andrews84,W01}.
The identity that arises after taking the large-$k$ limit of 
\eqref{Eq_LS-specialised} is equivalent to the statement that
\begin{subequations}\label{Eq_alpha-beta}
\begin{align}\label{Eq_alpha}
\alpha_i&=(-1)^i t^{\binom{i}{2}} q^{i(i+s)}
\frac{1-t^{2i+s}}{1-t^{i+s}} \qbin{i+s}{s}_t, \\
\beta_r&=q^{-\binom{r}{2}-\binom{r+s}{2}}
(q;q)_s \, P_{(2^r,1^s)}(1,q,q^2,\dots;t)
\label{Eq_beta}
\end{align}
\end{subequations}
forms a Bailey pair relative to $q^s$.
Specialising $s=0$ and $t=q^{2n+\delta}$ for $n$ a nonnegative
integer and $\delta\in\{0,1\}$, the resulting $\alpha$-sequence
corresponds to a known Bailey pair relative to $1$ (given by 
\cite[Equation (4.13)]{SW98} with $(k,i)\mapsto (n+2,2)$), 
with corresponding $\beta$-sequence
\[
\beta_r=\sum_{r_1,\dots,r_n\geq 0}
\frac{q^{r_1^2+\dots+r_n^2+r_1+\dots+r_n}}
{(q)_{r-r_1}(q)_{r_1-r_2}
\cdots(q)_{r_{n-1}-r_n}(q^{2-\delta};q^{2-\delta})_{r_n}}.
\]
Since the $\alpha$-sequence determines the $\beta$-sequence uniquely,
the above expression may be equated with \eqref{Eq_beta} for $s=0$ and
$t=q^{2n+\delta}$, completing the proof.

We remark that, by \eqref{Eq_6phi5} for $k\to\infty$, the
Bailey pair \eqref{Eq_alpha-beta} for $s=1$ and $t=q$ is equivalent
to the Bailey pair B(3) in Slater's famous list of Bailey
pairs~\cite{Slater51}.
Unfortunately, in the general $t=q^{2n+\delta}$ case the 
$\alpha$-sequence \eqref{Eq_alpha} for $s=1$ does not correspond to 
anything known.
\end{proof}

From \eqref{Eq_PS-2lambda} (which is a special case of \eqref{Eq_P-spec}),
\[
\sum_{\substack{\la \\[1pt] \la_1\leq k}}
(zq)^{\abs{\la}} P_{2\la}(1,q,q^2,\dots;q)=
\sum_{r_1,\dots,r_k\geq 0}\,
\prod_{i=1}^k \frac{z^{r_i} q^{r_i^2}}{(q;q)_{r_i-r_{i+1}}}.
\]
Conjecturally, this generalises to Hall--Littlewood functions of base $q^2$.

\begin{conjecture}\label{Con_Shun}
For $k$ a nonnegative integer,
\begin{equation}\label{Eq_q2}
\sum_{\substack{\la \\[1pt] \la_1\leq k}}
(zq)^{\abs{\la}} P_{2\la}(1,q,q^2,\dots;q^2)
=\sum_{\substack{r_1,\dots,r_k\geq 0 \\[1pt] s_1,\dots,s_k\geq 0}}
\, \prod_{i=1}^k \frac{z^{r_i+s_i} q^{(r_i+s_i)^2+s_i^2+s_i}}
{(q;q)_{r_i-r_{i+1}}(q^2;q^2)_{s_i-s_{i-1}}},
\end{equation}
where $r_{k+1}=s_0:=0$.
\end{conjecture}

Apart from the trivial case $k=0$ this also holds for $k=1$ by
\eqref{Eq_GOW27} for $n=1$, $\delta=0$ and $(r,r_1)\mapsto (r_1+s_1,s_1)$.


\section{Special cases of the CMPP conjecture}\label{Sec_Special}

There is a `rough' level-rank duality at play in
Conjecture~\ref{Con_A2n2}, corresponding to the interchange of $k$
and $n$.
It is not the case, however, that for every identity corresponding to
a level-$2k$ standard module of $\mathrm{A}^{(2)}_{2n}$ there is a
counterpart in terms of a level-$2n$ standard module of
$\mathrm{A}^{(2)}_{2k}$, and a one-to-one correspondence only occurs
for $n=1$ and $n=2$ (or, by level-rank duality, for $k=1$ and $k=2$).
For $n=1$, let $i,k$ be integers such that $0\leq i\leq k$.
Then
\begin{equation}\label{Eq_level-rank-n=1}
\varphi_1\big(\chi_{(k-i)\Lap_0+i\Lap_1}\big)=
\begin{cases}
\varphi_k\big(\chi_{\Lap_{i/2}}\big) & \text{if $i$ is even}, \\[2mm]
\varphi_k\big(\chi_{\Lap_{k-(i-1)/2}}\big) & \text{if $i$ is odd}.
\end{cases}
\end{equation}
Similarly, for $n=2$, let $i,j,k$ be integers such that
$0\leq i\leq j\leq k$.
Then
\begin{equation}\label{Eq_level-rank-n=2}
\varphi_2\big(\chi_{(k-j)\Lap_0+(j-i)\Lap_1+i\Lap_2}\big)
 = \begin{cases}
\varphi_k\big(\chi_{\Lap_{(j-i)/2}+\Lap_{(i+j)/2}}\big) 
& \text{if $i+j$ is even}, \\[2mm]
\varphi_k(\chi_{\Lap_{k-(i+j-1)/2}+\Lap_{k-(j-i-1)/2}}\big)
& \text{if $i+j$ is odd}.
\end{cases}
\end{equation}
As mentioned above, for general $n$ and $k$ there only is a partial
level-rank duality, depending in a non-trivial manner on the choice
of highest weight.
There are, however, some general patterns, the simplest of which are
\begin{align*}
\varphi_n\big(\chi_{k\Lap_0}\big)&=\varphi_k\big(\chi_{n\Lap_0}\big),
&& \hspace{-15mm} k,n\geq 1, \\
\varphi_n\big(\chi_{(k-2)\Lap_0+2\Lap_1}\big)
&=\varphi_k\big(\chi_{(n-2)\Lap_0+2\Lap_1}\big),
&& \hspace{-15mm} k,n\geq 2, \\
\varphi_n\big(\chi_{(k-3)\Lap_0+2\Lap_1+\Lap_2}\big)
&=\varphi_k\big(\chi_{(n-3)\Lap_0+2\Lap_1+\Lap_2}\big),
&& \hspace{-15mm} k,n\geq 3,
\end{align*}
where, in accordance with \eqref{Eq_level-rank-n=1} and
\eqref{Eq_level-rank-n=2}, the second duality holds for all $k,n\geq 1$ 
and the third duality for all $k,n\geq 2$ provided $-\Lap_0+2\Lap_1$ is
interpreted as $\Lap_1$.
In terms of Conjectures~\ref{Con_A2n2} and \ref{Con_A2n2-rep},
the duality \eqref{Eq_level-rank-n=1} may be expressed as follows.
For $a\in\{0,\dots,n\}$, let
\[
\A_a^{(n)}(N):=\A_{0^a,1,0^{n-a}}(N)
\quad\text{and}\quad
\A_a^{(n)}:=\bigcup_{N\geq 0} \A_a^{(n)}(N).
\]
Then \eqref{Eq_CMPP2} for $k_0+\dots+k_n=1$ gives
\[
\sum_{\la\in\A_a^{(n)}} q^{\abs{\la}}=
\varphi_n\big(\chi_{\Lap_a}\big).
\]
By \eqref{Eq_level-rank-n=1} with $k\mapsto n$ this yields
\[
\sum_{\la\in\A_a^{(n)}} q^{\abs{\la}} =
\begin{cases}
\varphi_1\big(\chi_{(n-2a)\Lap_0+2a\Lap_1}\big) &
\text{for $0\leq a\leq\floor{n/2}$}, \\[2mm]
\varphi_1\big(\chi_{(2a-n-1)\Lap_0+(2n-2a+1)\Lap_1}\big) &
\text{for $\floor{n/2}<a\leq n$}.
\end{cases}
\]
Finally, by \eqref{Eq_A2n-product} with $n=1$, $\kappa=2n+3$ and
$(\la_0,\la_1)=(2n-2a,2a)$ in the first case and
$(\la_0,\la_1)=(2a-1,2n-2a+1)$ in the second case, we find
\begin{equation}\label{Eq_JMS}
\sum_{\la\in\A_a^{(n)}} q^{\abs{\la}} =
\frac{(q^{2a+1},q^{2n-2a+2},q^{2n+3};q^{2n+3})_{\infty}}{(q;q)_{\infty}}.
\end{equation}
With some effort this may be seen to correspond to a coloured
partition theorem due to Jing, Misra and Savage, stated as
\cite[Theorem 1.2; $M$ odd]{JMS01}.
The exact correspondence between $\A_a^{(n)}(N)$ and
the set of coloured partitions $C_N(M,r)$ defined by Jing, Misra and
Savage is as follows:
\[
M=2n+3,\quad 
r=\begin{cases} 2a+1 & \text{for $0\leq a\leq \floor{n/2}$}, \\
2n-2a+2 & \text{for $\floor{n/2}<a\leq n$},
\end{cases}
\]
where in the second case, the colour labelling of \cite{JMS01} needs to be
reversed, i.e., the colour $c$ should be mapped to $n+1-c$.
(This also affects the actual definition of coloured partitions in
\cite{JMS01} since an order on the colour labels is assumed.)
Jing, Misra and Savage prove their result by showing that
the set of coloured partition $C_N(2n+3,r)$ is in bijection with the
set of ordinary partitions $\la\vdash N$ such that
$\la_i-\la'_i\in\{2-r,\dots,2n-r+1\}$ for all $1\leq i\leq d$ where
$d=\max\{i\geq 1: \la_i\geq i\}$ is the size of the Durfee square of $\la$.
By the work of Andrews on successive rank partitions 
\cite[Theorem 4.1]{Andrews72}, this gives the claimed product form.
An alternative proof of \eqref{Eq_JMS} was given in \cite{Russell23},
based on the fact that if
\[
A^{(n)}_a(z)=A^{(n)}_a(z,q):=
\sum_{\la\in\A_a^{(n)}} z^{l(\la)}q^{\abs{\la}},
\]
then the following system of functional equations holds:
\begin{subequations}\label{Eq_MR}
\begin{align}
\label{Eq_MR1}
A^{(n)}_a(z)-A^{(n)}_{n-a}(zq)
&=\sum_{i=1}^a zq^{2i-1} A^{(n)}_{a-i+1}(zq^{2i})
+\sum_{i=1}^a zq^{2i} A^{(n)}_{n-a+i}(zq^{2i+1}), \\
A^{(n)}_{n-a}(z)-A^{(n)}_{a+1}(zq)
&=\sum_{i=1}^{a+1} zq^{2i-1} A^{(n)}_{n-a+i-1}(zq^{2i})
+\sum_{i=1}^a zq^{2i} A^{(n)}_{a-i+1}(zq^{2i+1}),
\label{Eq_MR2}
\end{align}
where $0\leq a\leq\floor{n/2}$ in \eqref{Eq_MR1} and
$0\leq a\leq\floor{(n-1)/2}$ in \eqref{Eq_MR2}.
\end{subequations}
Since these are equivalent to the Corteel--Welsh equations \cite{CW19}
for cylindric partitions of rank $2$ and level $2n+1$, we have the
known solution~\cite{Russell23,W23}\footnote{To be precise,
$A_a^{(n)}(z,q)$ corresponds to the normalised generating function
for cylindric partitions with profile $(2n-a+1,a)$, denoted by
$G_{(2n-a+1,a)}(z,q)$ in \cite{CW19}.
The Corteel--Welsh functional equations for the function $G_c(z,q)$, 
for more general profiles $c$ is given by \cite[Equation (3.5)]{CW19}.}
\begin{equation}\label{Eq_A-F}
A^{(n)}_a(z,q)=\begin{cases}
F^{(n)}_{2a,1}(z,q) & \text{for $0\leq a\leq\floor{n/2}$}, \\[2mm]
F^{(n)}_{2n-2a+1,1}(z,q) & \text{for $\floor{n/2}<a\leq n$}, \\
\end{cases}
\end{equation}
where, for $0\leq a\leq n$ and $\delta\in\{0,1\}$,
\begin{equation}\label{Eq_Fdef}
F^{(n)}_{a,\delta}(z,q):=\sum_{r_1,\dots,r_n\geq 0}
\frac{z^{r_1} q^{r_1^2+\dots+r_n^2+r_{a+1}+\dots+r_n}}
{(q;q)_{r_1-r_2}\cdots(q;q)_{r_{n-1}-r_n}
(q^{2-\delta};q^{2-\delta})_{r_n}}.
\end{equation}
Since for $z=1$ and $\delta=1$ this is exactly the left-hand side of
\eqref{Eq_AG2} with $k\mapsto n$, this once again
implies \eqref{Eq_JMS}.

Also Conjecture~\ref{Con_A2n2-qseries} can be shown to hold for $k=1$.

\begin{proposition}
Conjecture~\ref{Con_A2n2-qseries} holds for $k=1$ and all positive
integers $n$.
That is,
\begin{subequations}
\begin{align}\label{Eq_Ann}
A^{(n)}_n(z,q)=\sum_{\la\in\A_n^{(n)}}
z^{l(\la)} q^{\abs{\la}} &= 
\sum_{r=0}^{\infty} (zq)^r P_{(2^r)}\big(1,q,q^2,\dots;q^{2n-1}\big) 
\intertext{and}
A^{(n)}_0(z,q)=\sum_{\la\in\A_0^{(n)}}
z^{l(\la)} q^{\abs{\la}} &= 
\sum_{r=0}^{\infty} (zq^2)^r P_{(2^r)}\big(1,q,q^2,\dots;q^{2n-1}\big).
\end{align}
\end{subequations}
\end{proposition}

\begin{proof}
Let $\sigma\in\{0,1\}$.
Then, by Proposition~\ref{Prop_GOW} with $n\mapsto n-1$, $\delta=1$
and $(r,r_1,\dots,r_{n-1}) \mapsto (r_1,r_2,\dots,r_n)$,
\begin{align*}
&\sum_{r=0}^{\infty} (zq^{2-\sigma})^r 
P_{(2^r)}\big(1,q,q^2,\dots;q^{2n-1}\big) \\
&\quad=\sum_{r_1,\dots,r_n\geq 0}
\frac{z^{r_1} q^{r_1^2+\dots+r_n^2+(1-\sigma)r_1+r_2+\dots+r_n}}
{(q)_{r_1-r_2}\cdots(q)_{r_{n-1}-r_n}(q;q)_{r_n}}
=F^{(n)}_{\sigma,1}(z,q).
\end{align*}
By \eqref{Eq_A-F} for $a=\sigma n$, this yields
$A^{(n)}_{\sigma n}(z,q)$.
\end{proof}

To express the functional equations for the two-variable generating
function of CMPP partitions for arbitrary level and rank it is more
convenient to use dominant integral weights instead of the coefficients
$k_0,\dots,k_n$ of the respective fundamental weights 
$\Lap_0,\dots,\Lap_n$ as indexing set.
This motivates the definition
\[
\mathcal{A}^{(n)}_{k_0\Lap_0+\dots+k_n\Lap_n}(z)=
\mathcal{A}^{(n)}_{k_0\Lap_0+\dots+k_n\Lap_n}(z,q):=
\sum_{\la\in\A_{k_0,\dots,k_n}} z^{l(\la)} q^{\abs{\la}},
\]
so that $A^{(n)}_a(z)=\mathcal{A}^{(n)}_{\Lap_a}(z)$.
We expect that for all fixed $k,n$, the generating functions in
\[
\{\mathcal{A}^{(n)}_{\La}(z)\}_{\lev(\La)=2k}
\]
satisfy a system of functional equations generalising \eqref{Eq_MR},
uniquely determining all such functions.
For $n=1$ this system is given by the well-known Rogers--Selberg
functional equations \cite{RR19,Selberg36}:
\begin{equation}\label{Eq_RS}
\mathcal{A}^{(1)}_{(k-a)\Lap_0+a\Lap_1}(z)-
\mathcal{A}^{(1)}_{(k-a+1)\Lap_0+(a-1)\Lap_1}(z)=
(zq)^a\mathcal{A}^{(1)}_{a\Lap_0+(k-a)\Lap_1}(zq),
\end{equation}
where $0\leq a\leq k$ with $\mathcal{A}^{(1)}_{(k+1)\Lap_0-\Lap_1}(z):=0$.
It is exactly these equations that were solved by Andrews in
\cite{Andrews74} to show that 
$\mathcal{A}^{(1)}_{(k-a+1)\Lap_0+(a-1)\Lap_1}(z)$ (for $1\leq a\leq k+1$)
is given by the multisum expression in \eqref{Eq_AG}.
Below we give an incomplete set of functional equations for arbitrary
$k$ and $n$.

\begin{proposition}\label{Prop_fun}
For $a,k$ integers such that $0\leq a\leq k$,
\begin{subequations}
\begin{equation}\label{Eq_fun}
\mathcal{A}^{(n)}_{(k-a)\Lap_0+a\Lap_n}(z)=
\sum_{i=0}^a (zq)^i \mathcal{A}^{(n)}_{i\Lap_0+(a-i)\Lap_1+(k-a)\Lap_n}(zq).
\end{equation}
Moreover, for $k\geq 1$,
\begin{align}\label{Eq_fun2}
\mathcal{A}^{(n)}_{(k-1)\Lap_0+\Lap_1}(z)&=
\mathcal{A}^{(n)}_{\Lap_{n-1}+(k-1)\Lap_n}(zq)+
(zq^2)^k \mathcal{A}^{(n)}_{k\Lap_0}(zq^2) \\
&\quad +zq\sum_{i=0}^{k-1} (zq^2)^i 
\mathcal{A}^{(n)}_{i\Lap_0+(k-i)\Lap_1}(zq^2). \notag
\end{align}
\end{subequations}
\end{proposition}

For $n=1$ the functional equation \eqref{Eq_fun} simplifies to
\[
\mathcal{A}^{(1)}_{(k-a)\Lap_0+a\Lap_1}(z)=
\sum_{i=0}^a (zq)^i\mathcal{A}^{(1)}_{i\Lap_0+(k-i)\Lap_1}(zq),
\]
which is easily seen to be equivalent to the Rogers--Selberg
equations~\eqref{Eq_RS}.
For arbitrary $n$ but $k=1$ the $a=0$ case of \eqref{Eq_fun}
is \eqref{Eq_MR1} for $a=0$ and the $a=1$ case of \eqref{Eq_fun}
is a combination of the $a=0$ cases of \eqref{Eq_MR1} and
\eqref{Eq_MR2}.
Similarly, the $k=1$ and $a=0$ case of \eqref{Eq_fun2}
is a combination of the $a=0$ case of \eqref{Eq_MR1} and
the $a=1$ case of \eqref{Eq_MR2}.
By \eqref{Eq_fun} with $a=k$ and $z\mapsto zq$, the functional
equation \eqref{Eq_fun2} can be simplified to
\[
\mathcal{A}^{(n)}_{(k-1)\Lap_0+\Lap_1}(z)=
\mathcal{A}^{(n)}_{\Lap_{n-1}+(k-1)\Lap_n}(zq)+
(1-zq)(zq^2)^k \mathcal{A}^{(n)}_{k\Lap_0}(zq^2)
+zq\mathcal{A}^{(n)}_{k\Lap_n}(zq).
\]

\begin{proof}
We first prove \eqref{Eq_fun}, which is the simplest of the two claims.
The task is to show that the generating function
$\mathcal{A}^{(n)}_{(k-a)\Lap_0+a\Lap_n}(z)$ may be expressed as the sum
given on the right-hand side of \eqref{Eq_fun}.
We begin by noting that if $k_0=k-a$, $k_n=a$ and
$k_2=\dots=k_{n-1}=0$, then a necessary condition for a partition
$\la\in\mathcal{P}^{(n)}$ to have an admissible frequency-array is
that $f_1^{(c)}=0$ for $1\leq c\leq n-1$ and $f_1^{(n)}\leq a$.
(More generally, $f_i^{(c)}=0$ for $1\leq i\leq n-1$ and 
$\floor{i/2}<c\leq n-\ceil{i/2}$.)
Replacing $f_1^{(n)}$ by $i$, where $0\leq i\leq a$, the
first four columns of the frequency array of any partition
contributing to $\mathcal{A}^{(n)}_{(k-a)\Lap_0+a\Lap_n}(z)$ take
the form as shown in the left-most of the following three
(partial) frequency arrays:
\tikzmath{\y=10;}
\begin{center}
\begin{tikzpicture}[scale=0.43,line width=0.3pt]
\draw (0,8) node {$\sc a$};
\draw (0,6) node {$\sc 0$};
\draw (0,4.25) node {$\vdots$};
\draw (0,2) node {$\sc 0$};
\draw (0,0) node {$\sc 0$};
\draw (\x,-1) node {$\sc k-a$};
\draw (\x,1) node {$\sc 0$};
\draw (\x,3.25) node {$\vdots$};
\draw (\x,5) node {$\sc 0$};
\draw (\x,7) node {$\sc 0$};
\draw (2*\x,0) node {\red{$\sc 0$}};
\draw (2*\x,2) node {\blue{$\sc 0$}};
\draw (2*\x,4.25) node {$\vdots$};
\draw (2*\x,6) node {\brown{$\sc 0$}};
\draw (2*\x,8) node {\green{$\sc i$}};
\draw (3*\x,-1) node {\red{$\sc f_2^{(1)}$}};
\draw (3*\x,1) node {\blue{$\sc f_2^{(2)}$}};
\draw (3*\x,3.25) node {$\vdots$};
\draw (3*\x,5) node {\brown{$\sc f_2^{(n-1)}$}};
\draw (3*\x,7) node {\green{$\sc f_2^{(n)}$}};
\draw (4.5*\x,3.5) node {$\mapsto$};
\draw[rotate=-30] (-3,6.9) ellipse (1.6 and 0.5);
\draw (\y+0,8) node {$\sc 0$};
\draw (\y+0,6) node {$\sc 0$};
\draw (\y+0,4.25) node {$\vdots$};
\draw (\y+0,2) node {$\sc 0$};
\draw (\y,0) node {$\sc 0$};
\draw (\y+\x,-1) node {$\sc k-a$};
\draw (\y+\x,1) node {$\sc 0$};
\draw (\y+\x,3.25) node {$\vdots$};
\draw (\y+\x,5) node {$\sc 0$};
\draw (\y+\x,7) node {$\sc a-i$};
\draw (\y+2*\x,0) node {\red{$\sc 0$}};
\draw (\y+2*\x,2) node {\blue{$\sc 0$}};
\draw (\y+2*\x,4.25) node {$\vdots$};
\draw (\y+2*\x,6) node {\brown{$\sc 0$}};
\draw (\y+2*\x,8) node {\green{$\sc i$}};
\draw (\y+3*\x,-1) node {\red{$\sc f_2^{(1)}$}};
\draw (\y+3*\x,1) node {\blue{$\sc f_2^{(2)}$}};
\draw (\y+3*\x,3.25) node {$\vdots$};
\draw (\y+3*\x,5) node {\brown{$\sc f_2^{(n-1)}$}};
\draw (\y+3*\x,7) node {\green{$\sc f_2^{(n)}$}};
\draw (\y+4.5*\x,3.5) node {$\mapsto$};
\draw (2*\y+0,8) node {$\sc k-a$};
\draw (2*\y+0,6) node {$\sc 0$};
\draw (2*\y+0,4.25) node {$\vdots$};
\draw (2*\y+0,2) node {$\sc 0$};
\draw (2*\y+0,0) node {$\sc a-i$};
\draw (2*\y+\x,-1) node {$\sc i$};
\draw (2*\y+\x,1) node {$\sc 0$};
\draw (2*\y+\x,3.25) node {$\vdots$};
\draw (2*\y+\x,5) node {$\sc 0$};
\draw (2*\y+\x,7) node {$\sc 0$};
\draw (2*\y+2*\x,0) node {\red{$\sc f_1^{(1)}$}};
\draw (2*\y+2*\x,2) node {\blue{$\sc f_1^{(2)}$}};
\draw (2*\y+2*\x,4.25) node {$\vdots$};
\draw (2*\y+2*\x,6) node {\brown{$\sc f_1^{(n-1)}$}};
\draw (2*\y+2*\x,8) node {\green{$\sc f_1^{(n)}$}};
\end{tikzpicture}
\end{center}
where, to emphasise that the proof does not rely on
the fact that $\blue{f_2^{(2)}},\dots,\brown{f_2^{(n-1)}}$
all need to be zero for admissibility, we have not placed
any zeros in the fourth column.
The sole reason for including this column is to visualise
a relabelling of the frequencies $f_i^{(c)}$ in the final
step of the proof.

The objective now is to eliminate the first column.
To achieve this we note that any path $P$ on the left-most
array terminating at the vertex labelled $a$ must be of
the form $P=(p_1,\dots,p_{2n-2},0,a)$, where the final 
two entries correspond to the encircled pair of vertices.
For each such path $P$ there is a companion
$Q=(p_1,\dots,p_{2n-2},0,\green{i})$.
Since $i\leq a$, if $P$ is admissible then so is $Q$.
Replacing the encircled vertices $a$ and $0$ by
$0$ and $a-i$ respectively, $P$ and $Q$ are mapped to
$P'=(p_1,\dots,p_{2n-2},a-i,0)$ and 
$Q'=(p_1,\dots,p_{2n-2},a-i,i)$.
This time the admissibility of $Q'$ guarantees the 
admissibility of $P'$. 
Moreover, $Q'$ is admissible if and only if $P$ is admissible.
Hence we may replace the left-most array by the array shown 
in the middle of the above diagram as it allows for the exact
same set of admissible paths on the full array.
Since the left-most column in the second array contains only
zeros it is redundant and may thus be deleted.
After redrawing the resulting array upside-down and relabelling
$f_j^{(c)}$ by $f_{j-1}^{(n-c+1)}$ for $j\geq 2$ and $1\leq c\leq n$,
this yields the third of the above three arrays.
The contribution of the full set of arrays of this form to the
generating function is
$(zq)^i\mathcal{A}^{(n)}_{i\Lap_0+(a-i)\Lap_1+(k-a)\Lap_n}(zq)$.
Here the prefactor $(zq)^i$ accounts for the fact that the vertex
labelled $i$ in the second column arose from $f_1^{(n)}=i$ in the
original array, thus contributing $(zq)^i$ to the generating function.
The argument $zq$ instead of $z$ in the generating function
accounts for the fact that $f_j$ has been replaced by $f_{j-1}$.
Adding all the contributions from $i\in\{0,1,\dots,a\}$ results in
the right-hand side of \eqref{Eq_fun}.

\medskip
Next we prove the more complicated \eqref{Eq_fun2}.
The initial conditions $k_0=k-1$ and $k_1=1$ force $f_i^{(c)}=0$
for $1\leq i\leq 2$ and $2\leq c\leq n$, so that the first four columns
of the frequency array of any admissible partition take the form
\begin{center}
\begin{tikzpicture}[scale=0.43,line width=0.3pt]
\draw (0,8) node {$\sc 0$};
\draw (0,5.25) node {$\vdots$};
\draw (0,2) node {$\sc 0$};
\draw (0,0) node {$\sc 1$};
\draw (\x,-1) node {$\sc k-1$};
\draw (\x,1) node {$\sc 0$};
\draw (\x,5.25) node {$\vdots$};
\draw (\x,3) node {$\sc 0$};
\draw (\x,7) node {$\sc 0$};
\draw (2*\x,0) node {\red{$\sc f_1^{(1)} $}};
\draw (2*\x,2) node {\blue{$\sc 0$}};
\draw (2*\x,5.25) node {$\vdots$};
\draw (2*\x,8) node {\green{$\sc 0$}};
\draw (3*\x,-1) node {\red{$\sc f_2^{(1)}$}};
\draw (3*\x,1) node {\blue{$\sc f_2^{(2)}$}};
\draw (3*\x,3) node {\pink{$\sc 0$}};
\draw (3*\x,5.25) node {$\vdots$};
\draw (3*\x,7) node {\green{$\sc 0$}};
\end{tikzpicture}
\end{center}
\noindent
Considering $\red{f_1^{(1)}}$ and $\red{f_2^{(1)}}$, there are 
three possible scenarios.
The first is $\red{f_1^{(1)}}=0$ and $\red{f_2^{(1)}}=k$ (which forces
$\blue{f_2^{(2)}}=0$), the second is $\red{f_1^{(1)}}=0$ and 
$\red{f_2^{(1)}}<k$
and the third is $\red{f_1^{(1)}}=1$ and $\red{f_2^{(1)}}=i$
for $1\leq i\leq k-1$ (which again forces $\blue{f_2^{(2)}}=0$).
These three cases lead to the following three partial frequency arrays:

\medskip
\begin{minipage}{0.3\textwidth}
\begin{center}
\begin{tikzpicture}[scale=0.43,line width=0.3pt]
\draw (0,8) node {$\sc 0$};
\draw (0,5.25) node {$\vdots$};
\draw (0,2) node {$\sc 0$};
\draw (0,0) node {$\sc 1$};
\draw (\x,-1) node {$\sc k-1$};
\draw (\x,1) node {$\sc 0$};
\draw (\x,5.25) node {$\vdots$};
\draw (\x,3) node {$\sc 0$};
\draw (\x,7) node {$\sc 0$};
\draw (2*\x,0) node {\red{$\sc 0$}};
\draw (2*\x,2) node {\blue{$\sc 0$}};
\draw (2*\x,5.25) node {$\vdots$};
\draw (2*\x,8) node {\green{$\sc 0$}};
\draw (3*\x,-1) node {\red{$\sc k$}};
\draw (3*\x,1) node {\blue{$\sc 0$}};
\draw (3*\x,3) node {\pink{$\sc 0$}};
\draw (3*\x,5.25) node {$\vdots$};
\draw (3*\x,7) node {\green{$\sc 0$}};
\end{tikzpicture}
\end{center}
\end{minipage}
\begin{minipage}{0.3\textwidth}
\begin{center}
\begin{tikzpicture}[scale=0.4,line width=0.3pt]
\draw (0,8) node {$\sc 0$};
\draw (0,5.25) node {$\vdots$};
\draw (0,2) node {$\sc 0$};
\draw (0,0) node {$\sc 1$};
\draw (\x,-1) node {$\sc k-1$};
\draw (\x,1) node {$\sc 0$};
\draw (\x,5.25) node {$\vdots$};
\draw (\x,3) node {$\sc 0$};
\draw (\x,7) node {$\sc 0$};
\draw (2*\x,0) node {\red{$\sc 0$}};
\draw (2*\x,2) node {\blue{$\sc 0$}};
\draw (2*\x,5.25) node {$\vdots$};
\draw (2*\x,8) node {\green{$\sc 0$}};
\draw (3*\x,-1) node {\red{$\sc f_2^{(1)}$}};
\draw (3*\x,1) node {\blue{$\sc f_2^{(2)}$}};
\draw (3*\x,3) node {\pink{$\sc 0$}};
\draw (3*\x,5.25) node {$\vdots$};
\draw (3*\x,7) node {\green{$\sc 0$}};
\draw[rotate=30] (1.0,0) ellipse (1.6 and 0.5);
\end{tikzpicture}
\end{center}
\end{minipage}
\begin{minipage}{0.3\textwidth}
\begin{center}
\begin{tikzpicture}[use Hobby shortcut,scale=0.4,line width=0.3pt]
\draw (0,8) node {$\sc 0$};
\draw (0,5.25) node {$\vdots$};
\draw (0,2) node {$\sc 0$};
\draw (0,0) node {$\sc 1$};
\draw (\x,-1) node {$\sc k-1$};
\draw (\x,1) node {$\sc 0$};
\draw (\x,5.25) node {$\vdots$};
\draw (\x,3) node {$\sc 0$};
\draw (\x,7) node {$\sc 0$};
\draw (2*\x,0) node {\red{$\sc 1$}};
\draw (2*\x,2) node {\blue{$\sc 0$}};
\draw (2*\x,5.25) node {$\vdots$};
\draw (2*\x,8) node {\green{$\sc 0$}};
\draw (3*\x,-1) node {\red{$\sc i$}};
\draw (3*\x,1) node {\blue{$\sc 0$}};
\draw (3*\x,3) node {\pink{$\sc 0$}};
\draw (3*\x,5.25) node {$\vdots$};
\draw (3*\x,7) node {\green{$\sc 0$}};
\draw[closed] (\x,-1.5)..(3.7,0.4)..(\x,-0.3)..(-0.236,0.4);
\end{tikzpicture}
\end{center}
\end{minipage}

\medskip
\noindent
where in the second array it is assumed that $\red{f_2^{(1)}}<k$.
The first two columns in the left-most array can be deleted
without affecting the admissibility of any of the paths.
Relabelling $f_i^{(c)}$ as $f_{i-2}^{(c)}$ for $i\geq 2$ this
gives the contribution $(zq^2)^k \mathcal{A}^{(n)}_{k\Lap_0}(zq^2)$
to the generating function.
In the middle array, if we swap the positions of the encircled
$0$ and $1$, no paths are affected except for paths of the
form $(\red{f_2^{(1)}},\red{0},0,\dots)$ where the second of the
two zeros corresponds to the encircled zero.
Such paths map to $(\red{f_2^{(1)}},\red{0},1,\dots)$ by the swap,
which is exactly what is needed to ensure that $\red{f_2^{(1)}}<k$.
After the swap, the first column can be deleted.
Drawing the resulting array upside down and relabelling
$f_i^{(c)}$ as $f_{i-1}^{(n-c+1)}$ for $i\geq 1$ and $1\leq c\leq n$,
yields the contribution $\mathcal{A}^{(n)}_{\Lap_{n-1}+(k-1)\Lap_n}(zq)$
to the generating function.
Finally, in the right-most array we can replace the encircled triple
$1,k-1,\red{1}$ by $0,0,k-i$.
Prior to the replacement there are two types of paths through
at least one of the vertices in the first two columns labelled $k-1$
and $1$;
paths of the form $(k-1,1,0,p_4,p_5,\dots)$ and paths of the form 
$(k-1,\red{1},\bar{p}_3,\bar{p}_4,\dots)$.
Since the first two entries sum to $k$ in both cases, this forces
all the $p_i$ and $\bar{p}_i$ to be zero.
By the replacement, these paths become
$(0,0,0,p_4,p_5,\dots)$ and $(0,\red{1},\bar{p}_3,\bar{p}_4,\dots)$,
which imposes no constraints on the $p_i$ and the minor constraint 
$\sum_i \bar{p}_i\leq k-1$ on the $\bar{p}_i$.
However, prior to the replacement we also have the companion paths
$(\red{i},\red{1},0,p_4,p_5,\dots)$ and 
$(\red{i},\red{1},\bar{p}_4,\bar{p}_5,\dots)$
(imposing weaker constraints on the $p_i$ and $\bar{p}_i$ than vanishing,
unless $i=k-1$) which map to $(\red{i},\red{k-i},0,p_4,p_5,\dots)$ and 
$(\red{i},\red{k-i},\bar{p}_3,\bar{p}_4,\dots)$,
once again forcing all $p_i$ and $\bar{p}_i$ to be zero.
This justifies the above change in the encircled triple.
After the change, the first two columns can again be deleted.
Once again relabelling $f_i^{(c)}$ as $f_{i-2}^{(c)}$ for $i\geq 2$ this
yields the contribution 
$zq (zq^2)^i \mathcal{A}^{(n)}_{i\Lap_0+(k-i)\Lap_1}(zq^2)$ for each
$0\leq i\leq k-1$.
Adding up all the various contributions to the generating function
completes the proof.
\end{proof}


\section{The CMPP conjectures for \texorpdfstring{$\mathrm{C}^{(1)}_n$}{C(1)n} and
\texorpdfstring{$\mathrm{D}^{(2)}_{n+1}$}{D(2)n+1}}\label{Sec_CD}

There is a second conjecture in the work of CMPP that was subsequently 
complemented by Primc in \cite{Primc24}.
These two conjectures concern (a) partitions in
$\mathcal{P}^{(n+1)}$ in which the parts of the $(n+1)$th colour are
all odd (\hspace{1sp}\cite[Conjecture~3.3]{CMPP22}) or (b)
partitions in $\mathcal{P}^{(n)}$ in which the parts of the $n$th colour
are all even (\hspace{1sp}\cite[Conjecture~2.1]{Primc24}).
For reasons of symmetry we will adopt somewhat different
coloured-partition models (along the lines of \cite{DK22}) to describe
these two additional conjectures, instead considering
(a') partitions in $\mathcal{P}^{(2n+1)}$ in which all parts with an 
odd colour label are odd and all parts with an even colour label are
even (i.e., $f_i^{(c)}=0$ unless $c+i$ is even) and
(b') partitions in $\mathcal{P}^{(2n-1)}$ in which all parts with an
even colour label are odd and all parts with an odd colour label are
even (i.e., $f_i^{(c)}=0$ unless $c+i$ is odd).
In the corresponding frequency arrays for these two sets of coloured
partitions the ``forbidden'' $f_i^{(c)}$ are omitted rather than 
represented as zeros. 
For fixed nonnegative integers $k_0,\dots,k_n$, which are again to be
viewed as initial or boundary conditions, the frequency arrays then
take the form

\smallskip
\tikzmath{\z=17.5;}
\begin{center}
\begin{tikzpicture}[scale=0.37,line width=0.3pt]
\draw (0,7) node {$\sc k_n$};
\draw (0,4.75) node {$\sc \vdots$};
\draw (0,2) node {$\sc k_1$};
\draw (0,0) node {$\sc k_0$};
\draw (\x,1) node {$\sc 0$};
\draw (\x,3.75) node {$\sc \vdots$};
\draw (\x,6) node {$\sc 0$};
\draw (2*\x,0) node {\red{$\sc f_1^{(1)}$}};
\draw (2*\x,2) node {\pink{$\sc f_1^{(3)}$}};
\draw (2*\x,4.75) node {$\sc \vdots$};
\draw (2*\x,7) node {\green{$\sc f_1^{(2n+1)}$}};
\draw (3*\x,1) node {\blue{$\sc f_2^{(2)}$}};
\draw (3*\x,3.75) node {$\vdots$};
\draw (3*\x,6) node {\brown{$\sc f_2^{(2n)}$}};
\draw (4*\x,-2) node {(a')};
\draw (4*\x,0) node {\red{$\sc f_3^{(1)}$}};
\draw (4*\x,2) node {\pink{$\sc f_3^{(3)}$}};
\draw (4*\x,4.75) node {$\vdots$};
\draw (4*\x,7) node {\green{$\sc f_3^{(2n+1)}$}};
\draw (5*\x,1) node {\blue{$\sc f_4^{(2)}$}};
\draw (5*\x,3.75) node {$\vdots$};
\draw (5*\x,6) node {\brown{$\sc f_4^{(2n)}$}};
\draw (6*\x,0) node {\red{$\sc f_5^{(1)}$}};
\draw (6*\x,2) node {\pink{$\sc f_5^{(3)}$}};
\draw (6*\x,4.75) node {$\vdots$};
\draw (6*\x,7) node {\green{$\sc f_5^{(2n+1)}$}};
\draw (7*\x,1) node {\blue{$\sc f_6^{(2)}$}};
\draw (7*\x,3.75) node {$\vdots$};
\draw (7*\x,6) node {\brown{$\sc f_6^{(2n)}$}};
\draw (8*\x,0) node {\red{$\dots$}};
\draw (8*\x,2) node {\pink{$\dots$}};
\draw (8*\x,7) node {\green{$\dots$}};
\draw (\z,6) node {$\sc k_{n-1}$};
\draw (\z,4.75) node {$\sc \vdots$};
\draw (\z,3) node {$\sc k_2$};
\draw (\z,1) node {$\sc k_1$};
\draw (\z+\x,0) node {$\sc k_0$};
\draw (\z+\x,2) node {$\sc 0$};
\draw (\z+\x,3.75) node {$\sc \vdots$};
\draw (\z+\x,5) node {$\sc 0$};
\draw (\z+\x,7) node {$\sc k_n$};
\draw (\z+2*\x,1) node {\blue{$\sc f_1^{(2)}$}};
\draw (\z+2*\x,3.75) node {$\sc \vdots$};
\draw (\z+2*\x,6) node {\brown{$\sc f_1^{(2n-2)}$}};
\draw (\z+3*\x,0) node {\red{$\sc f_2^{(1)}$}};
\draw (\z+3*\x,2) node {\pink{$\sc f_2^{(3)}$}};
\draw (\z+3*\x,4.75) node {$\vdots$};
\draw (\z+3*\x,7) node {\green{$\sc f_2^{(2n-1)}$}};
\draw (\z+4*\x,1) node {\blue{$\sc f_3^{(2)}$}};
\draw (\z+4*\x,3.75) node {$\vdots$};
\draw (\z+4*\x,-2) node {(b')};
\draw (\z+4*\x,6) node {\brown{$\sc f_3^{(2n-2)}$}};
\draw (\z+5*\x,0) node {\red{$\sc f_4^{(1)}$}};
\draw (\z+5*\x,2) node {\pink{$\sc f_4^{(3)}$}};
\draw (\z+5*\x,4.75) node {$\vdots$};
\draw (\z+5*\x,7) node {\green{$\sc f_4^{(2n-1)}$}};
\draw (\z+6*\x,1) node {\blue{$\sc f_5^{(2)}$}};
\draw (\z+6*\x,3.75) node {$\vdots$};
\draw (\z+6*\x,6) node {\brown{$\sc f_5^{(2n-2)}$}};
\draw (\z+7*\x,0) node {\red{$\sc f_6^{(1)}$}};
\draw (\z+7*\x,2) node {\pink{$\sc f_6^{(3)}$}};
\draw (\z+7*\x,4.75) node {$\vdots$};
\draw (\z+7*\x,7) node {\green{$\sc f_6^{(2n-1)}$}};
\draw (\z+8*\x,1) node {\blue{$\dots$}};
\draw (\z+8*\x,6) node {\brown{$\dots$}};
\end{tikzpicture}
\end{center}
where $n\geq 1$ in case (a') and $n\geq 2$ in case (b').
In order to test some of the conjectures below, we extend 
the above to $n=0$ in case (a') and $n=1$ in case (b'),
both corresponding to one-row frequency arrays.
In the case of (b') the correct labelling of the 
single left-boundary vertex is given by $k_0+k_1$.

In case (a'), set $m:=2n+1$, 
$f_{-1}^{(2c+1)}:=k_c$ for $0\leq c\leq n$ and
$f_0^{(2c)}:=0$ for $1\leq c\leq n$.
Similarly, in case (b'), set $m=2n-1$,
$f_{-1}^{(2c)}:=k_c$ for $1\leq c\leq n-1$ and
$f_0^{(2c-1)}:=k_0\delta_{c,1}+k_n\delta_{c,n}$ 
for $1\leq c\leq n$.
Then a path on the frequency array of type (a') (resp.\ (b')) is a sequence
$P=(p_1,p_2,\dots,p_m)=(f_{i_1}^{(1)},f_{i_2}^{(2)},\dots,f_{i_m}^{(m)})$
such that for all $1\leq c\leq m$, $i_c\geq -1$, $i_c+c$ is even
(resp.\ odd) and $\abs{i_r-i_{r+1}}=1$.
A coloured partition of type (a') or (b') is admissible if for all paths
on its frequency array
\[
\sum_{i=1}^m p_i\leq k_0+\dots+k_n.
\]
In both cases, frequency arrays of coloured partitions exhibit a
$\mathbb{Z}_2$-symmetry, corresponding to $f_i^{(c)}\mapsto f_i^{(m-c+1)}$.
As we shall see shortly, this is a reflection of the diagram automorphisms
of the Dynkin diagrams of the underlying affine Lie algebras: 
$\mathrm{C}^{(1)}_n$ for the coloured partitions of type (a') and 
$\mathrm{D}^{(2)}_{n+1}$ for the coloured partitions of type (b').

To state the further conjectures of CMPP and Primc, we denote
the set of all admissible partitions of type (a') (resp.\ (b))
by $\mathscr{C}_{k_0,\dots,k_n}$ (resp.\ $\mathscr{D}_{k_0,\dots,k_n}$).

\begin{conjecture}[CMPP {\cite[Conjecture 3.3]{CMPP22}}]
\label{Con_Cn1}
For $k,n$ nonnegative integers and $k_0,\dots,k_n$ 
nonnegative integers such that $k_0+\dots+k_n=k$,
\begin{align*}
\sum_{\la\in\C_{k_0,\dots,k_n}} q^{\abs{\la}} 
&=\frac{(q^{k+n+1};q^{2k+2n+2})_{\infty}
(q^{2k+2n+2};q^{2k+2n+2})_{\infty}^{n}}{(q;q^2)_{\infty}(q;q)_{\infty}^n}
\\[-2mm]
&\quad\times \prod_{i=1}^n
\theta\big(q^{\la_i+n-i+1};q^{k+n+1}\big) \\
&\quad\times \prod_{1\leq i<j\leq n}
\theta\big(q^{\la_i-\la_j-i+j},q^{\la_i+\la_j+2n-i-j+2};q^{2k+2n+2}\big),
\end{align*}
where $\la_i:=k_i+\dots+k_n$ for $1\leq i\leq n$.
\end{conjecture}

For $k_0=k$ and $k_1=\dots=k_n=0$ this was first stated by
Primc and \v{S}iki\'c in the form of a conjectural basis for the
$\mathrm{C}^{(1)}_n$-standard module $L(k\Lap_0)$, see
\cite[Conjecture~1]{PS19} (and \cite[Theorem 12.1]{PS16}
for the $L(\Lap_0)$-case).
Building on results from \cite{BPT18}, this basis-conjecture was
recently proved by Primc and Trup\v{c}evi\'c in \cite[Theorem 2.1]{PT24}
by establishing a connection with Feigin--Stoyanovsky-type subspaces 
for $\mathrm{C}^{(1)}_{2n}$.
Hence for this special case Conjecture~\eqref{Con_Cn1} is now a
theorem.

Once again CMPP do not give the product on the right in the above 
explicit form, instead using a description in terms of congruence
triangles, see also \cite{BKMP24}.
Unlike Conjecture~\ref{Con_A2n2}, however, CMPP do identify the
product as the principally specialised character of the standard 
$\mathrm{C}^{(1)}_n$-module of highest weight
$\La=k_0\Lap_0+\dots+k_n\Lap_n$, i.e., as $\phi_n(\chi_{\La})$ 
in accordance with \eqref{Eq_Cn-PS}.
Here $\mathrm{C}^{(1)}_1$ should be interpreted as $\mathrm{A}^{(1)}_1$.
Specifically, Conjecture~\ref{Con_Cn1} for $n=1$ amounts to
\[
\sum_{\la\in\C_{k_0,k_1}} q^{\abs{\la}} 
=\frac{(q^{k_0+1},q^{k_1+1},q^{k_0+k_1+2};q^{k_0+k_1+2})_{\infty}}
{(q;q^2)_{\infty}(q;q)_{\infty}}, 
\]
where the right-hand side exactly is the principally specialised
$\mathrm{A}^{(1)}_1$-standard module of highest weight
$\La=k_0\La_0+k_1\La_1$, see \cite[Theorem 5.9]{LM78}.
Mapping the frequencies of the three-colour partitions to frequencies
of two-colour partitions as follows:
\smallskip
\tikzmath{\z=17.5;}
\begin{center}
\begin{tikzpicture}[scale=0.37,line width=0.3pt]
\draw (0,2) node {$\sc k_1$};
\draw (0,0) node {$\sc k_0$};
\draw (\x,1) node {$\sc 0$};
\draw (2*\x,0) node {\red{$\sc f_1^{(1)}$}};
\draw (2*\x,2) node {\pink{$\sc f_1^{(3)}$}};
\draw (3*\x,1) node {\blue{$\sc f_2^{(2)}$}};
\draw (4*\x,0) node {\red{$\sc f_3^{(1)}$}};
\draw (4*\x,2) node {\pink{$\sc f_3^{(3)}$}};
\draw (5*\x,1) node {\blue{$\sc f_4^{(2)}$}};
\draw (6*\x,0) node {\red{$\sc f_5^{(1)}$}};
\draw (6*\x,2) node {\pink{$\sc f_5^{(3)}$}};
\draw (7*\x,1) node {\blue{$\sc f_6^{(2)}$}};
\draw (8*\x,0) node {\red{$\dots$}};
\draw (8*\x,2) node {\pink{$\dots$}};
\draw (9.1*\x,1) node {$\longmapsto$};
\draw (\z+0,2) node {$\sc k_1$};
\draw (\z+0,0) node {$\sc k_0$};
\draw (\z+\x,1) node {$\sc 0$};
\draw (\z+2*\x,0) node {\red{$\sc f^{(2)}_1$}};
\draw (\z+2*\x,2) node {\blue{$\sc f_1^{(1)}$}};
\draw (\z+3*\x,1) node {\blue{$\sc f_2^{(1)}$}};
\draw (\z+4*\x,0) node {\blue{$\sc f_3^{(1)}$}};
\draw (\z+4*\x,2) node {\red{$\sc f_3^{(2)}$}};
\draw (\z+5*\x,1) node {\blue{$\sc f_4^{(1)}$}};
\draw (\z+6*\x,0) node {\red{$\sc f_5^{(2)}$}};
\draw (\z+6*\x,2) node {\blue{$\sc f_5^{(1)}$}};
\draw (\z+7*\x,1) node {\blue{$\sc f_6^{(1)}$}};
\draw (\z+8*\x,0) node {$\dots$};
\draw (\z+8*\x,2) node {$\dots$};
\end{tikzpicture}
\end{center}
where the second (red) colour on the right has only odd parts and
where the condition on the paths is unchanged, i.e., for
every path $P=(p_1,p_2,p_3)$, $p_1+p_2+p_3\leq k_0+k_1$,
we obtain the two-colour partition model of 
\cite[Equations (11.2.10) \& (11.2.11)]{MP99}.
By results of that same paper, this establishes
Conjecture~\ref{Con_Cn1} for $n=1$.
The case $n=0$, which has no Lie-algebraic interpretation, is
even simpler.
Since $\C_k$ is the set of ordinary partitions into odd parts such
that no part occurs more than $k$ times,
\begin{equation}\label{Eq_finite-ell}
\sum_{\substack{\la\in\C_k \\[1pt]\la_1\leq 2\ell-1}} 
\Big(\frac{z}{q}\Big)^{l(\la)} q^{\abs{\la}}
=\frac{(z^{k+1};q^{2k+2})_{\ell}}{(z;q^2)_{\ell}}.
\end{equation}
Taking the $\ell\to\infty$ limit and setting $z=q$, this implies that
\[
\sum_{\la\in\C_k} q^{\abs{\la}}
=\frac{(q^{k+1};q^{2k+2})_{\infty}}{(q;q^2)_{\infty}}
=\frac{(q^2;q^2)_{\infty}(q^{k+1};q^{2k+2})_{\infty}}
{(q;q)_{\infty}},
\]
in accordance with Conjecture~\ref{Con_Cn1} for $n=0$.

A final extremal case of the conjecture is $k=1$, i.e., 
$k_i=\delta_{i,a}$ for some fixed $a\in\{0,1,\dots,n\}$.
For such $a$, define
\[
\C_a^{(n)}(N):=\C_{0^a,1,0^{n-a}}(N)
\quad\text{and}\quad
\C_a^{(n)}:=\bigcup_{N\geq 0} \C_a^{(n)}(N).
\]
Then 
\[
\sum_{\la\in\C_a^{(n)}} q^{\abs{\la}}= 
\frac{(q^{2a+2},q^{2n-2a+2},q^{2n+4};q^{2n+4})_{\infty}}{(q;q)_{\infty}}.
\]
For $a=0$ this is \cite[Theorem 12.1]{PS16} and for arbitrary $a$
this follows from the coloured partition theorem \cite[Theorem 1.2]{JMS01}
of Jing et al.
In particular, the correspondence between $\C_a^{(n)}$ and the set of
coloured partitions $C_N(M,r)$ defined by Jing, Misra and Savage is
\[
M=2n+4,\quad 
r=\begin{cases} 2a+2 & \text{for $0\leq a\leq \floor{n/2}$}, \\
2n-2a+2 & \text{for $\ceil{n/2}\leq a\leq n$},
\end{cases}
\]
where in the second case the colour labelling of \cite{JMS01} needs to be
reversed.
Two further proofs of Conjecture~\ref{Con_Cn1} for level-one modules
have been found to date.
In \cite{DK22}, Dousse and Konan employed the theory of perfect 
crystals to prove a formula for the full character 
of any level-one standard module of $\mathrm{C}^{(1)}_n$.
Upon principal specialisation, their result implies the CMPP
formula for level one.
A third proof using functional equations was obtained by the second
author in \cite{Russell23}.
If
\[
C^{(n)}_a(z,q):=\sum_{\la\in\C^{(n)}_a} z^{l(\la)}q^{\abs{\la}},
\]
it follows from \cite{Russell23} that
\begin{equation}\label{Eq_C-F}
C^{(n)}_a(z,q)=\begin{cases}
F^{(n+1)}_{2a+1,0}(z,q) & \text{for $0\leq a\leq\floor{n/2}$}, \\[2mm]
F^{(n+1)}_{2n-2a+1,0}(z,q) & \text{for $\ceil{n/2}\leq a\leq n$},
\end{cases}
\end{equation}
with $F^{(n)}_{a,\delta}(z,q)$ defined in \eqref{Eq_Fdef}.
This implies the claimed product for $z=1$ thanks to Bressoud's
even modulus analogue of the Andrews--Gordon identities proven
in~\cite{Bressoud79,Bressoud80}.

\medskip

Primc's complement of Conjecture~\ref{Con_Cn1} may be stated as follows.

\begin{conjecture}[Primc {\cite[Conjecture 2.1]{Primc24}}]\label{Con_Dn2}
For $k$ a nonnegative integer, $n$ a positive integer and
$k_0,\dots,k_n$ nonnegative integers such that $k_0+\dots+k_n=k$,
\begin{align*}
&\sum_{\la\in\D_{k_0,\dots,k_n}} q^{\abs{\la}} \\
&\qquad= \frac{(q^{2k+2n};q^{2k+2n})_{\infty}^n}{(q^2;q^2)_{\infty}
(q;q)_{\infty}^{n-1}}
\prod_{1\leq i<j\leq n}
\theta\big(q^{\la_i-\la_j-i+j},q^{\la_i+\la_j+2n-i-j+1};q^{2k+2n}\big),
\end{align*}
where $\la_i:=k_i+\dots+k_n$ for $1\leq i\leq n$.
\end{conjecture}

Primc does not actually state what the product on the right should be.
His conjecture simply is that the generating function of partitions in
$\D_{k_0,\dots,k_n}$ such that $k_0+\dots+k_n=k$ 
can be expressed as an infinite product with modulus $2k+2n$.

Since $\D_{k_0,k_1}$ is the set of ordinary partitions into even parts
such that no part occurs more than $k_0+k_1=k$ times,
\[
\sum_{\substack{\la\in\D_{k_0,k_1} \\[1pt]\la_1\leq 2\ell}} 
\Big(\frac{z}{q^2}\Big)^{l(\la)} q^{\abs{\la}}
=\frac{(z^{k+1};q^{2k+2})_{\ell}}{(z;q^2)_{\ell}}.
\]
Taking the $\ell\to\infty$ limit and setting $z=q^2$ gives
\[
\sum_{\la\in\D_{k_0,k_1}} q^{\abs{\la}}
=\frac{(q^{2k+2};q^{2k+2})_{\infty}}{(q^2;q^2)_{\infty}},
\]
proving Conjecture~\ref{Con_Dn2} for $n=1$.
Unlike Conjecture~\ref{Con_Cn1} however, the case of three-row
frequency arrays (i.e., $n=2$) appears to still be open.

For $a\in\{0,1,\dots,n\}$, let
\[
\D_a^{(n)}(N):=\D_{0^a,1,0^{n-a}}(N)
\quad\text{and}\quad
\D_a^{(n)}:=\bigcup_{N\geq 0} \D_a^{(n)}(N).
\]
Then Conjecture~\ref{Con_Dn2} for $k=1$ is
\begin{equation}\label{Eq_Dk1}
\sum_{\la\in\D_a^{(n)}} q^{\abs{\la}}= 
\frac{(q^{2a+1},q^{2n-2a+1},q^{2n+2};q^{2n+2})_{\infty}}{(q;q)_{\infty}},
\end{equation}
which again follows from the coloured partition theorem
\cite[Theorem 1.2]{JMS01} of Jing et al.
The correspondence between $\D_a^{(n)}$ and the set of
coloured partitions $C_N(M,r)$ considered in \cite{JMS01} is
\[
M=2n+2,\quad 
r=\begin{cases} 2a+1 & \text{for $0\leq a\leq \floor{n/2}$}, \\
2n-2a+1 & \text{for $\ceil{n/2}\leq a\leq n$},
\end{cases}
\]
where in the second case the colour labelling of \cite{JMS01} needs to be
reversed.
Alternatively, if
\[
D_a^{(n)}(z,q):=\sum_{\la\in\D_a^{(n)}} z^{l(\la)}q^{\abs{\la}},
\]
then \cite{Russell23}
\begin{equation}\label{Eq_D-F}
D_a^{(n)}(z,q)=\begin{cases}
F^{(n)}_{2a,0}(z,q) & \text{for $0\leq a\leq\floor{n/2}$}, \\[2mm]
F^{(n)}_{2n-2a,0}(z,q) & \text{for $\ceil{n/2}\leq a\leq n$}.
\end{cases}
\end{equation}
Specialising $z=1$ and again appealing to Bressoud's
Rogers--Ramanujan-type identities for even moduli
\cite{Bressoud79,Bressoud80} implies \eqref{Eq_Dk1}.

By Proposition~\ref{Prop_Dn} with $\la_i=k_i+\dots+k_n$ for
$0\leq i\leq n$, Primc's conjecture is expressible in terms of the
affine Lie algebra $\mathrm{D}^{(2)}_{n+1}$ in the following manner,
resolving one of the open problems from~\cite{Primc24}.

\begin{conjecture}
For nonnegative integers $k_0,\dots,k_n$, let $L(\La)$ be the
$\mathrm{D}^{(2)}_{n+1}$-standard module of highest weight
$\La=k_0\Lap_0+k_1\Lap_1+\dots+k_n\Lap_n\in P_{+}^{2(k_0+\dots+k_n)}$.
Then
\[
\sum_{\la\in\D_{k_0,\dots,k_n}}q^{\abs{\la}}=\varphi_n(\chi_{\La}).
\]
\end{conjecture}

In the following a number of functional equations satisfied by the
coloured partitions of type $\mathrm{C}^{(1)}_n$ and
$\mathrm{D}^{(2)}_{n+1}$ will be considered.
To describe these we define 
\begin{align*}
\mathcal{C}^{(n)}_{k_0\Lap_0+\dots+k_n\Lap_n}(z)=
\mathcal{C}^{(n)}_{k_0\Lap_0+\dots+k_n\Lap_n}(z,q)&:=
\sum_{\la\in\C_{k_0,\dots,k_n}} z^{l(\la)} q^{\abs{\la}}, \\
\mathcal{D}^{(n)}_{k_0\Lap_0+\dots+k_n\Lap_n}(z)=
\mathcal{D}^{(n)}_{k_0\Lap_0+\dots+k_n\Lap_n}(z,q)&:=
\sum_{\la\in\D_{k_0,\dots,k_n}} z^{l(\la)} q^{\abs{\la}},
\end{align*}
so that $C^{(n)}_a(z,q)=\mathcal{C}^{(n)}_{\La_a}(z,q)$ and
$D^{(n)}_a(z,q)=\mathcal{D}^{(n)}_{\La_a}(z,q)$.
Since the set of ``$\mathrm{D}^{(2)}_2$-partitions'' 
$\D_{k_0,k_1}$ only depends on the sum of $k_1$ and $k_2$,
this also applies to $\mathcal{D}^{(1)}_{k_0\Lap_0+k_1\Lap_1}(z)$.
By the $\mathbb{Z}_2$-symmetry of the $\mathrm{C}^{(1)}_n$ and
$\mathrm{D}^{(2)}_{n+1}$-partitions,
\begin{subequations}\label{Eq_automorphism}
\begin{align}
\mathcal{C}^{(n)}_{k_0\Lap_0+k_1\Lap_1+\dots+k_n\Lap_n}(z)=
\mathcal{C}^{(n)}_{k_n\Lap_0+\dots+k_1\Lap_{n-1}+k_0\Lap_n}(z), \\
\mathcal{D}^{(n)}_{k_0\Lap_0+k_1\Lap_1+\dots+k_n\Lap_n}(z)=
\mathcal{D}^{(n)}_{k_n\Lap_0+\dots+k_1\Lap_{n-1}+k_0\Lap_n}(z),
\label{Eq_automorphism-D}
\end{align}
\end{subequations}
reflecting the diagram automorphisms of the corresponding Dynkin diagrams.

\begin{proposition}\label{Prop_fun-CDa}
For $a,k,n$ integers such that $0\leq a\leq k$,
\begin{subequations}
\begin{align}\label{Eq_fun-CD1}
\mathcal{C}^{(n)}_{a\Lap_0+(k-a)\Lap_n}(z)&=
\sum_{i=0}^a \sum_{j=0}^{k-a}
(zq)^{i+j}\,\mathcal{D}^{(n+1)}_
{i\Lap_0+(a-i)\Lap_1+(k-a-j)\Lap_n+j\La_{n+1}}(zq), \\
\mathcal{D}^{(n+1)}_{a\Lap_0+(k-a)\Lap_{n+1}}(z)&=
\mathcal{C}^{(n)}_{a\Lap_0+(k-a)\La_n}(zq), 
\label{Eq_fun-CD2}
\end{align}
where $n\geq 1$ in \eqref{Eq_fun-CD1} and $n\geq 0$ in 
\eqref{Eq_fun-CD2}.
\end{subequations}
\end{proposition}

For $n=0$ the functional equation \eqref{Eq_fun-CD2} is independent of
the choice of $a$.
Moreover, for $n>0$ the range of $a$ may be restricted to
$0\leq a\leq \floor{k/2}$ due to the symmetry \eqref{Eq_automorphism}.

For the low-rank cases $\mathcal{C}^{(1)}_{\La}(z)$ and
$\mathcal{D}^{(2)}_{\La}(z)$ additional equations holds.
Since by \eqref{Eq_fun-CD2}
\begin{equation}\label{Eq_CDn2}
\mathcal{C}^{(1)}_{a\Lap_0+(k-a)\La_1}(z)=
\mathcal{D}^{(2)}_{a\Lap_0+(k-a)\Lap_{2}}(z/q),
\end{equation}
it suffices to consider functional equations for
$\mathcal{D}^{(2)}_{\La}(z)$.

\begin{proposition}\label{Prop_fun-CDb}
For $a,b,k$ nonnegative integers such that $a+b\leq k-1$,
\begin{align}\label{Eq_nis2}
&\mathcal{D}^{(2)}_{a\Lap_0+(k-a-b)\Lap_1+b\Lap_2}(z)-
\mathcal{D}^{(2)}_{(a+1)\Lap_0+(k-a-b-1)\Lap_1+b\Lap_2}(z) \\
&\quad =\sum_{i=0}^a \sum_{j=0}^{k-a} (zq)^{k+i-a+\min\{0,j-b\}} q^{i+j}
\mathcal{D}^{(2)}_{i\Lap_0+(k-i-j)\Lap_1+j\Lap_2}(zq^2). \notag
\end{align}
\end{proposition}

Not all of the equations in Proposition~\ref{Prop_fun-CDb} are
linearly independent.
Assuming that $a+b\leq k-2$ and taking the difference between
\eqref{Eq_nis2} and that same equation with $b$ replaced by $b+1$
gives
\begin{align*}
&\sum_{i,j=0}^1 (-1)^{i+j} 
\mathcal{D}^{(2)}_{(i+a)\Lap_0+(k-i-j-a-b)\Lap_1+(j+b)\Lap_2}(z) \\
&\quad =-(1-zq) (zq)^{k-a-b-1} \sum_{i=0}^a \sum_{j=0}^b (zq^2)^{i+j} 
\mathcal{D}^{(2)}_{i\Lap_0+(k-i-j)\Lap_1+j\Lap_2}(zq^2).
\end{align*}
By \eqref{Eq_automorphism} this is invariant under the
interchange of $a$ and $b$, whereas \eqref{Eq_nis2} is not.
If, symbolically, we write the equation \eqref{Eq_nis2} as $e_{a,b}$,
this implies that $e_{a,b}-e_{a,b+1}=e_{b,a}-e_{b,a+1}$.
Hence, of the $\binom{k+1}{2}$ equations of the proposition, there are
only
\[
\binom{k+1}{2}-\sum_{\substack{0\leq a<b \\[1pt] a+b\leq k-2}}1
=\binom{k+1}{2}-\Floor{\frac{(k-1)^2}{4}}=
\Floor{\frac{(k+2)^2}{4}}-1
\]
linearly independent equations.
This should be further combined with the $n=1$ case of
Proposition~\ref{Prop_fun-CDa}, which by \eqref{Eq_CDn2} may be stated 
as
\begin{equation}\label{Eq_fun-D2}
\mathcal{D}^{(2)}_{a\Lap_0+(k-a)\Lap_2}(z)=
\sum_{i=0}^a \sum_{j=0}^{k-a}
(zq^2)^{i+j}\,\mathcal{D}^{(2)}_{i\Lap_0+(k-i-j)\Lap_1+j\La_2}(zq^2)
\end{equation}
for $0\leq a\leq k$. 
As mentioned previously, this gives a further $\floor{k/2}+1$ independent
equations.
However, if we take the sum of \eqref{Eq_nis2} for $b=k-a-1$ and 
\eqref{Eq_fun-D2} with $a\mapsto a+1$, we find
\begin{align*}
\mathcal{D}^{(2)}_{a\Lap_0+\Lap_1+(k-a-1)\Lap_2}(z)
&=(1+zq)\sum_{i=0}^a \sum_{j=0}^{k-a-1} (zq^2)^{i+j}
\mathcal{D}^{(2)}_{i\Lap_0+(k-i-j)\Lap_1+j\Lap_2}(zq^2) \\
&\quad +\sum_{i=0}^a (zq^2)^{k+i-a} 
\mathcal{D}^{(2)}_{i\Lap_0+(a-i)\Lap_1+(k-a)\Lap_2}(zq^2) \\
&\quad +\sum_{i=0}^{k-a-1} (zq^2)^{i+a+1}\,
\mathcal{D}^{(2)}_{(a+1)\Lap_0+(k-a-i-1)\Lap_1+i\La_2}(zq^2)
\end{align*}
for $0\leq a\leq k-1$.
Since by \eqref{Eq_automorphism} this is is invariant under
the substitution $a\mapsto k-a-1$, we have an additional 
$\floor{k/2}$ dependencies.
Equations \eqref{Eq_nis2} and \eqref{Eq_fun-D2} combined thus give
\[
\Floor{\frac{(k+2)^2}{4}}-1+\bigg(\Floor{\frac{k}{2}}+1\bigg)
-\Floor{\frac{k}{2}}=\Floor{\frac{(k+2)^2}{4}}
\]
linearly independent equations.
This is the exact same number as weights of the form
$k_0\Lap_0+k_1\Lap_1+k_2\Lap_2$ such that $k_0+k_1+k_2=k$ and 
$k_2\leq k_0$, allowing us the conclude the following result.

\begin{lemma}
Let $k$ be a positive integer.
Subject to the initial conditions $\mathcal{D}^{(2)}_{\La}(0)=1$,
the functional equations 
\eqref{Eq_nis2} and \eqref{Eq_fun-D2} combined with the
symmetry relation \eqref{Eq_automorphism}
uniquely determine the set of generating functions 
$\{\mathcal{D}^{(2)}_{\La}(z,q)\}_{\La\in P_{+}^{2k}}$,
where $P_{+}^{2k}$ is the set of level-$2k$ dominant integral 
weights of $\mathrm{D}^{(2)}_3$.
\end{lemma}

\begin{proof}[Proof of Proposition~\ref{Prop_fun-CDa}]
Since the proof is very similar to that of Proposition~\ref{Prop_fun},
we will only give a minimal amount of detail.

A necessary condition for $\la$ to be in $\C_{a,0^{n-1},k-a}$ is
for the first five columns of its frequency array to be of the form as
shown in the following (partial) frequency array on the left:
\tikzmath{\y=11;}
\begin{center}
\begin{tikzpicture}[scale=0.43,line width=0.3pt]
\draw (0,10) node {$\sc k-a$};
\draw (0,8) node {$\sc 0$};
\draw (0,5.25) node {$\vdots$};
\draw (0,2) node {$\sc 0$};
\draw (0,0) node {$\sc a$};
\draw (\x,1) node {$\sc 0$};
\draw (\x,3) node {$\sc 0$};
\draw (\x,5.25) node {$\vdots$};
\draw (\x,7) node {$\sc 0$};
\draw (\x,9) node {$\sc 0$};
\draw (2*\x,0) node {\red{$\sc i$}};
\draw (2*\x,2) node {\pink{$\sc 0$}};
\draw (2*\x,5.25) node {$\vdots$};
\draw (2*\x,8) node {\orange{$\sc 0$}};
\draw (2*\x,10) node {\green{$\sc j$}};
\draw (3*\x,1) node {\blue{$\sc f_2^{(2)}$}};
\draw (3*\x,3) node {\yg{$\sc f_2^{(4)}$}};
\draw (3*\x,5.25) node {$\vdots$};
\draw (3*\x,7) node {\purple{$\sc f_2^{(2n-2)}$}};
\draw (3*\x,9) node {\brown{$\sc f_2^{(2n)}$}};
\draw (4*\x,0) node {\red{$\sc f_3^{(1)}$}};
\draw (4*\x,2) node {\pink{$\sc f_3^{(3)}$}};
\draw (4*\x,5.25) node {$\vdots$};
\draw (4*\x,8) node {\orange{$\sc f_3^{(2n-1)}$}};
\draw (4*\x,10) node {\green{$\sc f_3^{(2n+1)}$}};
\draw (5.5*\x,4.5) node {$\mapsto$};
\draw (\y+0,10) node {$\sc 0$};
\draw (\y+0,8) node {$\sc 0$};
\draw (\y+0,5.25) node {$\vdots$};
\draw (\y+0,2) node {$\sc 0$};
\draw (\y+0,0) node {$\sc 0$};
\draw (\y+\x,1) node {$\sc a-i$};
\draw (\y+\x,3) node {$\sc 0$};
\draw (\y+\x,5.25) node {$\vdots$};
\draw (\y+\x,7) node {$\sc 0$};
\draw (\y+\x,9) node {$\sc k-a-j$};
\draw (\y+2*\x,0) node {\red{$\sc i$}};
\draw (\y+2*\x,2) node {\pink{$\sc 0$}};
\draw (\y+2*\x,5.25) node {$\vdots$};
\draw (\y+2*\x,8) node {\orange{$\sc 0$}};
\draw (\y+2*\x,10) node {\green{$\sc j$}};
\draw (\y+3*\x,1) node {\blue{$\sc f_2^{(2)}$}};
\draw (\y+3*\x,3) node {\yg{$\sc f_2^{(4)}$}};
\draw (\y+3*\x,5.25) node {$\vdots$};
\draw (\y+3*\x,7) node {\purple{$\sc f_2^{(2n-2)}$}};
\draw (\y+3*\x,9) node {\brown{$\sc f_2^{(2n)}$}};
\draw (\y+4*\x,0) node {\red{$\sc f_3^{(1)}$}};
\draw (\y+4*\x,2) node {\pink{$\sc f_3^{(3)}$}};
\draw (\y+4*\x,5.25) node {$\vdots$};
\draw (\y+4*\x,8) node {\orange{$\sc f_3^{(2n-1)}$}};
\draw (\y+4*\x,10) node {\green{$\sc f_3^{(2n+1)}$}};
\end{tikzpicture}
\end{center}
where, $i\in\{0,1,\dots,a\}$ and $j\in\{0,1,\dots,k-a\}$.
The exact same set of admissible paths arises by replacing this
by the (partial) frequency array shown on the right.
Now eliminating the first column and relabelling
$f_i^{(c)}$ (for $i+c$ even) as $f_{i-1}^{(c)}$, we end up with a
frequency array of a $\mathrm{D}^{(2)}_{n+2}$-partition $\la$ in 
$\D_{i,a-i,0,\dots,0,k-a-j,j}$.
The contribution of the full set of arrays of this form to the
generating function is
\[
(zq)^{i+j}\mathcal{D}^{(n+1)}_
{i\Lap_0+(a-i)\Lap_1+(k-a-j)\Lap_n+j\La_{n+1}}(zq).
\]
Adding all the contributions from $i\in\{0,1,\dots,a\}$ and
$j\in\{0,1,\dots,k-a\}$ results in the right-hand side of 
\eqref{Eq_fun-CD1}.
The $n=1$ case of the above proof requires some additional
justification since in the frequency array on the right of the above 
figure the two vertices in the second column labelled $a-i$
and $k-a-j$ coincide.
The obvious interpretation would be to simply sum these two labels
to give $k-i-j$. 
This, however, is $a$-independent and would therefore no longer require
that $i\leq a$ or $j\leq k-a$.
Hence it would lead to a larger range of admissible values for $i$ and $j$.
Given that we have summed the contributions of frequency arrays
of the type shown on the right over $0\leq i\leq a$ and $0\leq j\leq k-a$
irrespective of the value of $n$, the proof and hence the functional 
equation \eqref{Eq_fun-CD1} holds for all $n\geq 1$.

The functional equation \eqref{Eq_fun-CD2} is trivial since
the frequency array of an admissible $\mathrm{D}^{(2)}_{n+2}$-partition
such that $k_0=a$, $k_1=\dots=k_n=0$ and $k_{n+1}=k-a$ has only zeros in
its first and third columns.
Deleting the first column gives the frequency array of an
$\mathrm{C}^{(1)}_n$-partition
such that $k_0=a$, $k_1=\dots=k_{n-1}=0$ and $k_n=k-a$.
\end{proof}

\begin{proof}[Proof of Proposition~\ref{Prop_fun-CDb}]
We use $c:=k-a-b$ in the proof to better fit some of the vertex labels
used in the diagrams.
The first four columns of the frequency arrays of partitions
contributing to $\mathcal{D}^{(2)}_{a\Lap_0+c\Lap_1+b\Lap_2}(z)$ 
and $\mathcal{D}^{(2)}_{(a+1)\Lap_0+(c-1)\Lap_1+b\Lap_2}(z)$, respectively,
take the form
\smallskip
\tikzmath{\z=14;}
\begin{equation}
\raisebox{-0.5cm}{
\begin{tikzpicture}[
scale=0.4,line width=0.3pt]
\draw (0,1) node {$\sc c$};
\draw (\x,2) node {$\sc b$};
\draw (\x,0) node {$\sc a$};
\draw (2*\x,1) node {\blue{$\sc l$}};
\draw (3*\x,0) node {\red{$\sc i$}};
\draw (3*\x,2) node {\pink{$\sc j$}};
\draw (9.2,1) node {and};
\draw (\z,1) node {$\sc c-1$};
\draw (\x+\z,2) node {$\sc b$};
\draw (\x+\z,0) node {$\sc a+1$};
\draw (2*\x+\z,1) node {\blue{$\sc l$}};
\draw (3*\x+\z,0) node {\red{$\sc i$}};
\draw (3*\x+\z,2) node {\pink{$\sc j$}};
\end{tikzpicture}}
\end{equation}
where the triple $(\blue{f_1^{(2)}},\red{f_2^{(1)}},\pink{f_2^{(3)}})$ 
has been replaced by $(\blue{l},\red{i},\pink{j})$.
For admissibility, the array on the left requires that
\[
i,j\geq 0,\quad i+l\leq k-b,\quad j+l\leq k-a,\quad 
0\leq l\leq c,\quad i+j+l\leq k,
\]
and the array on the right that
\[
i,j\geq 0,\quad i+l\leq k-b,\quad j+l\leq k-a-1,\quad
0\leq l\leq c-1,\quad i+j+l\leq k.
\]
Hence, by taking the difference
\[
\mathcal{D}^{(2)}_{a\Lap_0+c\Lap_1+b\Lap_2}(z)-
\mathcal{D}^{(2)}_{(a+1)\Lap_0+(c-1)\Lap_1+b\Lap_2}(z),
\]
the only arrays that contribute are those on the left of the above 
figure with
\begin{subequations}
\begin{equation}\label{Eq_inequalities1}
i,j\geq 0,\quad i+l\leq k-b,\quad j+l<k-a,\quad l=c,\quad i+j+l\leq k
\end{equation}
or
\begin{equation}\label{Eq_inequalities2}
i,j\geq 0,\quad i+l\leq k-b,\quad j+l=k-a,\quad 0\leq l\leq c,\quad 
i+j+l\leq k.
\end{equation}
\end{subequations}
The first set of inequalities may be simplified to
\[
0\leq i\leq a,\quad 0\leq j<b,\quad l=c.
\]
In terms of frequency arrays this may by symbolically written as
\smallskip
\tikzmath{\z=14;}
\begin{center}
\begin{tikzpicture}[scale=0.4,line width=0.3pt]
\draw (-2.8,1) node {$\displaystyle \sum_{i=0}^a \sum_{j=0}^{b-1}$};
\draw (0,1) node {$\sc c$};
\draw (\x,2) node {$\sc b$};
\draw (\x,0) node {$\sc a$};
\draw (2*\x,1) node {\blue{$\sc c$}};
\draw (3*\x,0) node {\red{$\sc i$}};
\draw (3*\x,2) node {\pink{$\sc j$}};
\end{tikzpicture}
\end{center}
Now dropping the first two columns and replacing the $c$ in the third
column by $k-i-j$, so that the first two columns of the new array sum
to $k$, this yields
\smallskip
\tikzmath{\z=14;}
\begin{center}
\begin{tikzpicture}[scale=0.4,line width=0.3pt]
\draw (-3.7,1) node {$\displaystyle \sum_{i=0}^a \sum_{j=0}^{b-1}$};
\draw (0,1) node {\blue{$\sc k-i-j$}};
\draw (1*\x,0) node {\red{$\sc i$}};
\draw (1*\x,2) node {\pink{$\sc j$}};
\end{tikzpicture}
\end{center}
Since, by the statement of the lemma, $a+b\leq k-1$ it follows that
$k-i-j\geq 0$ as required.
The contribution to the generating function of \eqref{Eq_inequalities1} is
thus
\[
\sum_{i=0}^a \sum_{j=0}^{b-1} (zq)^{c+i+j} q^{i+j}
\mathcal{D}^{(2)}_{i\Lap_0+(k-i-j)\Lap_1+j\Lap_2}(zq^2).
\]
Similarly, the inequalities \eqref{Eq_inequalities2} may be
simplified to
\[
0\leq i\leq a,\quad b\leq j\leq k-a,\quad l=k-a-j.
\]
In terms of frequency arrays this may by symbolically written as
\smallskip
\tikzmath{\z=14;}
\begin{center}
\begin{tikzpicture}[scale=0.4,line width=0.3pt]
\draw (-2.8,1) node {$\displaystyle \sum_{i=0}^a \sum_{j=b}^{k-a}$};
\draw (0,1) node {$\sc c$};
\draw (\x,2) node {$\sc b$};
\draw (\x,0) node {$\sc a$};
\draw (2*\x,1) node {\blue{$\sc k-a-j$}};
\draw (3*\x,0) node {\red{$\sc i$}};
\draw (3*\x,2) node {\pink{$\sc j$}};
\end{tikzpicture}
\end{center}
Deleting the first two columns and replacing $k-a-j$ by $k-i-j$, so that
once again the first two columns of the new array sum to $k$, this yields
\smallskip
\tikzmath{\z=14;}
\begin{center}
\begin{tikzpicture}[scale=0.4,line width=0.3pt]
\draw (-3.7,1) node {$\displaystyle \sum_{i=0}^a \sum_{j=b}^{k-a}$};
\draw (0,1) node {\blue{$\sc k-i-j$}};
\draw (1*\x,0) node {\red{$\sc i$}};
\draw (1*\x,2) node {\pink{$\sc j$}};
\end{tikzpicture}
\end{center}
Since, by the statement of the lemma, $a+b\leq k-1$ it follows that
$k-i-j\geq 0$ as required.
The contribution to the generating function of \eqref{Eq_inequalities2} 
therefore is
\[
\sum_{i=0}^a \sum_{j=b}^{k-a} (zq)^{(k-a-j)+i+j} q^{i+j}
\mathcal{D}^{(2)}_{i\Lap_0+(k-i-j)\Lap_1+j\Lap_2}(zq^2).
\]
Adding up both contributions results in the right-hand side 
of \eqref{Eq_nis2}.
\end{proof}

The analogue of Conjecture~\ref{Con_A2n2-qseries} for $\mathrm{C}^{(1)}_n$
and $\mathrm{D}^{(2)}_{n+1}$ is given by the following pair of identities.

\begin{conjecture}\label{Con_Cn1Dn2-qseries}
For $k$ a nonnegative integer,
\begin{subequations}
\begin{align}\label{Eq_C}
\mathcal{C}^{(n)}_{k\Lap_0}(z,q)=
\sum_{\la\in\C_{k,0^n}} z^{l(\la)} q^{\abs{\la}}
&=\sum_{\substack{\la \\[1pt] \la_1\leq k}}
(zq)^{\abs{\la}} P_{2\la}\big(1,q,q^2,\dots;q^{2n}\big) 
\intertext{and}
\mathcal{D}^{(n)}_{k\Lap_0}(z,q)=
\sum_{\la\in\D_{k,0^n}} z^{l(\la)} q^{\abs{\la}}
&=\sum_{\substack{\la \\[1pt] \la_1\leq k}}
(zq^2)^{\abs{\la}} P_{2\la}\big(1,q,q^2,\dots;q^{2n-2}\big),
\label{Eq_D}
\end{align}
where in \eqref{Eq_C} it is assumed that $n\geq 0$ and
in \eqref{Eq_D} that $n\geq 1$.
\end{subequations}
\end{conjecture}

\begin{proposition}
Equation \eqref{Eq_C} holds for $n=0$, \eqref{Eq_D} holds
for $n=1$, and both equations hold for $k=1$.
Moreover, \eqref{Eq_C} holds for $z=1$.
\end{proposition}

\begin{proof}
By the functional equation \eqref{Eq_fun-CD2} for $a=k$, it is enough
to establish the claims pertaining to \eqref{Eq_C}.

We first will show that \eqref{Eq_C} holds for $n=0$.
By the $t=1$ case of \cite[Theorem 1.2]{Stembridge90},
\[
\sum_{\substack{\la \\[1pt] \la_1\leq k}} m_{2\la}(x_1,\dots,x_{\ell})
=\sum_{\varepsilon_1,\dots,\varepsilon_{\ell} \in\{\pm 1\}}
\prod_{i=1}^{\ell} \frac{x_i^{k(1-\varepsilon_i)}}{1-x_i^{2\varepsilon_i}}
=\prod_{i=1}^{\ell} \frac{1-x_i^{2k+2}}{1-x_i^2}.
\]
Hence
\[
\sum_{\substack{\la\in\C_k \\[1pt] \la_1\leq 2\ell-1}}
\prod_{i=1}^{\ell} x_i^{2f_{2i-1}}=
\sum_{\substack{\la \\[1pt] \la_1\leq k}} m_{2\la}(x_1,\dots,x_{\ell}),
\]
where $f_i:=f_i^{(1)}$.
Specialising $x_i=(zq^{2i-1})^{1/2}$ and using the homogeneity of 
the monomial symmetric functions yields
\[
\sum_{\substack{\la\in\C_k \\[1pt] \la_1\leq 2\ell-1}}
z^{l(\la)} q^{\abs{\la}}=
\sum_{\substack{\la \\[1pt] \la_1\leq k}} (zq)^{\abs{\la}}
m_{2\la}\big(1,q,\dots,q^{\ell-1}\big).
\]
Since $P_{\la}(1)=m_{\la}$, this is a bounded version of \eqref{Eq_C}
for $n=0$.

To prove \eqref{Eq_C} for $k=1$ we note that by Proposition~\ref{Prop_GOW}
with $\delta=0$ we have
\[
\sum_{r=0}^{\infty} (zq)^r P_{(2^r)}\big(1,q,q^2,\dots;q^{2n}\big)
=F^{(n+1)}_{1,0}(z,q).
\]
By \eqref{Eq_C-F} for $a=0$ this is equal to $C^{(n)}_0(z,q)=
\mathcal{C}^{(n)}_{\La_0}(z,q)$.

Finally, the $z=1$ case of \eqref{Eq_C} follows by combining
\cite[Theorem 1.2]{GOW16} with the recent
Primc--Trup\v{c}evi\'c proof of Conjecture~\ref{Con_Cn1}
for $k_0=k$ and $k_2=\dots=k_n=0$.
\end{proof}

By \eqref{Eq_q2}, the $n=2$ instance of \eqref{Eq_D} admits
an alternative expression as a multisum.
We conjecture that similar such multisums hold for
$\mathcal{D}^{(2)}_{k\Lap_1}(z,q)$ and
$\mathcal{D}^{(2)}_{\La_0+(k-1)\Lap_1}(z,q)$.

\begin{conjecture}\label{Con_Shun2}
For $k$ a nonnegative integer,
\begin{align*}
\mathcal{D}^{(2)}_{k\Lap_0}(z,q)&=
\sum_{\substack{r_1,\dots,r_k\geq 0 \\[1pt] s_1,\dots,s_k\geq 0}}\,
\prod_{i=1}^k \frac{z^{r_i+s_i} q^{(r_i+s_i)^2+s_i^2+r_i+2s_i}}
{(q;q)_{r_i-r_{i+1}}(q^2;q^2)_{s_i-s_{i-1}}}, \\
\mathcal{D}^{(2)}_{k\Lap_1}(z,q)&=
\sum_{\substack{r_1,\dots,r_k\geq 0 \\[1pt] s_1,\dots,s_k\geq 0}}\,
\prod_{i=1}^k \frac{z^{r_i+s_i} q^{(r_i+s_i)^2+s_i^2}}
{(q;q)_{r_i-r_{i+1}}(q^2;q^2)_{s_i-s_{i-1}}}, \\
\mathcal{D}^{(2)}_{\Lap_0+(k-1)\Lap_1}(z,q)&=
\sum_{\substack{r_1,\dots,r_k\geq 0 \\[1pt] s_1,\dots,s_k\geq 0}}\,
\Omega_{r_1,\dots,r_k}^{s_1,\dots,s_k}(q)
\prod_{i=1}^k \frac{z^{r_i+s_i} q^{(r_i+s_i)^2+s_i^2}}
{(q;q)_{r_i-r_{i+1}}(q^2;q^2)_{s_i-s_{i-1}}},
\end{align*}
where
\[
\Omega_{r_1,\dots,r_k}^{s_1,\dots,s_k}(q):=
\sum_{i=1}^{k-1}q^{r_i+2s_i}(1-q^{2s_{i+1}-2s_i})+q^{r_k+2s_k}
\]
and $r_{k+1}=s_0:=0$.
\end{conjecture}

Combining this with \eqref{Eq_CDn2} and Conjectures \ref{Con_Cn1}
and \ref{Con_Dn2} gives the following conjectural Andrews--Gordon-type 
identities:
\[
\sum_{\substack{r_1,\dots,r_k\geq 0 \\[1pt] s_1,\dots,s_k\geq 0}}\,
\prod_{i=1}^k \frac{q^{(r_i+s_i)^2+s_i^2+s_i}}
{(q;q)_{r_i-r_{i+1}}(q^2;q^2)_{s_i-s_{i-1}}}
=\frac{(q,q^{k+1},q^{k+2};q^{k+2})_{\infty}}
{(q;q^2)_{\infty}(q;q)_{\infty}},
\]
\begin{align*}
\sum_{\substack{r_1,\dots,r_k\geq 0 \\[1pt] s_1,\dots,s_k\geq 0}} &\,
\prod_{i=1}^k \frac{q^{(r_i+s_i)^2+s_i^2+r_i+2s_i}}
{(q;q)_{r_i-r_{i+1}}(q^2;q^2)_{s_i-s_{i-1}}} \\
&\qquad=
\frac{(q,q^2,q^{2k+2},q^{2k+3},q^{2k+4},q^{2k+4};q^{2k+4})_{\infty}}
{(q^2;q^2)_{\infty}(q;q)_{\infty}},
\end{align*}
\begin{align*}
\sum_{\substack{r_1,\dots,r_k\geq 0 \\[1pt] s_1,\dots,s_k\geq 0}} &\,
\prod_{i=1}^k \frac{q^{(r_i+s_i)^2+s_i^2}}
{(q;q)_{r_i-r_{i+1}}(q^2;q^2)_{s_i-s_{i-1}}} \\
&\qquad=
\frac{(q^{k+1},q^{k+2},q^{k+2},q^{k+3},q^{2k+4},q^{2k+4};q^{2k+4})_{\infty}}
{(q^2;q^2)_{\infty}(q;q)_{\infty}}
\end{align*}
and
\begin{align*}
\sum_{\substack{r_1,\dots,r_k\geq 0 \\[1pt] s_1,\dots,s_k\geq 0}}\, 
\Omega_{r_1,\dots,r_k}^{s_1,\dots,s_k}(q) &
\prod_{i=1}^k \frac{q^{(r_i+s_i)^2+s_i^2}}
{(q;q)_{r_i-r_{i+1}}(q^2;q^2)_{s_i-s_{i-1}}} \\
&\qquad=
\frac{(q^k,q^{k+1},q^{k+3},q^{k+4},q^{2k+4},q^{2k+4};q^{2k+4})_{\infty}}
{(q^2;q^2)_{\infty}(q;q)_{\infty}}.
\end{align*}

Conjecture~\ref{Con_Shun2} for $k=2$ can be completed to a full
set of weights, which is provable using the functional equations
for $\mathcal{D}^{(2)}_{\La}$.

\begin{theorem}\label{Thm_kis2}
We have
\begin{align*}
\mathcal{D}^{(2)}_{2\Lap_0}(z,q)&=
\sum_{\substack{r_1,r_2\geq 0 \\[1pt] s_1,s_2\geq 0}}\,
\prod_{i=1}^2 \frac{z^{r_i+s_i}q^{(r_i+s_i)^2+s_i^2+r_i+2s_i}}
{(q;q)_{r_i-r_{i+1}}(q^2;q^2)_{s_i-s_{i-1}}}, \\
\mathcal{D}^{(2)}_{2\Lap_1}(z,q)&=
\sum_{\substack{r_1,r_2\geq 0 \\[1pt] s_1,s_2\geq 0}}\,
\prod_{i=1}^2 \frac{z^{r_i+s_i} q^{(r_i+s_i)^2+s_i^2}}
{(q;q)_{r_i-r_{i+1}}(q^2;q^2)_{s_i-s_{i-1}}}, \\
\mathcal{D}^{(2)}_{\Lap_0+\Lap_1}(z,q)&=
\sum_{\substack{r_1,r_2\geq 0 \\[1pt] s_1,s_2\geq 0}}\,
\prod_{i=1}^2 \frac{z^{r_i+s_i}
q^{(r_i+s_i)^2+s_i^2+\delta_{i,2}(r_i+2s_i)}}
{(q;q)_{r_i-r_{i+1}}(q^2;q^2)_{s_i-s_{i-1}}}\,
\Big(1+zq^{2+\sum_i(r_i+2s_i)}\Big), \\
\mathcal{D}^{(2)}_{\Lap_0+\Lap_2}(z,q)&=
\sum_{\substack{r_1,r_2\geq 0 \\[1pt] s_1,s_2\geq 0}}\,
\prod_{i=1}^2 \frac{z^{r_i+s_i}
q^{(r_i+s_i)^2+s_i^2+r_i+2\delta_{i,2}s_i}}
{(q;q)_{r_i-r_{i+1}}(q^2;q^2)_{s_i-s_{i-1}}}\,
\Big(1+zq^{2+\sum_i(r_i+2s_i)}\Big),
\end{align*}
where $r_3=s_0:=0$.
\end{theorem}

The expression for $\mathcal{D}^{(2)}_{\Lap_0+\Lap_1}$ agrees with
the one given in Conjecture~\ref{Con_Shun2}.
Indeed, splitting the above multisum for
$\mathcal{D}^{(2)}_{\Lap_0+\Lap_1}$ into two multisums in the
obvious manner, then replacing $s_2\mapsto s_2-1$ in the second
multisum, and finally recombining the two summations, yields
\begin{align*}
&\mathcal{D}^{(2)}_{\Lap_0+\Lap_1}(z,q) \\
&\quad=\sum_{\substack{r_1,r_2\geq 0 \\[1pt] s_1,s_2\geq 0}}\,
\prod_{i=1}^2 \frac{z^{r_i+s_i} q^{(r_i+s_i)^2+s_i^2}}
{(q;q)_{r_i-r_{i+1}}(q^2;q^2)_{s_i-s_{i-1}}}\Big(
q^{r_2+2s_2}+q^{r_1+2s_1}(1-q^{2s_2-2s_1})\Big).
\end{align*}
In the same manner it may be shown that
\begin{align*}
&\mathcal{D}^{(2)}_{\Lap_0+\Lap_2}(z,q) \\
&\quad=\sum_{\substack{r_1,r_2\geq 0 \\[1pt] s_1,s_2\geq 0}}\,
\prod_{i=1}^2 \frac{z^{r_i+s_i} q^{(r_i+s_i)^2+s_i^2+\delta_{i,1}r_1}}
{(q;q)_{r_i-r_{i+1}}(q^2;q^2)_{s_i-s_{i-1}}}\Big(
q^{r_2+2s_2}+q^{r_1+2s_1}(1-q^{2s_2-2s_1})\Big).
\end{align*}

\begin{proof}
For brevity, we denote 
\[
A(z)=\mathcal{D}^{(2)}_{2\La_0}(z),\quad 
B(z)=\mathcal{D}^{(2)}_{2\La_1}(z),\quad
C(z)=\mathcal{D}^{(2)}_{\La_0+\La_1}(z),\quad
D(z)=\mathcal{D}^{(2)}_{\La_0+\La_2}(z),
\]
where dependence on $q$ has been suppressed.
From the set of functional equations for $k=2$, we then
choose the following four linearly independent equations:
\begin{align*}
A(z)&=B(zq^2)+zq^2 C(zq^2)+z^2q^4 A(zq^2), \\
D(z)&=B(zq^2)+2zq^2 C(zq^2)+z^2q^4 D(zq^2), \\
B(z)-C(z)&=z^2q^2 B(zq^2)+z^2q^3 C(zq^2)+z^2q^4 A(zq^2), \\
C(z)-D(z)&=zq B(zq^2)+z^2q^3 C(zq^2)+z^2q^4 A(zq^2).
\end{align*}
The first two equations are \eqref{Eq_fun-D2} with $(a,k)=(2,2)$ and
$(1,2)$ respectively, and the last two are \eqref{Eq_nis2} with 
$(a,b,k)=(0,0,2)$ and $(0,1,2)$.
Here we have also applied the symmetry \eqref{Eq_automorphism} to
eliminate occurrences of the weights $\Lap_1+\Lap_2$ and $2\Lap_2$.
Because it will subsequently lead to a significant reduction in the
number of terms, we subtract the first equation from the second and
the fourth equation from the third, so that second and fourth
equations are replaced by
\begin{align*}
D(z)-A(z)&=zq^2 C(zq^2)+z^2q^4\big(D(zq^2)-A(zq^2)\big), \\
B(z)-2C(z)+D(z)&=-zq(1-zq)B(zq^2).
\end{align*}
Our aim is to show that the purported sums in Theorem~\eqref{Thm_kis2}
are the (unique) solutions to the above system of equations.
Instead of solving these equations in $\mathbb{Q}(q)[[z]]$
subject to the initial condition $f(0)=1$ for all
$f\in\{A,B,C,D\}=:\mathcal{S}$, we will solve the $w$-deformed equations
\begin{subequations}\label{Eq_wz-functional}
\begin{align}
\label{Eq_wz-functional1}
0&=A(z,w)-B(zq^2,wq^2)-zq^2 C(zq^2,wq^2)-z^2q^4 A(zq^2,wq^2), \\
\label{Eq_wz-functional2}
0&=D(z,w)-A(z,w)-wq^2 C(zq^2,wq^2) \\
&\qquad -z^2q^4\big(D(zq^2,wq^2)-A(zq^2,wq^2)\big), \notag \\
\label{Eq_wz-functional3}
0&=B(z,w)-\Big(1+\frac{w}{z}\Big)C(z,w)+\frac{w}{z}D(z,w) \\
&\qquad +w q\Big(1-\frac{z^2q}{w}\Big) B(zq^2,wq^2), \notag \\
\label{Eq_wz-functional4}
0&=C(z,w)-D(z,w)-zq B(zq^2,wq^2)-z^2q^3 C(zq^2,wq^2) \\
&\qquad -zwq^4 A(zq^2,wq^2), \notag
\end{align}
\end{subequations}
in $\mathbb{Q}(q)[[z,w]]$ subject to the initial conditions
$f(0,0)=1$ for all $f\in\mathcal{S}$.
If $f(z):=f(z,z)$ for $f\in \mathcal{S}$, then
$\{f(z):f\in \mathcal{S}\}$ will satisfy the undeformed equations,
obtained by setting $w=z$ in \eqref{Eq_wz-functional}.
Our claim is now that this system of equations is solved by
\begin{subequations}\label{Eq_AtoD}
\begin{align}
A(z,w)&=\sum_{\substack{r_1,r_2\geq 0 \\[1pt] s_1,s_2\geq 0}}\,
\prod_{i=1}^2 \frac{z^{r_i}w^{s_i}q^{(r_i+s_i)^2+s_i^2+r_i+2s_i}}
{(q;q)_{r_i-r_{i+1}}(q^2;q^2)_{s_i-s_{i-1}}}, \\
B(z,w)&=\sum_{\substack{r_1,r_2\geq 0 \\[1pt] s_1,s_2\geq 0}}\,
\prod_{i=1}^2 \frac{z^{r_i}w^{s_i} q^{(r_i+s_i)^2+s_i^2}}
{(q;q)_{r_i-r_{i+1}}(q^2;q^2)_{s_i-s_{i-1}}}, \\
C(z,w)&=\sum_{\substack{r_1,r_2\geq 0 \\[1pt] s_1,s_2\geq 0}}\,
\prod_{i=1}^2 \frac{z^{r_i}w^{s_i}
q^{(r_i+s_i)^2+s_i^2+\delta_{i,2}(r_i+2s_i)}}
{(q;q)_{r_i-r_{i+1}}(q^2;q^2)_{s_i-s_{i-1}}}\,
\Big(1+wq^{2+\sum_i(r_i+2s_i)}\Big), \\
D(z,w)&=\sum_{\substack{r_1,r_2\geq 0 \\[1pt] s_1,s_2\geq 0}}\,
\prod_{i=1}^2 \frac{z^{r_i}w^{s_i}
q^{(r_i+s_i)^2+s_i^2+r_i+2\delta_{i,2}s_i}}
{(q;q)_{r_i-r_{i+1}}(q^2;q^2)_{s_i-s_{i-1}}}\,
\Big(1+wq^{2+\sum_i(r_i+2s_i)}\Big).
\end{align}
\end{subequations}
Clearly, each of the four functions trivialises to $1$ for $z=w=0$,
as required.

To prove our claim, we adopt the computer-assisted procedure of \cite{KR23}
(see also \cite{Chern20}).
For this it will be convenient to define
\begin{align*}
&S_{k_1,k_2,\ell_1,\ell_2}=S_{k_1,k_2,\ell_1,\ell_2}(z,w) \\[1mm]
&\,:=\sum_{\substack{m_1,m_2 \geq 0 \\[1pt] n_1,n_2\geq 0}}\!
\frac{z^{M_1+M_2}w^{N_1+N_2}
q^{(M_1+N_2)^2+(M_2+N_1)^2+N_1^2+N_2^2+k_1m_1+k_2m_2+2\ell_1n_1+2\ell_2n_2}}
{(q;q)_{m_1}(q;q)_{m_2}(q^2;q^2)_{n_1}(q^2;q^2)_{n_2}},
\end{align*}
where $k_1,k_2,\ell_1,\ell_2\in\mathbb{Z}$, $M_i=m_i+\dots+m_2$,
$N_i:=n_i+\dots+n_2$ and where dependency on $q$ is still being
suppressed.
There is some redundancy in this definition since
\[
S_{k_1,k_2,\ell_1,\ell_2}(zq^m,wq^{2n})=
S_{k_1+m,k_2+2m,\ell_1+n,\ell_2+2n}(z,w).
\]
In terms of the function $S$, the claimed expressions \eqref{Eq_AtoD} take
the form
\begin{align*}
A(z,w)&=S_{1,2,1,2}(z,w), \hspace{-20pt} &
C(z,w)&=S_{0,1,1,1}(z,w)+wq^2\,S_{1,3,2,3}(z,w),\\
B(z,w)&=S_{0,0,0,0}(z,w), \hspace{-20pt} & 
D(z,w)&=S_{1,2,1,1}(z,w)+wq^2\,S_{2,4,2,3}(z,w).
\end{align*}
We substitute this into \eqref{Eq_wz-functional} in which 
$z,w$ have been replaced by $z/q,w/q^2$ in \eqref{Eq_wz-functional1},
and $z,w$ has been replaced by $z/q,w$ in \eqref{Eq_wz-functional2}.
The resulting four equations are
\begin{subequations}
\begin{align}
\label{Eq_toshow1}
0&=S_{0,0,0,0}-S_{1,2,0,0}-zq\big(S_{1,3,1,1}+wq^2\,S_{2,5,2,3}\big)
-z^2q^2\,S_{2,4,1,2}, \\[1mm]
\label{Eq_toshow2}
0&=S_{0,0,1,1}+wq^2\,S_{1,2,2,3}-S_{0,0,1,2}
-wq^2 \big(S_{1,3,2,3}+wq^4\,S_{2,5,3,5}\big) \\
&\quad -z^2q^2 \big( S_{2,4,2,3}+wq^4\,S_{3,6,3,5}-S_{2,4,2,4}\big),
\notag \\[1mm]
\label{Eq_toshow3}
0&=S_{0,0,0,0}-\Big(1+\frac{w}{z}\Big)
\big(S_{0,1,1,1}+wq^2\,S_{1,3,2,3}\big) \\
&\quad +\frac{w}{z}\big(S_{1,2,1,1}+wq^2\,S_{2,4,2,3}\big)
+w q\Big(1-\frac{z^2q}{w}\Big) S_{2,4,1,2}, \notag \\
\label{Eq_toshow4}
0&=S_{0,1,1,1}+wq^2\,S_{1,3,2,3}-S_{1,2,1,1}-wq^2\,S_{2,4,2,3} \\
&\quad
-zq\,S_{2,4,1,2}-z^2q^3\big(S_{2,5,2,3}+wq^4\,S_{3,7,3,5}\big)
-zwq^4\,S_{3,6,2,4}. \notag
\end{align}
\end{subequations}
It is easily deduced that the function $S$ satisfies the following four
atomic relations:
\begin{align*}
R^{(1)}_{k_1,k_2,\ell_1,\ell_2}&:=
S_{k_1,k_2,\ell_1,\ell_2}-S_{k_1+1,k_2,\ell_1,\ell_2} 
-zq^{k_1+1}\,S_{k_1+2,k_2+2,\ell_1,\ell_2+1}=0, \\[1mm]
R^{(2)}_{k_1,k_2,\ell_1,\ell_2}&:=
S_{k_1,k_2,\ell_1,\ell_2}-S_{k_1,k_2+1,\ell_1,\ell_2}
-z^2q^{k_2+2}\,S_{k_1+2,k_2+4,\ell_1+1,\ell_2+2}=0, \\[1mm]
R^{(3)}_{k_1,k_2,\ell_1,\ell_2}&:=
S_{k_1,k_2,\ell_1,\ell_2}-S_{k_1,k_2,\ell_1+1,\ell_2}
-wq^{2\ell_1+2}\,S_{k_1,k_2+2,\ell_1+2,\ell_2+2}=0,\\[1mm]
R^{(4)}_{k_1,k_2,\ell_1,\ell_2}&:=
S_{k_1,k_2,\ell_1,\ell_2}-S_{k_1,k_2,\ell_1,\ell_2+1}
-w^2q^{2\ell_2+4}\,S_{k_1+2,k_2+4,\ell_1+2,\ell_2+4}=0
\end{align*}
for all $k_1,k_2,\ell_1,\ell_2\in\mathbb{Z}$.
We now show that the equations \eqref{Eq_toshow1}--\eqref{Eq_toshow4} are 
in the linear span over $\mathbb{Q}(z,w,q)$ of a finite subset of 
$\{R^{(i)}_{k_1,k_2,\ell_1,\ell_2}\}_
{k_1,k_2,\ell_1,\ell_2\in\mathbb{Z},i\in{1,2,3,4}}$
First, \eqref{Eq_toshow1} is the same as
\[
R^{(1)}_{0,1,0,0}-zq R^{(1)}_{1,3,1,1}+R^{(2)}_{0,0,0,0}
+R^{(2)}_{1,1,0,0}+zq R^{(3)}_{2,3,0,1}.
\]
Similarly, \eqref{Eq_toshow2} is
\[
R^{(2)}_{0,0,1,1}-R^{(2)}_{0,0,1,2}+wq^2 R^{(2)}_{1,2,2,3}+R^{(4)}_{0,1,1,1}.
\]
The linear combination for \eqref{Eq_toshow3} is by far the
most involved of the four cases:
\begin{align*}
&\frac{w^2}{z^3} R^{(1)}_{0,0,1,1}+R^{(1)}_{0,1,0,1} 
-\Big(1+\frac{w^2}{z^3}\Big) R^{(1)}_{0,1,1,1} 
+\frac{w^2q}{z^2} R^{(1)}_{0,1,1,2} \\
& \quad -\frac{w}{z}\Big(1+\frac{wq}{z}\Big) R^{(1)}_{0,2,1,2}
-\frac{w}{z} \big( R^{(1)}_{1,1,0,1}-R^{(1)}_{1,1,1,1}\big)
-\frac{w^2q^2}{z} R^{(1)}_{2,4,1,3} \\
& \quad +\Big(1-\frac{w}{z^2q}\Big) R^{(2)}_{0,0,0,0}
+\frac{w}{z^2q} R^{(2)}_{0,0,0,1} 
-\frac{w^2}{z^3}\big(R^{(2)}_{0,0,1,1}-R^{(2)}_{1,0,1,1}\big) \\
& \quad -\frac{w}{z}\Big(1+\frac{wq}{z}\Big) R^{(2)}_{0,1,1,2}
+\frac{w^2q}{z^2} \big(R^{(2)}_{1,1,1,2}+R^{(2)}_{2,2,0,2}\big) \\
& \quad +\Big(1+\frac{w}{z}\Big)R^{(3)}_{1,1,0,1}
-\frac{w}{z} R^{(3)}_{2,1,0,1}-\frac{w^2q}{z^2} R^{(3)}_{2,2,0,2} 
+\Big(1+\frac{w^2}{z^3}\Big)zq R^{(3)}_{2,3,0,2} \\
& \quad +\frac{w^2q^2}{z}\big(R^{(3)}_{2,3,1,3}-R^{(3)}_{3,4,1,3}\big)
-wq^2 R^{(3)}_{3,3,0,2} \\
& \quad +\frac{w}{z^2q} R^{(4)}_{0,0,0,0} 
+\Big(1-\frac{w}{z^2q}\Big) R^{(4)}_{0,1,0,0}
-\frac{w}{z}\big(R^{(4)}_{0,1,1,1}-R^{(4)}_{1,2,1,1}\big).
\end{align*}
Finally, \eqref{Eq_toshow4} is
\begin{align*}
&R^{(1)}_{0,1,1,1}+R^{(1)}_{1,2,0,1}-R^{(1)}_{1,2,1,1}
-z^2q^3 R^{(1)}_{2,5,2,3}
+R^{(2)}_{1,1,0,1}+zq R^{(2)}_{2,3,1,2} \\
&\quad -R^{(3)}_{1,1,0,1}+R^{(3)}_{2,2,0,1}
+zq^2 R^{(3)}_{3,4,0,2}+z^2q^3 R^{(3)}_{3,5,1,3}. \qedhere
\end{align*}
\end{proof}

We conclude this section with a comment on the functional
equations \eqref{Eq_wz-functional} and their solution \eqref{Eq_AtoD}.
From this solution and the remark preceding Proposition~\ref{Prop_fun}
it follows that $A(z,0,q)=D(z,0,q)=\A^{(1)}_{2\Lap_0}(z,q)$,
$B(z,0,q)=\A^{(1)}_{2\Lap_1}(z,q)$ and
$C(z,0,q)=\A^{(1)}_{\Lap_0+\Lap_1}(z,q)$.
Combinatorially this means that by setting $w=0$ those partitions in
$\D_{k_0,k_1,k_2}$ that have a nonzero entry in the top row of their
frequency array are eliminated. 
The resulting set of partitions is in one-to-one correspondence with
$\A_{k_0,k_1}$.
Accordingly, it is not hard to see that the $w=0$ case of
\eqref{Eq_wz-functional} is equivalent to the Rogers--Selberg equations
\eqref{Eq_RS} for $k=2$.
Interestingly, these same results with $(z,q)$ replaced by $(w,q^2)$
arise as a further special case since
$A(0,w,q)=\A^{(1)}_{2\Lap_0}(w,q^2)$,
$B(0,w,q)=\A^{(1)}_{2\Lap_1}(w,q^2)$ and
$C(0,w,q)=D(0,w,q)=\A^{(1)}_{\Lap_0+\Lap_1}(w,q^2)$.
The corresponding simplification of \eqref{Eq_wz-functional} is,
up to the trivial $\A^{(1)}_{2\Lap_0}(w,q^2)=\A^{(1)}_{2\Lap_1}(wq^2,q^2)$,
identical to the Rogers--Selberg equations with $(z,q)\mapsto (w,q^2)$.
We do not yet know what the appropriate set of $w$-deformed equations
is for $k\geq 3$ with the exception of the two equations
\[
\mathcal{D}^{(2)}_{k\Lap_0}(z,w)=\sum_{i=0}^k 
(zq^2)^i\,\mathcal{D}^{(2)}_{i\Lap_0+(k-i)\Lap_1}(zq^2,wq^2)
\]
and
\[
\mathcal{D}^{(2)}_{k\Lap_1}(z,w)-
\mathcal{D}^{(2)}_{\Lap_0+(k-1)\Lap_1}(z,w)
=(zq)^k \sum_{i=0}^k \Big(\frac{wq}{z}\Big)^i\,
\mathcal{D}^{(2)}_{i\Lap_0+(k-i)\Lap_1}(zq^2,wq^2).
\]
For $\mathcal{D}^{(2)}_{(k-a)\La_0+a\La_1}(z,w)\in \mathbb{Q}(q)[[z,w]]$
these equations have the property
\begin{align*}
\mathcal{D}^{(2)}_{a\Lap_0+(k-a)\Lap_1}(z,0,q)&=
\mathcal{A}^{(1)}_{a\Lap_0+(k-a)\Lap_1}(z,q), \\
\mathcal{D}^{(2)}_{a\Lap_0+(k-a)\Lap_1}(0,w,q)
&=\mathcal{A}^{(1)}_{a\Lap_0+(k-a)\Lap_1}(w,q^2).
\end{align*}
A reasonable guess is thus that
\begin{align*}
\mathcal{D}^{(2)}_{k\Lap_0}(z/q,w/q^2,q)&=
\mathcal{D}^{(2)}_{k\Lap_1}(z,w,q) \\
&=\sum_{\substack{r_1,\dots,r_k\geq 0 \\[1pt] s_1,\dots,s_k\geq 0}}\,
\prod_{i=1}^k \frac{z^{r_i} w^{s_i} q^{(r_i+s_i)^2+s_i^2}}
{(q;q)_{r_i-r_{i+1}}(q^2;q^2)_{s_i-s_{i-1}}}, \\
\mathcal{D}^{(2)}_{\La_0+(k-1)\Lap_1}(z,w,q)
&=\sum_{\substack{r_1,\dots,r_k\geq 0 \\[1pt] s_1,\dots,s_k\geq 0}}
\Omega_{r_1,\dots,r_k}^{s_1,\dots,s_k}(q)
\prod_{i=1}^k \frac{z^{r_i} w^{s_i} q^{(r_i+s_i)^2+s_i^2}}
{(q;q)_{r_i-r_{i+1}}(q^2;q^2)_{s_i-s_{i-1}}},
\end{align*}
where $r_{k+1}=s_0:=0$.
Indeed, for $w=0$ this yields the multisums on the right of
\eqref{Eq_AG} for $a=0$, $a=k$ and, since
$\Omega_{r_1,\dots,r_k}^{0,\dots,0}(q)=q^{r_k}$, $a=k-1$.
Similarly, for $z=0$ it gives \eqref{Eq_AG}
with $(z,q)\mapsto(w,q^2)$ for $a=0$, $a=k$ and, since 
$\Omega_{s_1,\dots,s_k}^{s_1,\dots,s_k}(q)=q^{2s_1}$, $a=k-1$.
The above series thus have the structure of two interwoven copies
of the Andrews--Gordon multisums.
It is an open problem to find the $w,z$-generalisations
of the Andrews--Gordon multisums \eqref{Eq_AG} for $1\leq a\leq k-2$.
Of course, the correct such sum should equate to
\[
\frac{(q^{a+1},q^{a+2},q^{2k-a+2},q^{2k-a+3},
q^{2k+4},q^{2k+4};q^{2k+4})_{\infty}}
{(q^2;q^2)_{\infty}(q;q)_{\infty}}
\]
when $z=w=1$.


\section{Towards completing Conjectures~\ref{Con_A2n2-qseries}
and \ref{Con_Cn1Dn2-qseries}}\label{Sec_completion}

Recall the definition of the bivariate generating function for coloured
partitions of type $\mathrm{A}^{(2)}_{2n}$, $\mathrm{C}^{(1)}_n$ and
$\mathrm{D}^{(2)}_{n+1}$:
\[
\mathcal{G}^{(n)}_{\La}(z,q):=
\sum_{\la\in\G_{k_0,\dots,k_n}} z^{l(\la)} q^{\abs{\la}},
\]
where $\La=k_0\Lap_0+\dots+k_n\Lap_n$ and $(\mathcal{G},\G)$ is one of
$(\mathcal{A},\A)$, $(\mathcal{C},\C)$ or $(\mathcal{D},\D)$.
For the weights $\La=k\Lap_0$ and $\La=k\Lap_n$ explicit formulas for these
generating functions in terms of Hall--Littlewood symmetric functions are 
proposed in Conjectures~\eqref{Con_A2n2-qseries} and
\eqref{Con_Cn1Dn2-qseries}.
No such expressions in terms of Hall--Littlewood functions seem to exist
for other weights.
By \cite[Lemma 2.1]{GOW16}, for $k\geq 0$, $n\geq 1$ and $t:=q^n$,
\begin{align*}
&\sum_{\substack{\la \\[1pt] \la_1\leq k}} 
(zq)^{\abs{\la}} P_{2\la}(1,q,q^2,\dots;t)\\
&\quad=\sum \prod_{i=1}^{2k}\Bigg\{
\frac{(zq)^{\frac{1}{2}\mu^{(0)}_i}}
{(t;t)_{\mu^{(0)}_i-\mu^{(0)}_{i+1}}}
\prod_{a=1}^n
q^{\mu_i^{(a)}} t^{\binombig{\mu^{(a-1)}_i-\mu^{(a)}_i}{2}}
\qbinbig{\mu^{(a-1)}_i-\mu^{(a)}_{i+1}}{\mu^{(a-1)}_i-\mu^{(a)}_i}_t
\Bigg\} \\[1mm]
&\quad=:\HL_{k,n}(z,q),
\end{align*}
where the sum on the right is over sequences of partitions
$0=\mu^{(n)}\subseteq\cdots\subseteq\mu^{(1)}\subseteq\mu^{(0)}$
such that all parts of $(\mu^{(0)})'$ are even and $l(\mu^{(0)})\leq 2k$.
Here $\mu\subseteq\la$ is shorthand for partition-inclusion, that is,
$\mu_i\leq\la_i$ for all $i\geq 1$, and, by abuse of notation, $0$ denotes
the unique partition of $0$.
The condition on the partition $\mu^{(0)}$ implies that 
$\mu^{(0)}_{2i-1}=\mu^{(0)}_{2i}$ for all $1\leq i\leq k$,
so that
\[
\prod_{i=1}^{2k} \frac{(zq)^{\frac{1}{2}\mu^{(0)}_i}}
{(t;t)_{\mu^{(0)}_i-\mu^{(0)}_{i+1}}}
=\prod_{i=1}^k \frac{(zq)^{\mu^{(0)}_{2i-1}}}
{(t;t)_{\mu^{(0)}_{2i-1}-\mu^{(0)}_{2i+1}}}.
\]
Conjectures~\ref{Con_A2n2-qseries} and \ref{Con_Cn1Dn2-qseries} can
thus be written in the following alternative form.
 
\begin{conjecture}
For $k$ a nonnegative integer and $n$ a positive integer,
\begin{align*}
\mathcal{A}^{(n)}_{k\Lap_n}(z,q)&=\HL_{k,2n-1}(z,q), \\
\mathcal{A}^{(n)}_{k\Lap_0}(z,q)&=\HL_{k,2n-1}(zq,q), \\
\mathcal{C}^{(n)}_{k\Lap_0}(z,q)&=\HL_{k,2n}(z,q), \\
\mathcal{D}^{(n)}_{k\Lap_0}(z,q)&=\HL_{k,2n-2}(zq,q),
\end{align*}
where the final equation requires $n\geq 2$. 
\end{conjecture}

By the symmetry \eqref{Eq_automorphism}, $\Lap_0$ may be replaced by
$\Lap_n$ in the last two results.

Let $S_{k,n}$ denote the set of all sequences
$0=\mu^{(n)}\subseteq\cdots\subseteq\mu^{(1)}\subseteq\mu^{(0)}$ of
partitions such that $(\mu^{(0)})'$ is even and $l(\mu^{(0)})\leq 2k$.
For $\boldsymbol{\mu}\in S_{k,n}$ and $t:=q^n$, define
\[
\HL_{k,n;\boldsymbol{\mu}}(z,q)
:=\prod_{i=1}^{2k}\Bigg\{
\frac{(zq)^{\frac{1}{2}\mu^{(0)}_i}}
{(t;t)_{\mu^{(0)}_i-\mu^{(0)}_{i+1}}}
\prod_{a=1}^n
q^{\mu_i^{(a)}}t^{\binombig{\mu^{(a-1)}_i-\mu^{(a)}_i}{2}}
\qbinbig{\mu^{(a-1)}_i-\mu^{(a)}_{i+1}}{\mu^{(a-1)}_i-\mu^{(a)}_i}_t
\Bigg\}.
\]

\begin{conjecture}
For $k,n$ positive integers,
\begin{subequations}
\begin{equation}\label{Eq_HL-variant1}
\sum_{\boldsymbol{\mu}\in S_{1,n}}
q^{n\mu^{(0)}_1-n\mu^{(1)}_1}\HL_{1,n;\boldsymbol{\mu}}(z/q,q)=
\begin{cases}
\mathcal{A}^{(n/2+1/2)}_{\Lap_1}(z,q), & \text{for odd $n$}, \\[2mm]
\mathcal{D}^{(n/2+1)}_{\Lap_1}(z,q), & \text{for even $n$}
\end{cases}
\end{equation}
and
\begin{equation}\label{Eq_HL-variant2}
\sum_{\boldsymbol{\mu}\in S_{k,2}}
q^{\sum_{i=1}^{2k} \mu^{(0)}_i-2\mu^{(1)}_1}
\HL_{1,2;\boldsymbol{\mu}}(z/q,q)=
\mathcal{D}^{(2)}_{(k-1)\Lap_0+\Lap_1}(z,q).
\end{equation}
\end{subequations}
\end{conjecture}

Since $\mu^{(0)}_{2i-1}=\mu^{(0)}_{2i}$ the two conjectures are
consistent.
Equation \eqref{Eq_HL-variant1} for $n=1$ is \eqref{Eq_Ann} for $n=1$.
For $n=2$ it may be proved as follows.

\begin{proof}
Taking $\mu^{(0)}=(r_1,r_1)$ and $\mu^{(2)}=(r_2,r_3)$,
we must show that
\[
\mathcal{D}^{(2)}_{\Lap_1}(z,q)=
\sum_{r_1,r_2,r_3\geq 0} \frac{z^{r_1} 
q^{(r_1-r_2)^2+(r_1-r_3)^2+r_2^2+r_3^2-r_2+r_3}}
{(q^2;q^2)_{r_1-r_2}(q^2;q^2)_{r_2-r_3}(q^2;q^2)_{r_3}}.
\]
Since $\mathcal{D}^{(2)}_{\Lap_1}(z,q)=D_1^{(2)}(z,q)$, if follows from
\eqref{Eq_D-F} with $a=1$ and $n=2$ that 
\[
\mathcal{D}^{(2)}_{\Lap_1}(z,q)=
\sum_{r_1,r_2\geq 0} \frac{z^{r_1} 
q^{r_1^2+r_2^2}}
{(q;q)_{r_1-r_2}(q^2;q^2)_{r_2}}.
\]
If we can prove that these two multisums are the same we are done.
Equating coefficients of $z^{r_1} q^{-r_1^2}$, we are to show that
\[
\sum_{r_2\geq 0} \frac{q^{r_2^2}} {(q;q)_{r_1-r_2}(q^2;q^2)_{r_2}}
=\sum_{r_2,r_3\geq 0}
\frac{q^{(r_1-r_2-r_3)^2+(r_2-r_3)^2-r_2+r_3}}
{(q^2;q^2)_{r_1-r_2}(q^2;q^2)_{r_2-r_3}(q^2;q^2)_{r_3}}
\]
for all nonnegative integers $r_1$.
To this end, we multiply the above by $z^{r_1}$ and sum
over $r_1$.
By then making the substitution $r_1\mapsto r_1+r_2$ on the left
and $(r_1,r_2)\mapsto (r_1+r_2+r_3,r_2+r_3)$ on the right, this
boils down to showing that
\[
\sum_{r_1\geq 0}\frac{z^{r_1}}{(q;q)_{r_1}}\cdot
\sum_{r_2\geq 0}\frac{z^{r_2}q^{r_2^2}}{(q^2;q^2)_{r_2}}
=\sum_{r_1,r_3\geq 0}\frac{z^{r_1+r_3}q^{(r_1-r_3)^2}}
{(q^2;q^2)_{r_1}(q^2;q^2)_{r_3}}\cdot
\sum_{r_2\geq 0}\frac{(z/q)^{r_2}q^{r_2^2}}{(q^2;q^2)_{r_2}}
\]
for $\abs{z}<1$.
The sums over $r_1$ and $r_2$ on the left are can be carried out by
\cite[Equation (II.2)]{GR04}
\[
\sum_{n\geq 0} \frac{z^n}{(q;q)_n}=\frac{1}{(z;q)_{\infty}},
\quad\abs{z}<1,
\]
and \cite[Equation (II.1)]{GR04}
\begin{equation}\label{Eq_qEuler}
\sum_{n\geq 0} \frac{z^n q^{\binom{n}{2}}}{(q;q)_n}
=(-z;q)_{\infty}
\end{equation}
respectively.
On the right the sum over $r_2$ can also be performed by \eqref{Eq_qEuler}.
Moreover, we can use this same summation to sum over either $r_1$ or $r_3$.
Choosing $r_3$, we are left with
\[
\sum_{r_1\geq 0}
\frac{(-q/z;q^2)_{r_1}}{(q^2;q^2)_{r_1}}\, z^{2r_1}
=\frac{(-zq;q^2)_{\infty}}{(z^2;q^2)_{\infty}},
\]
which is a special case of the $q$-binomial theorem
\cite[Equation (II.3)]{GR04}
\[
\sum_{n\geq 0} \frac{(a;q)_n}{(q;q)_n}\,z^n=
\frac{(az;q)_{\infty}}{(z;q)_{\infty}},\quad \abs{z}<1. \qedhere
\]
\end{proof}


\appendix

\section{Proof of the non-standard specialisations}

In this appendix we prove the non-standard specialisations given
in \eqref{Eq_A2n-product} and \eqref{Eq_D2n-product}, which were
first stated without proof in \cite{GOW16}.
The results presented here can be viewed as an addendum to the recent
survey \cite{BKMP24} of the Lepowsky and Wakimoto product formulas.
We should remark that there exists a second non-standard specialisation 
for the affine Lie algebra $\mathrm{D}^{(2)}_{n+1}$ that will not be 
considered below, see \cite[Theorem 5.14]{RW21}.
This specialisation differs from \eqref{Eq_NPS} in that
\[
\big(\Exp{-\alpha_0},\dots,\Exp{-\alpha_n}\big)
\mapsto (q,q^2,\dots,q^2,-1).
\]

Let $\gfrak=\gfrak(A)$ be an arbitrary affine Lie algebra with 
generalised Cartan matrix $A$ of size $(n+1)\times(n+1)$, and
let $L(\La)$ be a standard module of $\gfrak$ with normalised character 
$\chi_{\La}\in\mathbb{Z}[[\Exp{-\alpha_0},\dots,\Exp{-\alpha_n}]]$,
defined in the exact same manner as was done in Section~\ref{Sec_Lie}
for $\gfrak$ one of $\mathrm{A}^{(2)}_{2n}$, $\mathrm{C}^{(1)}_n$
or $\mathrm{D}^{(2)}_{n+1}$.
Denote the usual principal specialisation by $\phi_n$, that
is, $\phi_n:\mathbb{Z}[[\Exp{-\alpha_0},\dots,\Exp{-\alpha_n}]]
\to\mathbb{Z}[[q]]$ is given by $\phi_n(\Exp{-\alpha_i})\to q$
for all $0\leq i\leq n$.
Then \cite{Lepowsky82}
\begin{equation}\label{Eq_PS}
\phi_n(\chi_{\La})=
\prod \bigg(\frac{1-q^{\ip{\La+\rho}{\alpha}}}
{1-q^{\ip{\rho}{\alpha}}}\bigg)^{\mult(\alpha)},
\end{equation}
where the product runs over the positive roots $\alpha$ of the dual root 
system of $\gfrak$, i.e., the positive coroots of $\gfrak$.
To prove \eqref{Eq_PS}, one first applies $\phi_n$ to the 
Weyl--Kac formula \eqref{Eq_Weyl-Kac} (which holds for all $\gfrak(A)$).
To then turn the sum into a product requires the denominator or Macdonald
identity
\[
\sum_{w\in W}\sgn(w) \Exp{w(\rho)-\rho}=
\prod_{\alpha>0}(1-\Exp{-\alpha})^{\mult(\alpha)}
\]
applied to the case of the dual affine Lie algebra 
$\gfrak(\prescript{t}{}A)$.
In contrast, the specialisations of $\mathrm{A}^{(2)}_{2n}$
and $\mathrm{D}^{(2)}_{n+1}$ proven below require the Macdonald
identities for $\mathrm{B}^{(1)}_n$ and $\mathrm{D}^{(1)}_n$ 
respectively.
In other words, instead of dualising by reversing the arrows of the
Dynkin diagram of $\gfrak$, arrows between blue vertices are reversed
but in the case of a red-blue pair we have

\smallskip

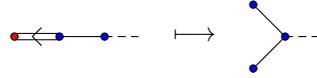
\begin{center}
\begin{tikzpicture}[scale=0.6]
\draw (0, 0.07)--(1, 0.07);
\draw (0,-0.07)--(1,-0.07);
\draw (0.6,0.2)--(0.4,0)--(0.6,-0.2);
\draw (1,0)--(2,0);
\draw[densely dashed] (2,0)--(2.75,0);
\draw[fill=red] (0,0) circle (0.08cm);
\draw[fill=blue] (1,0) circle (0.08cm);
\draw[fill=blue] (2,0) circle (0.08cm);
\draw (4,0) node {$\longmapsto$};
\draw (6-0.707,0.707)--(6,0)--(6-0.707,-0.707);
\draw[densely dashed] (6,0)--(6.75,0);
\draw[fill=blue] (6,0) circle (0.08cm);
\draw[fill=blue] (6-0.707,0.707) circle (0.08cm);
\draw[fill=blue] (6-0.707,-0.707) circle (0.08cm);
\end{tikzpicture}
\end{center}

\noindent
This maps $\mathrm{A}^{(2)}_{2n}$ to $\mathrm{B}^{(1)}_n$ (for $n\geq 3$),
$\mathrm{D}^{(2)}_{n+1}$ to $\mathrm{D}^{(1)}_n$ (for $n\geq 4$),
and $\mathrm{C}^{(1)}_n$ to $\mathrm{D}^{(2)}_{n+1}$ (for $n\geq 2$).
For small values of $n$ the appropriate degenerations of
$\mathrm{B}^{(1)}_n$ and $\mathrm{D}^{(1)}_n$ need to be used instead.

\subsection{The \texorpdfstring{$\mathrm{A}^{(2)}_{2n}$}{A(2)2n} case}

\begin{proposition}
For $k\geq 0$ an integer or half-integer and $n\geq 1$ an integer,
parametrise $\La\in P_{+}^{2k}$ as in \eqref{Eq_La-para2},
where $(\la_1,\dots,\la_n)$ is a partition such that $\la_1\leq\floor{k}$.
Then
\begin{subequations}
\begin{align}\label{Eq_A2n-spec}
\varphi_n(\chi_{\La})
&=\frac{(q^{2k+2n+1};q^{2k+2n+1})_{\infty}^n}{(q;q)_{\infty}^n}
\prod_{i=1}^n \theta\big(q^{\la_i+n-i+1};q^{2k+2n+1}\big) \\
&\quad\times\prod_{1\leq i<j\leq n}
\theta\big(q^{\la_i-\la_j-i+j},q^{\la_i+\la_j+2n-i-j+2};q^{2k+2n+1}\big)
\notag
\end{align}
if $k$ is an integer, and
\begin{equation}\label{Eq_A2n-spec-zero}
\varphi_n(\chi_{\La})=0    
\end{equation}
\end{subequations}
if $k$ is a half-integer.
\end{proposition}

This should be compared with
\begin{align*}
\phi_n(\chi_{\La})
&=\frac{(q^{2n+1};q^{4n+2})_{\infty}(q^{2k+2n+1};q^{2k+2n+1})_{\infty}^n}
{(q;q^2)_{\infty}(q;q)_{\infty}^n} \\
&\quad\times
\prod_{i=1}^n \theta\big(q^{\la_i+n-i+1};q^{2k+2n+1}\big)
\theta\big(q^{2k-2\la_i+2i-1};q^{4k+4n+2}\big) \\
&\quad\times\prod_{1\leq i<j\leq n}
\theta\big(q^{\la_i-\la_j-i+j},q^{\la_i+\la_j+2n-i-j+2};q^{2k+2n+1}\big),
\end{align*}
as follows from \eqref{Eq_PS}.
Writing $\chi^{\mathfrak{g}}_{\La}$ instead of $\chi_{\La}$, 
we also remark that in the non-vanishing or integral-$k$ case 
\[
\frac{1}{(q;q^2)_{\infty}}\,
\varphi_n\Big(\chi^{\mathrm{A}^{(2)}_{2n}}_{\La}\Big)
=\phi_n\Big(\chi^{\mathrm{A}^{(2)}_{2n-1}}_{\La'}\Big),
\]
where
\[
\La'=(2k+1-\la_1-\la_2)\La_0+
(\la_1-\la_2)\La_1+\dots+(\la_{n-1}-\la_n)\La_{n-1}+\la_n\La_n
\in P_{+}^{2k+1}
\]
is a weight of $\mathrm{A}^{(2)}_{2n-1}$ and $1/(q;q^2)_{\infty}=
\phi_n\Big(\chi^{\mathrm{A}^{(2)}_{2n-1}}_{\La_0}\Big)$.

\begin{proof}
For $x=(x_1,\dots,x_n)$, let
\[
\Delta_{\mathrm{B}}(x):=
\prod_{i=1}^n(1-x_i)\prod_{1\leq i<j\leq n}(1-x_i/x_j)(1-x_ix_j)
\]
be the Vandermonde product for the root system $\mathrm{B}_n$.
As mentioned above, at the heart of the proof of \eqref{Eq_A2n-spec}
and \eqref{Eq_A2n-spec-zero} is the $\mathrm{B}_n^{(1)}$ Macdonald
identity \cite{Macdonald72}
\begin{align*}
& \sum_{\substack{r\in\mathbb{Z}^n \\[1pt] \abs{r}\text{ even}}}
\Delta_{\mathrm{B}}(xq^r)
\prod_{i=1}^n q^{(2n-1)\binom{r_i}{2}+(i-1)r_i} x_i^{(2n-1)r_i} \\
&\qquad=(q;q)_{\infty}^n \prod_{i=1}^n \theta(x_i;q)
\prod_{1\leq i<j\leq n} \theta(x_i/x_j,x_ix_j;q)=:\Pi_{\mathrm{B}}(x,q),
\end{align*}
which in the above explicit form holds for all $n\geq 1$.
By the substitution $(r_1,x_1)\mapsto (r_1+1,x_1/q)$ and the use of the
quasi-periodicity relation $\theta(a;q)=-a\theta(aq;q)$, this yields
the exact same identity as above except that the condition on the parity
of $\abs{r}:=r_1+\dots+r_n$ has switched from even to odd and the
right-hand side has picked up a minus sign.
Taking the sum respectively difference of the even and odd cases implies
\[
\sum_{r\in\mathbb{Z}^n}
\Delta_{\mathrm{B}}(xq^r) \prod_{i=1}^n (-\sigma)^{r_i}
q^{(2n-1)\binom{r_i}{2}+(i-1)r_i} x_i^{(2n-1)r_i} 
=2\,\Pi_{\mathrm{B};\sigma}(x,q),
\]
where $\sigma\in\{-1,1\}$ and
$\Pi_{\mathrm{B};1}(x,q)=\Pi_{\mathrm{B}}(x,q)$,
$\Pi_{\mathrm{B};-1}(x,q)=0$.
By the $\mathrm{B}_n$ Vandermonde determinant
\[ 
\Delta_{\mathrm{B}}(x)=
\det_{1\leq i,j\leq n}\big(x_i^{j-n}-x_i^{n-j+1}\big)
\prod_{i=1}^n x_i^{n-i}
\]
followed by multilinearity, this may also be written as
\begin{align*}
&\det_{1\leq i,j\leq n} \bigg( \sum_{r\in\mathbb{Z}}
(-\sigma)^r q^{(2n-1)\binom{r}{2}+(n-1)r} x_i^{(2n-1)r+n-i} 
\Big((x_iq^r)^{j-n}-(x_iq^r)^{n-j+1}\Big)\bigg) \\
&\qquad\qquad\qquad=2\,\Pi_{\mathrm{B};\sigma}(x,q).
\end{align*}
Writing the sum over $r$ as $\sum_r (a_r-b_r)$, this can be
replaced by $\sum_r (a_r-b_{-r})$. 
Then once again using multilinearity, we obtain
\begin{align*}
\sum_{r\in\mathbb{Z}^n} & \bigg(\det_{1\leq i,j\leq n}
\Big(x_i^{(2n-1)r_j+j-n}-x_i^{-(2n-1)r_j+n-j+1}\Big) \\
& \quad\times \prod_{i=1}^n (-\sigma)^{r_i}
q^{(2n-1)\binom{r_i}{2}+(i-1)r_i} x_i^{n-i}\bigg)
=2\,\Pi_{\mathrm{B};\sigma}(x,q).
\end{align*}

We are now ready to prove \eqref{Eq_A2n-spec} and 
\eqref{Eq_A2n-spec-zero}.
First we recall the rewriting of the Weyl--Kac character formula 
\eqref{Eq_Weyl-Kac} in the case of $\mathrm{A}^{(2)}_{2n}$ given in
\cite[Lemma 2.2]{BW15}:
\begin{align*}
\chi_{\La}&=\frac{1}{(q;q)_{\infty}^n
\prod_{i=1}^n \theta(x_i;q)\theta(x_i^2/q;q^2)
\prod_{1\leq i<j\leq n} x_j\theta(x_i/x_j,x_ix_j/q;q)} \\[2mm]
&\quad\times 
\sum_{r\in\mathbb{Z}^n} 
\det_{1\leq i,j\leq n}
\bigg( q^{\kappa\binom{r_i}{2}} x_i^{\kappa r_i+\la_i+n} 
\Big( (x_iq^{r_i})^{j-n-1-\la_j}
-(x_iq^{r_i-1})^{n-j+1+\la_j}\Big)\bigg),
\end{align*}
where $\kappa:=2k+2n+1$, $q:=\Exp{-\delta}$ and
$x_i:=q\Exp{\alpha_0+\dots+\alpha_{i-1}}$ 
for $1\leq i\leq n$.
The marks for $\mathrm{A}^{(2)}_{2n}$ are the comarks read in reverse
order, i.e., $a_i=a^{\vee}_{n-i}$.
Hence $\delta=2\alpha_0+\dots+2\alpha_{n-1}+\alpha_n$, so that
$\varphi_n(q)=q^{2n-1}$ and $\varphi_n(x_i)=-q^{2n-i}$.
This implies
\begin{align*}
\varphi_n(\chi_{\La})=\frac{1}{2(q;q)_{\infty}^n}
\sum_{r\in\mathbb{Z}^n}\bigg(& 
\prod_{i=1}^n (-\sigma)^{r_i}
p^{(2n-1)\binom{r_i}{2}+(2n-i)r_i} y_i^{n-i} \\
&\times \det_{1\leq i,j\leq n}\Big( 
y_j^{-(2n-1)r_i+i-n}-y_j^{(2n-1)r_i+n-i+1}\Big)\bigg),
\end{align*}
where $p:=q^{\kappa}$, $y_i:=q^{\la_i+n-i+1}$ and
$\sigma=1$ if $\kappa$ is odd and $\sigma=-1$ if $\kappa$ is even,
i.e., $\sigma=1$ if $k$ is an integer and $\sigma=-1$ if $k$ is a
half-integer.
Replacing $r_i\mapsto -r_i$ and interchanging $i$ and $j$ in the 
determinant yields
\[
\varphi_n(\chi_{\La})=\frac{1}{(q;q)_{\infty}^n}\,\Pi_{\mathrm{B};\sigma}(p,y).
\]
For $\sigma=-1$ this implies the vanishing of $\varphi_n(\chi_{\La})$ and
for $\sigma=1$ this yields the claimed product form since, for
$p=q^{\kappa}$ and $x_i=q^{\la_i+n-i+1}$, 
\[
\Pi_{\mathrm{B};1}(x,p)=
\prod_{i=1}^n \theta\big(q^{\la_i+n-i+1};q^{\kappa}\big)
\prod_{1\leq i<j\leq n}
\theta\big(q^{\la_i-\la_j-i+j},q^{\la_i+\la_j+2n-i-j+2};q^{\kappa}\big).
\qedhere
\]
\end{proof}

\subsection{The \texorpdfstring{$\mathrm{D}^{(2)}_{n+1}$}{D(2)n+1} case}

A half-partition $(\la_1,\dots,\la_n)$ is a weakly decreasing
sequence such that all $\la_i-1/2\in\mathbb{N}_0$ for all $i$.

\begin{proposition}\label{Prop_Dn}
For $k\geq 0$ an integer or half-integer and $n\geq 2$ an integer,
parametrise $\La\in P_{+}^{2k}$ as in \eqref{Eq_La-para2},
where $(\la_1,\dots,\la_n)$ is a partition or half-partition
such that $\la_1\leq k$.
Then
\begin{subequations}
\begin{align}\label{Eq_Dn-spec}
&\varphi_n(\chi_{\La}) \\
&\quad =\frac{(q^{2k+2n};q^{2k+2n})_{\infty}^n}
{(q^2;q^2)_{\infty}(q;q)_{\infty}^{n-1}}  
\prod_{1\leq i<j\leq n} 
\theta\big(q^{\la_i-\la_j-i+j},q^{\la_i+\la_j+2n-i-j+1};q^{2k+2n}\big)
\notag
\end{align}
if $k$ is an integer and $\la$ is a partition, and
\begin{equation}\label{Eq_Dn-spec-zero}
\varphi_n(\chi_{\La})=0
\end{equation}
\end{subequations}
if $k$ is a half-integer or $\la$ is a half-partition.
\end{proposition}

This is to be compared with the principal specialisation
\begin{align*}
\phi_n(\chi_{\La})
&=\frac{(q^{2k+2n};q^{2k+2n})_{\infty}^n}
{(q;q^2)_{\infty}(q;q)_{\infty}^n} 
\prod_{i=1}^n \theta\big(q^{2\la_i+2n-2i+1};q^{2k+2n}\big) \\
&\quad\times\prod_{1\leq i<j\leq n}
\theta\big(q^{\la_i-\la_j-i+j},q^{\la_i+\la_j+2n-i-j+1};q^{2k+2n}\big)
\end{align*}
which holds for any level and $\lambda$ a partition or half-partition.

\begin{proof}
The key steps of the proof are identical to the
$\mathrm{A}^{(2)}_{2n}$ case with the exception of how we deal
with the half-partitions.
Starting point is the $\mathrm{D}^{(1)}_n$ Macdonald identity
\cite{Macdonald72}
\begin{align*}
&\sum_{\substack{r\in\mathbb{Z}^n \\[1pt] \abs{r}\text{ even}}}
\Delta_{\mathrm{D}}(xq^r)
\prod_{i=1}^n q^{2(n-1)\binom{r_i}{2}+(i-1)r_i} x_i^{2(n-1)r_i} \\
&\qquad =(q;q)_{\infty}^n
\prod_{1\leq i<j\leq n} \theta(x_i/x_j,x_ix_j;q)
=:\Pi_{\mathrm{D}}(x,q),
\end{align*}
where $n\geq 2$ and, for $x=(x_1,\dots,x_n)$,
\[
\Delta_{\mathrm{D}}(x):=\prod_{1\leq i<j\leq n}(1-x_i/x_j)(1-x_ix_j).
\]
As before, we carry out the substitution
$(r_1,x_1)\mapsto (r_1+1,x_1/q)$ and use quasi-periodicity.
This time the exact same identity as above arises but with
the condition on the parity of $\abs{r}$ changed to odd.
Taking the sum respectively difference of the even and odd
cases implies
\[
\sum_{r\in\mathbb{Z}^n} \Delta_{\mathrm{D}}(xq^r)
\prod_{i=1}^n \sigma^{r_i} q^{2(n-1)\binom{r_i}{2}+(i-1)r_i}
x_i^{2(n-1)r_i}=2\,\Pi_{\mathrm{D};\sigma}(x,q),
\]
where $\sigma\in\{-1,1\}$ and
$\Pi_{\mathrm{D};1}(x,q)=\Pi_{\mathrm{D}}(x,q)$,
$\Pi_{\mathrm{D};-1}(x,q)=0$.
To subsequently be able to handle the half-partition case, we enhance
the above identity to
\[
\sum_{r\in\mathbb{Z}^n} \Delta_{\mathrm{D}}^{(\tau)}(xq^r)
\prod_{i=1}^n \sigma^{r_i} q^{2(n-1)\binom{r_i}{2}+(i-1)r_i}
x_i^{2(n-1)r_i}=2\,\Pi_{\mathrm{D};\sigma,\tau}(x,q),
\]
where $\tau\in\{-1,1\}$,  
\[
\Delta_{\mathrm{D}}^{(\tau)}(x)=
\begin{cases}
\Delta_{\mathrm{D}}(x) & \text{if $\tau=1$}, \\
0 & \text{if $\tau=-1$},
\end{cases}
\quad\text{and}\quad
\Pi_{\mathrm{D};\sigma,\tau}(x,q)=
\begin{cases}
\Pi_{\mathrm{D};\sigma}(x,q) & \text{if $\tau=1$}, \\
0 & \text{if $\tau=-1$}.
\end{cases}
\]
In other words, for $\tau=-1$ we simply get $0=0$.
Applying the $\mathrm{D}_n$-Vandermonde determinant
\[
\Delta_{\mathrm{D}}^{(\tau)}(x)=
\frac{1}{2}\det_{1\leq i,j\leq n}\big(x_i^{j-n}+\tau x_i^{n-j}\big)
\prod_{i=1}^n x_i^{n-i} 
\]
and then carrying out the same sequence of steps as in the
$\mathrm{A}^{(2)}_{2n}$ case, we obtain
\begin{align}\label{Eq_D-Macdonald}
\sum_{r\in\mathbb{Z}^n}&\bigg(\det_{1\leq i,j\leq n}
\Big(x_i^{2(n-1)r_j+j-n}+\tau x_i^{-2(n-1)r_j+n-j}\Big) \\
&\quad\times\prod_{i=1}^n \sigma^{r_i}
q^{2(n-1)\binom{r_i}{2}+(i-1)r_i} x_i^{n-i}
=4\,\Pi_{\mathrm{D};\sigma,\tau}(x,q). \notag
\end{align}

Next we recall the rewriting of the Weyl--Kac character formula 
\eqref{Eq_Weyl-Kac} in the case of $\mathrm{D}^{(2)}_{n+1}$
as given in \cite[Lemma 2.4]{BW15}:
\begin{align*}
\chi_{\La}&=\frac{1}{(q^2;q^2)_{\infty}^{n-1}(q;q)_{\infty}
\prod_{i=1}^n \theta(x_i;q)
\prod_{1\leq i<j\leq n} x_j \theta(x_i/x_j,x_ix_j;q^2)} \\[2mm]
&\quad\times \sum_{r\in\mathbb{Z}^n} 
\det_{1\leq i,j\leq n} \bigg(
q^{\kappa r_i^2-(2n+2\la_i-1)r_i} x_i^{\kappa r_i} \\
& \qquad\qquad\qquad\qquad\times
\Big( (x_iq^{2r_i})^{j-1+\la_i-\la_j}-
(x_iq^{2r_i})^{2n-j+\la_i+\la_j}\Big)\bigg),
\end{align*}
where $\kappa:=2k+2n$, $q:=\Exp{-\delta}$ and
$x_i:=\Exp{-\alpha_i-\cdots-\alpha_n}$ for $1\leq i\leq n$.
Since the marks for $\mathrm{D}^{(2)}_{n+1}$ are the comarks for 
$\mathrm{C}^{(1)}_n$, $\delta=\alpha_0+\alpha_1+\dots+\alpha_n$.
Hence $\varphi_n(q)=q^{n-1}$ and $\varphi_n(x_i)=-q^{n-i}$, leading to
\begin{align*}
\varphi_n(\chi_{\La})&=
\frac{1}{4(q;q)_{\infty}^{n-1}(q^2;q^2)_{\infty}} \\
&\quad\times \sum_{r\in\mathbb{Z}^n}\bigg( \prod_{i=1}^n \sigma^{r_i} 
p^{2(n-1)\binom{-r_i}{2}-(i-1)r_i} y_i^{n-i} \\
&\qquad\qquad\qquad\qquad\times 
\det_{1\leq i,j\leq n}
\Big(y_j^{-2(n-1)r_i+i-n}+\tau y_j^{2(n-1)r_i+n-i}\Big)\bigg),
\end{align*}
where $p:=q^{\kappa}$, $y_i:=q^{\la_i+n-i+1/2}$, and
$\sigma=1$ if $\kappa$ is even (i.e., $k$ is an integer),
$\sigma=-1$ if $\kappa$ is odd (i.e., $k$ is a half-integer),
$\tau=1$ if $\la$ is a partition and $\tau=-1$ if $\la$ is a
half-partition.
Replacing $r_i\mapsto -r_i$ and interchanging $i$ and $j$ in the
determinant, it follows from \eqref{Eq_D-Macdonald} that
\[
\varphi_n(\chi_{\La})=\frac{\Pi_{\mathrm{D};\sigma,\tau}(y,p)}
{(q;q)_{\infty}^{n-1}(q^2;q^2)_{\infty}}.
\]
For $\sigma=\tau=1$ this yields the product-form \eqref{Eq_Dn-spec}
and for $\sigma=-1$ or $\tau=-1$ it implies \eqref{Eq_Dn-spec-zero}.
\end{proof}


\bibliographystyle{alpha}

\end{document}